\documentclass[article,a4page]{elsarticle}
\usepackage[utf8]{inputenc}
\usepackage{graphicx}
\usepackage{srcltx}
\usepackage{eurosym}
\usepackage{mathtools}
\allowdisplaybreaks
\usepackage{enumerate}
\usepackage{amsmath}
\usepackage{amsfonts}
\usepackage{amssymb}
\usepackage{amsthm}
\usepackage{graphicx}
\usepackage{mathrsfs}
\usepackage{xcolor}
\usepackage{exscale}
\usepackage{latexsym}
\usepackage{esvect}
\usepackage[T1]{fontenc}

\usepackage{calc}
\usepackage[titletoc,toc,page]{appendix}
\usepackage[colorlinks,plainpages=true,pdfpagelabels,hypertexnames=true,colorlinks=true,pdfstartview=FitV,linkcolor=blue,citecolor=red,urlcolor=black]{hyperref}
\PassOptionsToPackage{unicode}{hyperref}
\PassOptionsToPackage{naturalnames}{hyperref}
\AtBeginDocument{%
\hypersetup{citecolor=red}}
\usepackage{enumerate}
\usepackage[shortlabels]{enumitem}
\usepackage{bookmark}
\usepackage{wasysym}
\usepackage{esint}
\usepackage{setspace}
\usepackage[ddmmyyyy]{datetime}
\numberwithin{equation}{section}
\everymath{\displaystyle}
\usepackage[margin=2.2cm]{geometry}

\usepackage[capitalize,nameinlink]{cleveref}
\crefname{section}{Section}{Sections}
\crefname{subsection}{Subsection}{Subsections}
\crefname{condition}{Condition}{Conditions}
\crefname{hypothesis}{Hypothesis}{Hypothesis}
\crefname{assumption}{Assumption}{Assumptions}
\crefname{lemma}{Lemma}{Lemmas}
\crefname{claim}{Claim}{Claims}
\crefname{remark}{Remark}{Remarks}

\crefformat{equation}{\textup{#2(#1)#3}}
\crefrangeformat{equation}{\textup{#3(#1)#4--#5(#2)#6}}
\crefmultiformat{equation}{\textup{#2(#1)#3}}{ and \textup{#2(#1)#3}}
{, \textup{#2(#1)#3}}{, and \textup{#2(#1)#3}}
\crefrangemultiformat{equation}{\textup{#3(#1)#4--#5(#2)#6}}%
{ and \textup{#3(#1)#4--#5(#2)#6}}{, \textup{#3(#1)#4--#5(#2)#6}}%
{, and \textup{#3(#1)#4--#5(#2)#6}}

\Crefformat{equation}{#2Equation~\textup{(#1)}#3}
\Crefrangeformat{equation}{Equations~\textup{#3(#1)#4--#5(#2)#6}}
\Crefmultiformat{equation}{Equations~\textup{#2(#1)#3}}{ and \textup{#2(#1)#3}}
{, \textup{#2(#1)#3}}{, and \textup{#2(#1)#3}}
\Crefrangemultiformat{equation}{Equations~\textup{#3(#1)#4--#5(#2)#6}}%
{ and \textup{#3(#1)#4--#5(#2)#6}}{, \textup{#3(#1)#4--#5(#2)#6}}%
{, and \textup{#3(#1)#4--#5(#2)#6}}

\crefdefaultlabelformat{#2\textup{#1}#3}
\newtheorem{theorem}{Theorem}[section]
\newtheorem{lemma}[theorem]{Lemma}

\newtheorem{corollary}[theorem]{Corollary}

\newtheorem{definition}[theorem]{Definition}
\newtheorem{remark}[theorem]{Remark}        

\numberwithin{equation}{section}

\def\YYint#1#2#3{{\setbox0=\hbox{$#1{#2#3}{\iint}$}
\vcenter{\hbox{$#2#3$}}\kern-.50\wd0}}

\def\XXint#1#2#3{{\setbox0=\hbox{$#1{#2#3}{\int}$}
\vcenter{\hbox{$#2#3$}}\kern-.50\wd0}}

\makeatletter
\def\namedlabel#1#2{\begingroup
\def\@currentlabel{#2}%
\label{#1}\endgroup
}
\makeatother
\makeatletter
\newcommand{\rmh}[1]{\mathpalette{\raisem@th{#1}}}
\newcommand{\raisem@th}[3]{\hspace*{-1pt}\raisebox{#1}{$#2#3$}}
\makeatother

\newcommand{\descref}[2]{\hyperref[#1]{\textcolor{black}{(}\textcolor{blue}{\bf #2}\textcolor{black}{)}}}

\newcommand{\dref}[2]{\hyperref[#1]{\textcolor{black}{(}\textcolor{blue}{\bf #2}\textcolor{black}{)}}}

\makeatletter
\g@addto@macro\normalsize{%
\setlength\abovedisplayskip{3pt}
\setlength\belowdisplayskip{3pt}
\setlength\abovedisplayshortskip{1pt}
\setlength\belowdisplayshortskip{3pt}
}
\makeatletter
\def\ps@pprintTitle{%
\let\@oddhead\@empty
\let\@evenhead\@empty
\def\@oddfoot{}%
\let\@evenfoot\@oddfoot}
\makeatother

\newcommand{\ep}{\epsilon}

\newcounter{whitney}
\refstepcounter{whitney}

\newcounter{ineqcounter}
\refstepcounter{ineqcounter}

\begin{document}
\begin{frontmatter}
\title{Regularity and existence for a mixed local-nonlocal
parabolic equation with variable singularities and measure data
}

\author{Stuti Das}
\ead{stutid21@iitk.ac.in}

\address{Department of Mathematics and Statistics,\\ Indian Institute of Technology Kanpur, Uttar Pradesh, 208016, India}

\newcommand*{\avint}{\mathop{\, \rlap{--}\!\!\int}\nolimits}

\begin{abstract}
This article proves the existence, non-existence, regularity and asymptotic behavior of weak solutions for a class of mixed
local-nonlocal parabolic problems involving singular nonlinearities and measure data extending the works of \cite{sanjitgarain,lazermc}. A central contribution of this work is the
inclusion of a variable singular exponent. We examine both the purely singular and perturbed singular cases in the context of measure-valued data, where the source terms can simultaneously take the form of measures. To the best of our knowledge, this
phenomenon is new, even in the case of a constant singular exponent involving only a local operator. Further, all our results are also true for the operator being local only. 
\end{abstract}


\end{frontmatter}

\begin{singlespace}
\tableofcontents
\end{singlespace}

\section{Introduction}
 In this article, we explore the existence, uniqueness, regularity properties, and asymptotic behavior of weak solutions for the following parabolic mixed local-nonlocal measure data problem with variable singular exponent
 \begin{equation}{\label{mainproblemintro}}
    \begin{array}{c}
         u_t-\Delta u+(-\Delta)^s u=\frac{\nu}{u^{\gamma(x,t)}} +\mu \text { in } \Omega_T, \smallskip\\u>0 \text{ in }\Omega_T,\quad u=0  \text { in }(\mathbb{R}^n \backslash \Omega) \times(0, T), \smallskip\\ u(x, 0)=u_0(x) \text { in } \Omega ;
    \end{array}
\end{equation}
where 
\begin{itemize}
  \item  $(-\Delta)^{s}u=\text{P.V.}\int_{\mathbb{R}^{n}} {\frac{u(x,t)-u(y,t)}{{\left|x-y\right|}^{n+2s}}}dy$ and P.V. stands for the Cauchy principle value.  \item $\Omega$ is a bounded ${C}^{1}$ domain 
in $\mathbb{R}^{n}$, $\gamma$ is a positive continuous function on $\overline{\Omega}_T$, \item $n> 2$, $s\in(0,1)$, $0<T<+\infty$, 
$0\leq u_0\in L^1(\Omega)$ and $\nu,\mu$ are bounded non-negative Radon measures on $\Omega_T$ with $\nu$ not being identically $0$.
    \end{itemize}
The positivity of $\gamma$ leads to a blow-up of the nonlinearity in \cref{mainproblemintro} near the origin, a phenomenon referred to as singularity. Consequently, equation \cref{mainproblemintro} encompasses a broad spectrum of mixed singular parabolic problems, including both constant and variable exponent singular nonlinearities with measure data.

In the elliptic scenario, the purely local singular Laplace equation
\begin{equation}{\label{intt1}}
    \begin{array}{c}
-\Delta u=\frac{f}{u^{\gamma(x)}} \text { in } \Omega, \\
u=0 \text { on } \partial \Omega \text { and } u>0 \text { in } \Omega
\end{array}
\end{equation}
has been widely studied for both the constant and variable exponent $\gamma$. When $\gamma$ is a positive constant, existence of a unique classical solution is obtained in \cite{CrRaTa} under the assumption that $\partial \Omega$ is of class $C^3$ and $f \in C^1(\bar{\Omega}) \backslash\{0\}$ is non-negative. Indeed, the authors in \cite{CrRaTa} also treated more general singularities and operators. Interestingly enough, Lazer-Mckenna \cite{LaMc} showed that the unique solution obtained by \cite{CrRaTa} is indeed in $W_0^{1,2}(\Omega)$ iff $0<\gamma<3$. Boccardo-Orsina \cite{orsina} in a beautiful paper, showed the existence of the singular local problem and introduced a different notion of boundary behavior. They also showed that for $f$ being a non-negative bounded Radon measure concentrated on a set of $0$ $q$-capacity, the problem does not have any solution. Moreover, when $f$ is a non-negative bounded Radon measure on $\Omega$, existence results can be found in \cite{lazermc}. Further results on such measure data problems have been obtained in \cite{lindda}. When $\gamma$ is a variable, for some positive $f \in L^m(\Omega)$ with $m \geq 1$, existence results are established in \cite{variablegamma}. Further a perturbed singular measure data problem have been discussed in \cite{esaim}.

The nonlocal variant given by
\begin{eqnarray*}
\begin{array}{c}
  (-\Delta)^s u=\frac{f}{u^{\gamma(x)}}  \text { in } \Omega, \\
u>0  \text { in } \Omega, \quad
u=0  \text { in } \mathbb{R}^n\backslash \Omega,
\end{array}
\end{eqnarray*}
was studied in \cite{fracsingular} for $\gamma(x)=\gamma \in \mathbb{R}^{+}$. The authors proved the existence and uniqueness of positive solutions, according to the range of $\gamma$ and summability of $f$. 
For variable singular exponent; the existence results have been obtained in \cite{fracvarisin} in a quasilinear setting. When $f$ is a non-negative bounded Radon measure and $\gamma$ is a constant, existence results can be found in \cite{ghoshfrac}.

In recent years, mixed local-nonlocal problems have drawn a great attention due to its wide range of applications in biology, stochastic processes, image processing and related areas. A reference problem follows as
\begin{equation}{\label{intt3}}
\begin{array}{c}
   -\Delta u+ (-\Delta)^s u=\frac{f}{u^{\gamma(x)}}  \text { in } \Omega, \\
u>0  \text { in } \Omega, \quad
u=0  \text { in } \mathbb{R}^n\backslash \Omega.
\end{array}    
\end{equation}
 Arora-Rădulescu \cite{arora} for $\gamma(x)=\gamma\in \mathbb{R}^{+}$, obtained the existence, uniqueness, and regularity properties of weak solutions of \cref{intt3} by deriving uniform a priori estimates and using the approximation technique. 
 We also refer \cite{lazermc2}. The quasilinear case $p>1$, with constant exponent $\gamma$ has been considered in \cite{garain}, and the variable exponent case can be found in \cite{Biroud,garainkim}.

Very recently, the perturbed mixed local nonlocal singular problem 
\begin{equation}{\label{inttt2}}
\begin{array}{c}
   -\Delta u+ (-\Delta)^s u=\frac{f}{u^{\gamma(x)}}  +g\text { in } \Omega, \\
u>0  \text { in } \Omega, \quad
u=0  \text { in } \mathbb{R}^n\backslash \Omega;
\end{array}
\end{equation} is studied in \cite{ghosh22}, when $\gamma$ is a positive constant, $f$ is a positive integrable function and $g$ is a non-negative bounded Radon measure in $\Omega$. When both $f$ and $g$ are non-negative bounded Radon measures and $\gamma$ being a continuous function, existence and regularity properties have been addressed in \cite{biswas2025existence,sanjitgarain}.

In the parabolic setting the literature on problems such as \cref{intt1}, \cref{intt3} and \cref{inttt2} is, by far, much more limited. 
If $f \equiv 0$ (the case of a quasilinear parabolic equation with measure data) we refer to \cite{pett2} (see also [\citealp{softmea},\citealp{pett},\citealp{pett3}]) for a complete account on existence and uniqueness in the context of renormalized solutions. Both weak and strong regularity of the so-called SOLA solutions (solutions obtained as limit of approximations) have been also obtained (see [\citealp{1,2,22}] and references therein).

The presence of the singular term in the equation models many different physical problems. For example, it appears in the theory of heat conduction in electrically conducting materials, as described by Fulks and Maybee in 1960 in their pioneering paper \cite{fulks}. 
Concerning the model equation
\begin{equation}{\label{intt4}}
\begin{array}{c}
u_t-\Delta_p u=\frac{f}{u^{\gamma}}+g(x,u) \text { in } \Omega_T, \smallskip\\ u=0  \text { in }(\mathbb{R}^n \backslash \Omega) \times(0, T), \smallskip\\ u(x, 0)=u_0(x)  \text { in } \Omega ,
\end{array}
\end{equation}the first result in the literature appears to be the paper of Fulks and Maybee recalled above (see \cite{fulks}) where the authors study the case $g=0$ and of regular bounded data $f$ and $u_0$ when the principal part is the Laplacian operator. Most further studies in the parabolic setting have been published only in the last two decades. For suitably smooth data $f$, $g$ and $u_0$, the existence of solutions was investigated in [\citealp{parabolic1,idda}] (see also [\citealp{13,idanodea}]). If one restricts the range of $\gamma$ in $(0,2+1/p-1)$
, then various existence, uniqueness, and regularity results were obtained in $W_0^{1,p}(\Omega)\cap L^{\infty}(\Omega)$ in Bal-Badra-Giacomoni \cite{BaBaGi, BaBaGi1, BaBaGi2} for $f= 1$ and $g(x,u)$ being sub-homogeneous w.r.t the second variable. This was obtained by restricting the solution to a conical shell $\mathbf{C}$ of $L^{\infty}$ defined using the distance function depending on the exponent $\gamma$. 
For more details regarding singular parabolic problems we refer [\citealp{boga,10,BoGi}] and the references therein. Finally, in \cite{oliva} is studied the irregular case when $u_0\in L^1(\Omega)$, $f$ is a summable function and $g$ is a bounded non-negative Radon measure. 
The Nonlocal case $(s\in(0,1))$ for the parabolic problem (with $f\in L^1(\Omega_T), g\equiv 0$) was handled in \cite{abdellaoui2} for $\gamma>0$ constant to show existence and uniqueness results in the sense of Boccardo-Orsina \cite{orsina}, whereas the mixed local nonlocal case have been dealt in \cite{pap}.

As for the boundedness of weak solutions, we refer to Aronson-Serrin \cite{AS}, where the summability requirement of $f$ for boundedness was obtained for the case of second-order differential equations without singularity. Outside of the Aronson-Serrin domain, the optimal summability of solutions for the local case without singularity was obtained in Boccardo-Porzio-Primo \cite{Summability}. Analogous results for the nonlocal case were obtained in Leonori-Peral-Primo-Soria \cite{Peral}, too, for the non-singular case. For the mixed local-nonlocal operator with singularity, these summability results (for $\gamma$ being a constant) were obtained in \cite{pap}.

Finally, we remark that, the large time behavior of almost bounded solutions to \cref{intt4}, with $g=0$ and $\gamma$ being a constant have been studied recently in \cite{idanodea}, assuming further summability conditions on $f$.

To the best of our knowledge, mixed local-nonlocal parabolic problems are not yet understood in the presence of measure data with variable singular exponent. Our main purpose in this article is to fill this gap and extend analogus results for elliptic problems (see [\citealp{ghosh22,sanjitgarain,lazermc}]). We would like to emphasize that some of our results are valid, even when both $f$ and $g$ are measures. such a phenomenon is new even in the constant singular exponent case involving a local operator. The non-existence result we provide (see \cref{sec3}), is completely new in singular parabolic problems, to the best of our knowledge. Furthermore, regularity results for singular problems with variable exponent singularities were also unknown. In this paper we study also the behavior in time of the solutions of \cref{mainproblemintro}, provided $f,g\in L^1(\Omega_T)$ and $\gamma$ being a positive continuous function.
In particular we show that our problem admits global solutions. Moreover, we
prove that all the global solutions satisfy the
same asymptotic behavior, independently from the value of the initial datum.
In other words, for large values of $t$, all global solutions of \cref{mainproblemintro} exhibit the
same behavior even if they assume different initial data. Further, uniqueness results for finite energy solutions have been shown for the perturbed problem with $\mu\neq 0$, $\gamma\equiv\gamma(x,t)$.

To demonstrate our main results, we adopt the approximation approach outlined in [\citealp{orsina,oliva}]. Specifically, we establish the existence of solutions to the approximated problem using fixed-point arguments. The presence of the nonlocal operator prevents the use of classical Galerkin methods, so we resort to a different abstract strategy. Passing to the limit requires delicate a priori estimates, which we obtain by testing the approximated problem with suitably chosen functions.

The plan of this paper is as follows: In the next section, we will introduce basic function spaces, notion of weak solutions and state our main results. \cref{sec2} contains preliminaries for the non-existence and existence results while \cref{sec3} and \cref{sec4} contain their proofs respectively. \cref{sec5} holds a priori estimates for the regularity results while \cref{sec6} contains their proofs. Further, uniqueness results and asymtotic behavior have been discussed in the next two sections. Some possible open questions have been addressed in \cref{sec8}.
\section{Preliminaries}{\label{Prelis}}{\label{prelims}} 
\subsection{\textbf{Notations}} We gather all the standard notations that will be used throughout the paper.
\smallskip\\$\bullet$ We will take $n(>2)$ to be the space dimension and denote by $z=(x, t)$ to be a point in $\mathbb{R}^n \times(0, T)$, where $(0,T)\subset \mathbb{R}$ for some $0<T<\infty$.\smallskip\\
$\bullet$ Let $\Omega$ be an open bounded domain in $\mathbb{R}^n$ with ${C}^1$ boundary $\partial \Omega$ and for $0<T <\infty$, let $\Omega_T:=\Omega \times(0, T)$.\smallskip\\
$\bullet$ We denote the parabolic boundary $\Gamma_T$ by $\Gamma_T=(\Omega\times\{t=0\})\cup(\partial\Omega\times(0,T))$.\smallskip\\ 
$\bullet$ We define the set $(\Omega_T)_\delta=\{(x,t)\in\Omega_T:\operatorname{dist}((x,t),\Gamma_T)<\delta\}$ for $
\delta>0$ fixed.
\smallskip\\$\bullet$ We shall alternately use $\partial_t g$ or $\frac{\partial g}{\partial t}\text{ or } g_t $ to denote the time derivative (partial) of a function $g$.\smallskip\\
$\bullet$ For $r>1$, the H\"older conjugate exponent of $r$ will be denoted by $r^\prime=\frac{r}{r-1}$.\smallskip\\
$\bullet$ The Lebesgue measure of a measurable subset $\mathrm{S}\subset \mathbb{R}^n$ will be denoted by $|\mathrm{S}|$.\smallskip\\
$\bullet$ For any open subset $\Omega$ of $\mathbb{R}^n$, $K\subset\subset \Omega $ will imply $K$ is compactly contained in $\Omega.$\smallskip\\
$\bullet$ $\int$ will denote integration concerning either space or time only, and integration on $\Omega \times \Omega$ or $\mathbb{R}^n \times \mathbb{R}^n$ will be denoted by a double integral $\iint$.\smallskip\\
$\bullet$ We will use $\iiint$ to denote integral over $\mathbb{R}^n \times \mathbb{R}^n \times(0, T)$.
\\
$\bullet$ The notation $a \lesssim b$ will be used for $a \leq C b$, where $C$ is a universal constant which only depends on the dimension $n$ and sometimes on $s$ too. $C$ may vary from line to line or even in the same line.\smallskip\\
$\bullet$ For a measurable function $f$ over a measurable set $S$ and given constants $c, d$, we write
$c \leq u \leq  d$ in $S$ to mean that $c \leq u \leq  d$ a.e. in $S$.\smallskip\\
$\bullet$ $\langle\cdot,\cdot\rangle$ will denote the duality pairing between elements of a Banach space and its dual.\smallskip\\
$\bullet$ For any function $h$, we denote the positive and negative parts of it by $h_+=\operatorname{max}\{h,0\}$ and $h_-=\operatorname{max}\{-h,0\}$ respectively.\smallskip\\
$\bullet$ For $k\in \mathbb{N}$, we denote $T_k(\sigma)=\max \{-k, \min \{k, \sigma\}\}$ and $G_k(\sigma)=\sigma-T_k(\sigma)$, for $\sigma \in \mathbb{R}$.
\subsection{\textbf{Function Spaces}}
In this section, we present 
 definitions and properties of some function spaces that will be useful for our work. We recall that for $E \subset \mathbb{R}^n$, the Lebesgue space
$L^p(E), 1 \leq p<\infty$, is defined to be the space of $p$-integrable functions $u: E \rightarrow \mathbb{R}$ with the finite norm
\begin{equation*}
\|u\|_{L^p(E)}=\left(\int_E|u(x)|^p d x\right)^{1 / p} .
\end{equation*}
By $L_{\operatorname{loc }}^p(E)$ we denote the space of locally $p$-integrable functions, i.e. $u \in L_{\operatorname{loc }}^p(E)$ if and only if $u \in L^p(F)$ for every $F \subset\subset E$. In the case $0<p<1$, we denote by $L^p(E)$ a set of measurable functions such that $\int_E|u(x)|^p d x<\infty$.
\begin{definition}
   Let $\Omega$ be a bounded open set in $\mathbb{R}^n$ with $\mathcal{C}^1$ boundary. The Sobolev space $W^{1, p}(\Omega)$, for $1 \leq p<\infty$, is defined as the Banach space of all integrable and weakly differentiable functions $u: \Omega \rightarrow \mathbb{R}$ equipped with the following norm
\begin{equation*}
\|u\|_{W^{1, p}(\Omega)}=\|u\|_{L^p(\Omega)}+\|\nabla u\|_{L^p(\Omega)} .
\end{equation*}
\end{definition}
The space $W_0^{1, p}(\Omega)$ is defined as the closure of the space $\mathcal{C}_c^{\infty}(\Omega)$, in the norm of the Sobolev space $W^{1, p}(\Omega)$, where $\mathcal{C}^\infty_c(\Omega)$ is the set of all smooth functions whose supports are compactly contained in $\Omega$. As $\partial\Omega$ is $C^1$, by [\citealp{brezis2011functional}, Proposition 9.18], $W_0^{1, p}(\Omega)$ can be identified by $\mathbb{X}_p(\Omega)=\{u\in W^{1,p}(\mathbb{R}^n): u=0 \text{ a.e. in } \mathbb{R}^n\backslash\Omega\}$. For $1<p<\infty$, the spaces $W^{1,p}(\Omega)$ and $W^{1,p}_0(\Omega)$ are reflexive.
\begin{definition}
    Let $0<s<1$ and $\Omega$ be a open connected subset of $\mathbb{R}^n$ with $\mathcal{C}^1$ boundary. The fractional Sobolev space $W^{s, p}(\Omega)$ for any $1\leq p<+\infty$ is defined by
\begin{equation*}
    W^{s, p}(\Omega)=\left\{u \in L^p(\Omega): \frac{|u(x)-u(y)|}{|x-y|^{\frac{n}{p}+s}} \in L^p(\Omega\times\Omega)\right\},
\end{equation*}
and it is endowed with the norm
\begin{equation}{\label{norm}}
\|u\|_{W^{s, p}(\Omega)}=\left(\int_{\Omega}|u(x)|^p d x+\int_{\Omega} \int_{\Omega} \frac{|u(x)-u(y)|^p}{|x-y|^{n+p s}}d x d y\right)^{1/p}.
\end{equation}
\end{definition}
It can be treated as an intermediate space between $W^{1,p}(\Omega)$ and $L^p(\Omega)$. For $0<s\leq s^{\prime}<1$, $W^{s^{\prime},p}(\Omega)$ is continuously embedded in $W^{s,p}(\Omega)$, see [\citealp{frac}, Proposition 2.1]. The fractional Sobolev space with zero boundary values is defined by
\begin{equation*}
W_0^{s, p}(\Omega)=\left\{u \in W^{s, p}(\mathbb{R}^n): u=0 \text { in } \mathbb{R}^n \backslash \Omega\right\}.
\end{equation*}
However if $sp\neq 1$, $W_0^{s, p}(\Omega)$ can be treated as the closure of $\mathcal{C}^\infty_c(\Omega)$ in $W^{s,p}(\Omega)$ with respect to the fractional Sobolev norm defined in \cref{norm}. Both $W^{s, p}(\Omega)$ and $W_0^{s, p}(\Omega)$ are reflexive Banach spaces, for $p>1$, for details we refer to the readers [\citealp{frac}, Section 2].\smallskip\\
The following result asserts that the classical Sobolev space is continuously embedded in the fractional Sobolev space; see [\citealp{frac}, Proposition 2.2]. The idea applies an extension property of $\Omega$ so that we can extend functions from $W^{1,p}(\Omega)$ to $W^{1,p}(\mathbb{R}^n)$ and that the extension operator is bounded.
\begin{lemma}{\label{embedding}}
    Let $\Omega$ be a bounded domain in $\mathbb{R}^n$ with $\mathcal{C}^{1}$ boundary and $0<s<1$. There exists a positive constant $C=C(\Omega, n, s)$ such that
\begin{equation*}
\|u\|_{W^{s, p}(\Omega)} \leq C\|u\|_{W^{1,p}(\Omega)},
\end{equation*}
for every $u \in W^{1,p}(\Omega)$.
\end{lemma}
The next embedding result for the fractional Sobolev spaces with zero boundary value follows from [\citealp{frac2}, Lemma 2.1]. The fundamental difference of it compared to \cref{embedding} is that the result holds for any bounded domain (without any condition of smoothness of the boundary), since for the Sobolev spaces with zero boundary value, we always have a zero extension to the complement.
\begin{lemma}{\label{embedding2}} Let $\Omega$ be a bounded domain in $\mathbb{R}^n$ and $0<s<1$. Then there exists $C=C(n, s, \Omega)>0$ such that
\begin{equation*}
\int_{\mathbb{R}^n} \int_{\mathbb{R}^n} \frac{|u(x)-u(y)|^p}{|x-y|^{n+p s}} d x d y \leq C \int_{\Omega}|\nabla u|^pd x
\end{equation*}
for every $u \in W_0^{1,p}(\Omega)$. Here, we consider the zero extension of $u$ to the complement of $\Omega$.
\end{lemma}
We now proceed with the essential Poincar\'{e} inequality, which can be found in [\citealp{LCE}, Chapter 5, Section 5.8.1].
\begin{lemma}{\label{p}}
  Let $\Omega\subset \mathbb{R}^n$ be a bounded domain with $\mathcal{C}^1$ boundary and $p\geq 1$. Then there exist a positive constant $C$ depending only on $n$ and $ \Omega$, such that \begin{equation*} 
  \int_\Omega |u|^pd x\leq C\int_\Omega |\nabla u|^p d x, \qquad\forall u\in W^{1,p}_0(\Omega).
  \end{equation*}
  Specifically if we take $\Omega=B_r$, then we will get for all $u\in W^{1,p}(B_r)$,
  \begin{equation*}
  \fint_{B_{r}}\left|u-(u)_{B_r}\right|^p d x \leq c r^p \fint_{B_{r}} |\nabla u|^p d x,
  \end{equation*}
  where $c$ is a constant depending only on $n$, and $(u)_{B_r}$ denotes the average of $u$ in $B_r$, and $B_r$ denotes a ball of radius $r$ centered at $x_0\in \mathbb{R}^n$. Here, $\fint$ denotes the average integration.
 \end{lemma}
Using \cref{embedding2}, and the above Poincar\'e inequality, we observe that the following norm on the space $W^{1,p}_0(\Omega)$ defined by 
 \begin{equation*}
\|u\|_{W^{1,p}_0(\Omega)}=\left(\int_\Omega |\nabla u|^p d x +\int_{\mathbb{R}^n} \int_{\mathbb{R}^n} \frac{|u(x)-u(y)|^p}{|x-y|^{n+p s}} d x d y \right)^{\frac{1}{p}},
\end{equation*}
is equivalent to the norm
 \begin{equation*}
\|u\|_{W^{1,p}_0(\Omega)}=\left(\int_\Omega |\nabla u|^p d x  \right)^{\frac{1}{p}} .     
 \end{equation*}
For $p=2$, the spaces $W^{1,2}(\Omega)$, $W^{1,2}_0(\Omega)$, $W^{s,2}(\Omega)$ and $W^{s,2}_0(\Omega)$ are Hilbert spaces and are denoted by $W^{1,2}(\Omega)=H^1(\Omega)$, $W^{1,2}_0(\Omega)=H^1_0(\Omega)$, $W^{s, 2}(\Omega)=H^s(\Omega)$ and $W_0^{s, 2}(\Omega)=H_0^s(\Omega)$. 
Moreover $W^{-1, p^{\prime}}(\Omega)$ and $W^{-s, p^{\prime}}(\Omega)$ are defined to be the dual spaces of $W_0^{1, p}(\Omega)$ and $W_0^{s, p}(\Omega)$ respectively, where $p^{\prime}:=\frac{p}{p-1}$. Now, we define the local spaces as
\begin{equation*}
     W_{\operatorname{loc }}^{1, p}(\Omega)=\left\{u: \Omega \rightarrow \mathbb{R}: u \in L^p(K), \int_K |\nabla u|^p d x<\infty, \text { for every } K \subset \subset \Omega\right\} ,
\end{equation*}
and 
\begin{equation*}
     W_{\operatorname{loc }}^{s, p}(\Omega)=\left\{u: \Omega \rightarrow \mathbb{R}: u \in L^p(K), \int_K \int_K \frac{|u(x)-u(y)|^p}{|x-y|^{n+p s}} d x d y<\infty, \text { for every } K \subset \subset \Omega\right\} .
\end{equation*}
Now, for $n>p$, we define the critical Sobolev exponent as $p^*=\frac{np}{n-p}$. We get the following embedding result for any bounded open subset $\Omega$ of class $\mathcal{C}^1$ in $\mathbb{R}^n$, see for details [\citealp{LCE}, Chapter 5].
\begin{theorem}{\label{Sobolev embedding}} The embedding operators
\begin{eqnarray*}
    W^{1, p}(\Omega) \hookrightarrow \begin{cases}L^t(\Omega), & \text { for } t \in\left[1, p^*\right], \text { if } p \in(1, n), \\ L^t(\Omega), & \text { for } t \in[1, \infty), \text { if } p=n, \\ L^{\infty}(\Omega), & \text { if } p>n,\end{cases}
\end{eqnarray*}
are continuous. Also, the above embeddings are compact, except for $t=p^*$, if $p \in(1, n)$.
\end{theorem}
We now recall the following interpolation inequality from [\citealp{Summability}, Lemma 3.1] that will be useful for proving the boundedness of weak solutions.
\begin{theorem}{\label{gagliardo}}
    Let $v$ be a function in $W_0^{1, h}(\Omega) \cap L^\rho(\Omega)$, with $h \geq 1$ and $\rho \geq 1$. Then there exists a positive constant $C_1$, depending only on $n, h$ and $\rho$, such that
\begin{equation*}
\|v\|_{L^\eta(\Omega)} \leq C_1\|\nabla v\|_{L^h(\Omega)}^\theta\|v\|_{L^\rho(\Omega)}^{1-\theta},
\end{equation*}
for every $\eta$ and $\theta$ satisfying
\begin{equation*}
0 \leq \theta \leq 1, \quad 1 \leq \eta<+\infty, \quad \frac{1}{\eta}=\theta\left(\frac{1}{h}-\frac{1}{n}\right)+\frac{1-\theta}{\rho} .    
\end{equation*}
\end{theorem}The article will extensively use the embedding results and corresponding inequalities. We need to deal with spaces involving time for the parabolic equations, so we introduce them here. As in the classical case, we define the corresponding Bochner spaces as the following
\begin{equation*}
    \begin{array}{c}
L^p(0, T ; W_0^{1, p}(\Omega)) =\left\{u \in L^p(\Omega \times(0, T)),\|u\|_{L^p(0, T ; W_0^{1, p}(\Omega))}<\infty\right\}, \smallskip\\
L^p(0, T ; W_0^{s, p}(\Omega)) =\left\{u \in L^p(\Omega \times(0, T)),\|u\|_{L^p(0, T ; W_0^{s, p}(\Omega))}<\infty\right\},
    \end{array}
\end{equation*}
where
\begin{equation*}
\begin{array}{c}
\|u\|_{L^p(0, T ; W_0^{1, p}(\Omega))}  =\left(\int_0^T \int_{\Omega} |\nabla u|^p d x d t\right)^{\frac{1}{p}}, \quad
\|u\|_{L^p(0, T ; W_0^{s, p}(\Omega))}  =\left(\int_0^T \int_{\Omega} \int_{\Omega} \frac{|u(x, t)-u(y, t)|^p}{|x-y|^{n+p s}} d x d y d t\right)^{\frac{1}{p}}, 
\end{array}
\end{equation*}
with their dual spaces $L^{p^{\prime}}(0, T ; W^{-1, p^{\prime}}(\Omega))$ and $L^{p^{\prime}}(0, T ; W^{-s, p^{\prime}}(\Omega))$ 
respectively. Again, the local spaces are defined as
\begin{equation*}
    L^p_{\operatorname{loc}}(0, T ; W_{\operatorname{loc}}^{1,p}(\Omega))=\left\{u \in L^p(K \times[t_1,t_2]) :\int_{t_1}^{t_2} \int_K |\nabla u|^p 
    <\infty,
\forall K \subset \subset \Omega, [t_1,t_2]\subset(0,T)\right\},
\end{equation*}and\begin{equation*}
    L^p_{\operatorname{loc}}(0, T ;  W_{\operatorname{loc}}^{s,p}(\Omega))=\left\{u \in L^p(K \times[t_1,t_2]) :\int_{t_1}^{t_2} \int_K \int_K \frac{|u(x, t)-u(y, t)|^p}{|x-y|^{n+p s}} <\infty,
\forall K \subset \subset \Omega,[t_1,t_2]\subset(0,T)\right\}.
\end{equation*}
Finally we define a certain function space. For $0<\gamma<\infty$, the space $\mathcal{M}^\gamma\left(\Omega_T, \mathbb{R}^l\right)$ is the so-called Marcinkiewicz space (or the weak-${L}^\gamma$ space), defined as the set of all measurable maps $f: \Omega_T \rightarrow \mathbb{R}^l$ such that
\begin{equation*}
    \|f\|_{\mathcal{M}^\gamma\left(\Omega_T, \mathbb{R}^l\right)}:=\sup _{\lambda>0} \lambda\left|\left\{z \in \Omega_T:|f(z)|>\lambda\right\}\right|^{\frac{1}{\gamma}}<\infty.
\end{equation*}
One can observe the following connection between the Marcinkiewicz and Lebesgue spaces: $L^\gamma(\Omega_T, \mathbb{R}^l) \subset$ $\mathcal{M}^\gamma\left(\Omega_T, \mathbb{R}^l\right) \subset L^{\gamma-\varepsilon}(\Omega_T, \mathbb{R}^l)$ for any $\varepsilon \in(0, \gamma)$. 
\\The following lemma will be useful to deal with the nonlocal operator.
\begin{lemma}{\label{fracsolu}}
    Let $0<s<1<p<\infty$ and $u \in  L^p_{\operatorname{loc}}(0, T ;  W_{\operatorname{loc}}^{s,p}(\Omega))\cap L^1(\Omega_T)$ and $u=0$ a.e. in $(\mathbb{R}^n \backslash \Omega)\times(0,T)$. Then for any $\phi \in C_c^1(\Omega_T)$, we have
\begin{equation*}
\int_0^T\int_{\mathbb{R}^n} \int_{\mathbb{R}^n}\frac{(u(x,t)-u(y,t))(\phi(x,t)-\phi(y,t))}{|x-y|^{n+2s}} d x d ydt<\infty.
\end{equation*}
\end{lemma}
\begin{proof}
    The proof follows from [\citealp{nonlinearity}, Proposition 2.3]. We give some details for sake of completeness. Let
$
\omega,\omega_1$, $[t_1,t_2]$ be such that $\operatorname{supp} \phi\subset\subset\omega_1\times[t_1,t_2]\subset\subset\omega\times[t_1,t_2]\subset\subset\Omega_T$ and set $\mathcal{Q}_{\phi}:=[t_1,t_2]\times(\mathbb{R}^{2 n} \backslash(\omega^c \times \omega^c))$. Thus
\begin{equation*}
    \int_0^T\iint_{\mathbb{R}^{2 n}} \frac{(u(x,t)-u(y,t))(\phi(x,t)-\phi(y,t))}{|x-y|^{n+2s}} d x d y dt = \iiint_{\mathcal{Q}_{\phi}}\frac{(u(x,t)-u(y,t))(\phi(x,t)-\phi(y,t))}{|x-y|^{n+2s}} d x d y dt.
\end{equation*}
Now $\mathcal{Q}_{\phi}=[t_1,t_2]\times\left((\omega\times \omega)\cup(\omega\times\omega^c)\cup(\omega^c\times\omega)\right)$. Since $u \in  L^p_{\operatorname{loc}}(0, T ;  W_{\operatorname{loc}}^{s,p}(\Omega))$, it is easy to check using H\"older inequality that $\int_{t_1}^{t_2}\iint_{\omega\times\omega} \frac{(u(x,t)-u(y,t))(\phi(x,t)-\phi(y,t))}{|x-y|^{n+2s}} d x d y dt<\infty$.
Now \begin{equation*}
    \int_{t_1}^{t_2}\int_\omega\int_{\omega^c}\frac{(u(x,t)-u(y,t))(\phi(x,t)-\phi(y,t))}{|x-y|^{n+2s}} d x d y dt=\int_{t_1}^{t_2}\int_{\omega_1}
    \int_{\omega^c}\frac{(u(x,t)-u(y,t))\phi(x,t)}{|x-y|^{n+2s}} d x d y dt.
\end{equation*}
Observe, for all $x \in  \omega_1,
y \in \omega^c
$, it holds $|x-y| \geqslant \delta>0$, for some positive constant $\delta$. Hence
\begin{equation*}
    \int_{t_1}^{t_2}\int_{\omega}\int_{\omega^c}\frac{(u(x,t)-u(y,t))(\phi(x,t)-\phi(y,t))}{|x-y|^{n+2s}} d x d y dt\leq C
\end{equation*}
with $C=C\left(\delta, \operatorname{supp}\phi, \omega,t_1,t_2,\|u\|_{L^1(\mathbb{R}^n\times(0,T))},\|\phi\|_{L^{\infty}\left(\omega\times[t_1,t_2]\right)}\right)$ a positive constant. Here we have used the fact that $u \in L^1(\mathbb{R}^n\times(0,T))$ (since $u \in L^1(\Omega_T)$ and $u=0$ a.e. in $(\mathbb{R}^n\backslash \Omega)\times(0,T)$ ) and $\phi \in C_c^{1}(\Omega_T)$. By symmetry, the proof is now complete.
\end{proof}
For $V$ being a Banach space and $V^*$ being its dual, we recall that an operator $\mathcal{A}: V \rightarrow V^*$ is said to be
\\- monotone if for every $u, v \in V$,
\begin{equation*}
\langle\mathcal{A}(u)-\mathcal{A}(v), u-v\rangle \geq 0 ,   
\end{equation*}
- hemicontinuous if the real function $\lambda \mapsto\langle\mathcal{A}(u+\lambda v), v\rangle$ is continuous, for every $u, v \in V$.
\\The following theorem will be useful to show the existence of solutions for approximated problems. See [\citealp{evolution}, Theorem A.2] for the proof. 
\begin{theorem}{\label{evolu}}
    Let $V$ be a separable, reflexive Banach space and let $\mathcal{V}=L^p(I ; V)$, for $1<p<\infty$, where $I=\left[t_0, t_1\right]$. Suppose that $H$ is a Hilbert space (with $(\cdot,\cdot)$ being the associated inner product) such that $V$ is dense and continuously embedded in $H$ and that $H$ is embedded into $V^*$ according to the relation
    \begin{equation*}
        \langle h,v\rangle=(h,v)_H, \text{ for each } h\in H, v\in V.
    \end{equation*}Assume that the family of operators $\mathcal{A}(t, \cdot): V \rightarrow V^*, t \in I$ satisfies:\begin{enumerate}
        \item for every $v \in V$, the function $\mathcal{A}(\cdot, v): I \rightarrow V^*$ is measurable;\item  for almost every $t \in I$, the operator $\mathcal{A}(t, \cdot): V \rightarrow V^*$ is monotone, hemicontinuous and bounded by
\begin{equation*}
    \|\mathcal{A}(t, v)\|_{V^*} \leq C\left(\|v\|_V^{p-1}+k(t)\right), \quad \text { for } v \in V \quad \text { and } \quad k \in L^{p^{\prime}}(I),
\end{equation*}\item there exist a real number $\beta>0$ and a function $\ell \in L^1(I)$ such that
\begin{equation*}
\langle\mathcal{A}(t, v), v\rangle+\ell(t) \geq \beta\|v\|_V^p, \quad \text { for a. e. } t \in I \text { and } v \in V .    
\end{equation*}
 \end{enumerate}
Then for each $f \in \mathcal{V}^*=L^{p^{\prime}}(I ; V^*)$ and $u_0 \in H$, there exists a unique $u \in W_p(I)=\{v\in L^p(I;V):v^\prime \in L^{p^{\prime}}(I ; V^*)\}\subset C(I;H)$ satisfying
\begin{equation*}
u^{\prime}(t)+\mathcal{A}(t, u(t))=f(t), \quad \text { in } \mathcal{V}^*, \quad u\left(t_0\right)=u_0 \text { in } H.    
\end{equation*}
This means that $u \in \mathcal{V}, u^{\prime} \in \mathcal{V}^*$ and
\begin{equation*}
    \int_I\left\langle u^{\prime}(t), \phi(t)\right\rangle \mathrm{d} t+\int_I\langle\mathcal{A}(t, u(t)), \phi(t)\rangle \mathrm{d} t=\int_I\langle f(t), \phi(t)\rangle \mathrm{d} t, \quad \text { for all } \phi \in \mathcal{V}.
\end{equation*}
\end{theorem}
The following theorem will be useful to pass limits in the approximated problems.
\begin{theorem}{\label{limits}}
    Suppose $\{f_k\}_{k\in \mathbb{N}}$ be a sequence in $L^1(\Omega_T)$ such that $f_k \rightharpoonup f$ weakly in $L^1(\Omega_T)$ and $\{g_k\}_{k\in \mathbb{N}}$ be a sequence in $L^{\infty}(\Omega_T)$ such that $g_k$ converges to $g$ in a.e. in $\Omega_T$ and weak* in $L^{\infty}(\Omega_T)$. Then
\begin{equation*}
    \lim _{k \rightarrow \infty} \iint_{\Omega_T} f_kg_k d x d t=\iint_{\Omega_T} f g d x d t.
\end{equation*}
\end{theorem} 
\begin{proof}
We first pick any $\epsilon>0$. By Egorov's theorem, there exists a subset $S_\epsilon \subset \Omega_T$ with small complement $|\Omega_T \backslash S_\epsilon| \leq \epsilon$ such that $g_k\rightarrow g$ uniformly in $S_\epsilon$. Thus by standard $L^{\infty}-L^1$ strong-weak convergence the term
\begin{equation*}
\iint_{S_\epsilon} f_k g_k d x dt\longrightarrow \iint_{S_\epsilon} f g  d xdt.
\end{equation*}
For the remaining term (the integral on $\Omega_T \backslash S_\epsilon$ ), the Dunford-Pettis theorem guarantees that $\{f_k\}_k$ is uniformly integrable. Given $\delta>0$, this means that $\iint_{\Omega_T \backslash S_\epsilon}|f_k| \leq \delta$ uniformly in $k$, as soon as $|\Omega_T \backslash S_\epsilon| \leq \epsilon$ is sufficiently small. Since $\{g_k\}$ is weak* convergent, so $\|g_k\|_{\infty} \leq M$ for some constant $M$, this immediately gives
\begin{equation*}
    \left|\iint_{\Omega_T \backslash S_\epsilon} f_k g_k  d xdt\right| \leq M\delta.
\end{equation*}
Putting everything together and playing a bit with $\epsilon, \delta, k\geq k_0$, we get the desired result (also note that $\iint_{S_\epsilon} f g  d xdt \rightarrow \iint_{\Omega_T} f g  d xdt$ if $\epsilon \rightarrow 0$).
\end{proof}
See [\citealp{fracvari1}, Lemma 3.5] for the following result.
\begin{lemma}{\label{algebraic2}}
    Let $q>1$ and $\epsilon>0$. For $(x, y) \in \mathbb{R}^2$, let us set
\begin{equation*}
S_\epsilon^x:=\{x \geq \epsilon\} \cap\{y \geq 0\} \text { and } S_\epsilon^y=\{y \geq \ep\} \cap\{x \geq 0\} .
\end{equation*}
Then for every $(x, y) \in S_\epsilon^x \cup S_\epsilon^y$, we have
\begin{equation*}
    \ep^{q-1}|x-y| \leq\left|x^q-y^q\right| .
\end{equation*}
\end{lemma}
Next, we mention some preliminary results related to measures. Let $\mathcal{M}(\Omega_T)$ be the space of all signed Radon measures on $\Omega_T$ with finite total variation. If $\nu \in \mathcal{M}(\Omega_T)$ is a non-negative Radon measure, then by Lebesgue's decomposition theorem [\citealp{royden}, Page 384],
\begin{equation*}
    \nu=\nu_a+\nu_s,
\end{equation*}
where $\nu_a$ is absolutely continuous with respect to the Lebesgue measure $\mathcal{L}$, i.e., $\nu_a \ll \mathcal{L}$ and $\nu_s$ is singular with respect to the Lebesgue measure $\mathcal{L}$, i.e., $\nu_s \perp \mathcal{L}$. By Radon-Nikodym theorem [\citealp{royden}, Page 382], there exists a non-negative Lebesgue measurable function $f$ such that for every measurable set $E \subset \Omega_T$,
\begin{equation*}
\nu_a(E)=\iint_E f d x.    
\end{equation*}
Furthermore, if $\nu$ is bounded then $f \in L^1(\Omega_T)$. If the function $f$ is not an identically zero function, then we say that $\nu$ is non-singular with respect to the Lebesgue measure $\mathcal{L}$; otherwise, it is called a singular measure.
We now recall the notion of parabolic $p$-capacity associated to our problem (for further details, see for instance, \cite{softmea}).
\begin{definition}
    Let $p>1$ and let us call $V=W_0^{1, p}(\Omega) \cap L^2(\Omega)$ endowed with its natural norm $\|\cdot\|_{W_0^{1, p}(\Omega)}+\|\cdot\|_{L^2(\Omega)}$ and
\begin{equation*}
    W=\left\{u \in L^p(0, T ; V), u_t \in L^{p^{\prime}}(0, T ; V^{\prime})\right\},
\end{equation*}
endowed with its natural norm $\|u\|_W=\|u\|_{L^p(0, T ; V)}+\left\|u_t\right\|_{L^{p^{\prime}}(0, T ; V^{\prime})}$. So, if $U \subseteq \Omega_T$ is an open set, we define the parabolic $p$-capacity of $U$ as
\begin{equation*}
    \operatorname{cap}_p(U)=\inf \left\{\|u\|_W: u \in W, u \geq \chi_U \text { a.e. in } \Omega_T\right\}
\end{equation*}
where as usual we set $\inf \emptyset=+\infty$. For any Borel set $B \subseteq \Omega_T$ we then define
\begin{equation*}
    \operatorname{cap}_p(B)=\inf \left\{\operatorname{cap}_p(U), U \text { open set of } \Omega_T, B \subseteq U\right\}.
\end{equation*}
\end{definition} 
Let us denote with $\mathcal{M}^p_0(\Omega_T)$ the set of all Radon measures with bounded total variation over $\Omega_T$ that do not charge the sets of zero $p$-capacity, that is if $\mu \in \mathcal{M}^p_0(\Omega_T)$, then $\mu(E)=0$, for all $E \subseteq \Omega_T$ such that $\operatorname{cap}_p(E)=0$. Such measures are called diffuse measures. The following representation theorem for diffuse measures was obtained in [\citealp{softmea}, Theorem 1.1].
\begin{theorem}{\label{representation}}
     If $\mu \in \mathcal{M}^p_0(\Omega_T)$ then there exist $h \in L^{p^{\prime}}(0, T ; W^{-1, p^{\prime}}(\Omega)), g \in L^p(0, T ; W_0^{1, p}(\Omega) \cap L^2(\Omega))$ and $f \in L^1(\Omega_T)$, such that
     \begin{equation*}
         \iint_{\Omega_T} \varphi \,d \mu=\int_0^T\langle h, \varphi\rangle d t-\int_0^T\left\langle\varphi_t, g\right\rangle d t+\iint_{\Omega_T} f \varphi \,d x d t
     \end{equation*}
for any $\varphi \in C_c^{\infty}([0, T] \times \Omega)$, where $\langle\cdot, \cdot\rangle$ denotes the duality between $V^{\prime}$ and $V$.
\end{theorem}
We recall that, given a Radon measure $\mu$ on $\Omega_T$ and a Borel set $E \subset \Omega_T$ , then $\mu$ is said
to be concentrated on $E$ if $\mu(B) = \mu(B \cap E)$ for every Borel set $B\subset\Omega_T$.
\begin{definition} A sequence $\{\mu_k\}_{k \in \mathbb{N}} \subset C_c^\infty(\Omega_T)$ is said to converge to a measure $\mu \in \mathcal{M}(\Omega_T)$ in narrow topology if for every $\phi \in C_c^{\infty}(\Omega_T)$,
\begin{equation*}
    \iint_{\Omega_T} \phi \mu_kdx dt \longrightarrow \iint_{\Omega_T} \phi\, d \mu \text { as } k \rightarrow \infty.
\end{equation*}
\end{definition}
\subsection{\textbf{Weak solutions and main results}}
We are now in a position to state the main results obtained in this article. We consider the following mixed local nonlocal parabolic singular problem given by \begin{equation}{\label{mainproblem}}
    \begin{array}{c}
         u_t-\Delta u+(-\Delta)^s u=\frac{\nu}{u^{\gamma(x,t)}} +\mu \text { in } \Omega_T, \smallskip\\u>0 \text{ in }\Omega_T,\quad u=0  \text { in }(\mathbb{R}^n \backslash \Omega) \times(0, T), \smallskip\\ u(x, 0)=u_0(x) \text { in } \Omega ;
    \end{array}
\end{equation}
where 
\begin{itemize}
  \item  $(-\Delta)^{s}u=\text{P.V.}\int_{\mathbb{R}^{n}} {\frac{u(x,t)-u(y,t)}{{\left|x-y\right|}^{n+2s}}}dy$ and P.V. stands for the Cauchy principle value.  \item $\Omega$ is a bounded ${C}^{1}$ domain 
in $\mathbb{R}^{n}$, $\gamma$ is a positive continuous function on $\overline{\Omega}_T$, \item $n> 2$, $s\in(0,1)$, $0<T<+\infty$, 
$0\leq u_0\in L^1(\Omega)$ and $\nu,\mu$ are bounded non-negative Radon measures on $\Omega_T$.
    \end{itemize}
    In order to deal with the singularity and measure data, we consider an approximation of \cref{mainproblem}, for each $k\in\mathbb{N}$ as
    \begin{equation}{\label{aproxproblem}}
    \begin{array}{c}
         (u_k)_t-\Delta u_k+(-\Delta)^s u_k=\frac{\nu_k}{\left(u_k+\frac{1}{k}\right)^{\gamma(x,t)}} +\mu_k \text { in } \Omega_T, \smallskip\\u_k>0 \text{ in }\Omega_T,\quad u_k=0  \text { in }(\mathbb{R}^n \backslash \Omega) \times(0, T), \smallskip\\ u_k(x, 0)=u_{0k}(x) \text { in } \Omega ;
    \end{array}
\end{equation}
where 
\begin{itemize}
  \item  $\nu_k, \mu_k\in L^\infty(\Omega_T)$ are non-negative and bounded in $L^1(\Omega_T)$ and converges in narrow topology to $\nu,\mu$ respectively. The existence of such sequences can be obtained by standard convolution arguments. 
  \item $u_{0k}(x)=T_k(u_0(x))\in L^\infty(\Omega)$ so that $u_{0k}\rightarrow u_0$ in $L^1(\Omega)$.
    \end{itemize}
   \subsubsection{Non-existence results} For a non-negative measure concentrated on a set with $0$ parabolic $2$ capacity, we show that the solutions to the approximated problems involving only a local operator go to $0$. 
   \begin{theorem}{\label{nonexistence}}
        Let $u_0\equiv 0$, $\gamma(x,t)\equiv \gamma\geq 1$, be a constant, $\mu\equiv 0$ so that $\mu_k\equiv 0$  and $\nu$ be a non-negative bounded Radon measure concentrated on a Borel set of $0$ parabolic $2$-capacity.  Let $\nu_k$ be a sequence of non-negative $L^\infty(\Omega_T)$ functions that converges to $\nu$ in the narrow topology 
        and $u_k$ be the solution of the approximated problem
          \begin{equation}{\label{aproxproblemnonn}}
    \begin{array}{c}
         (u_k)_t-\Delta u_k=\frac{\nu_k}{\left(u_k+\frac{1}{k}\right)^{\gamma(x,t)}}  \text { in } \Omega_T, \smallskip\\u_k>0 \text{ in }\Omega_T,\quad u_k=0  \text { in }(\mathbb{R}^n \backslash \Omega) \times(0, T), \smallskip\\ u_k(x, 0)=0 \text { in } \Omega .
    \end{array}
\end{equation}
        Then $u_k^{\frac{\gamma+1}{2}}$ converges weakly to zero in $L^2(0,T;W^{1,2}_0(\Omega))$.
   \end{theorem}
   \subsubsection{Existence results}
   Once the non-existence result is obtained for measures which are concentrated on sets of $0$ parabolic $2$-capacity, we now look for the existence of solutions for a class of diffuse measures and involving mixed operator. Keeping in mind \cref{representation}, we take $\nu$ to be a non-negative bounded radon measure which is in $L^1(\Omega_T)+L^{p^\prime}(0,T;W^{-1,p^\prime}(\Omega))$ for some $p>1$, i.e. for every $\phi\in C_c^\infty(\Omega_T)$, it holds
   \begin{equation*}
         \iint_{\Omega_T} \phi \,d \nu=\iint_{\Omega_T} f \phi \,d x d t+\int_0^T G\cdot\nabla \phi \, dxd t,
     \end{equation*}
   for some $f\in L^1(\Omega_T)$ and $G\in \left(L^{p^\prime}(\Omega_T)\right)^n$. Formally we write $\nu=f-\operatorname{div}(G)$. By \cite{diffuse}, such measures are also diffuse. We denote the set of such measures as $\overline{\mathcal{M}}_0^p(\Omega_T)$.
   \begin{definition}{\label{weaksolu}}
       Let $0<s<1<q<\infty$, $\nu\in\overline{\mathcal{M}}^q_0(\Omega_T)$ and $\mu$ is a non-negative bounded Radon measure on $\Omega_T$. A \textcolor{blue}{weak solution} to problem \cref{mainproblem} is a function $u\in L_{\mathrm{loc}}^q((0, T) ; W_{\mathrm{loc}}^{1, q}(\Omega))\cap L^1(\Omega_T)$ satisfying\begin{enumerate}
           \item $u \geq c$ a.e. in $\omega\times(t_1,t_2)$ for each $\omega\subset\subset \Omega$, $0<t_1<t_2<T$ for some $c\equiv c(\omega,t_1,t_2)>0$,\item $u\equiv 0$ in $\mathbb{R}^n\backslash\Omega_T$
and for any $\phi \in C_c^{\infty}\left(\Omega_T\right)$
       \end{enumerate}
$$
\begin{aligned}
  &-\iint_{\Omega_T} u \frac{\partial \phi}{\partial t}dxdt+\iint_{\Omega_T} \nabla u \cdot\nabla \phi\, dxdt+\int_0^T\int_{\mathbb{R}^{n}}\int_{\mathbb{R}^{n}} 
\frac{(u(x,t)-u(y,t))(\phi(x,t)-\phi(y,t))}{|x-y|^{n+2s}}dxdydt\\&\qquad\qquad=\iint_{\Omega_T} \frac{\phi}{u^{\gamma(x,t)}} d\nu+\iint_{\Omega_T}\phi\,d\mu.
\end{aligned}
$$ Finally, $u$ is the a.e. limit in $\Omega_T$ of a sequence $u_k\in C([0, T] ; L^2(\Omega)) \cap L^2(0, T ; W_0^{1, 2}(\Omega)) \cap L^{\infty}(\Omega_T)$ satisfying
$
u_k(0)=u_{0 k} \in L^{\infty}(\Omega)$ with $ u_{0 k} \rightarrow u_0$ in $L^1(\Omega)$.
   \end{definition}
   \begin{remark}
     We remark that \cref{weaksolu} is well stated. More precisely, if $\phi \in C_c^{\infty}(\Omega_T)$ and $K=\operatorname{supp} \phi$, since we have
\begin{equation*}
    \left|\iint_{\Omega_T}  \nabla u \cdot \nabla \phi \,d x dt\right| \leq \|\phi\|_{L^{\infty}(\Omega_T)}\|\nabla u\|_{L^1(K)}<\infty.
\end{equation*}
Moreover, combining \cref{embedding} and \cref{fracsolu}, it follows that
\begin{equation*}
    \left|\int_0^T\int_{\mathbb{R}^n} \int_{\mathbb{R}^n}\frac{(u(x,t)-u(y,t))(\phi(x,t)-\phi(y,t))}{|x-y|^{n+2s}} d x d ydt\right|<\infty.
\end{equation*}
Since $\phi \in C_c^{\infty}(\Omega_T)$ and $\mu$ is a non-negative Radon measure, therefore
\begin{equation*}
\iint_{\Omega_T} \phi \,d \mu<\infty   . 
\end{equation*}
Furthermore, since $\nu \in \overline{\mathcal{M}}_0^q(\Omega_T)$, so $\nu \in L^1(\Omega_T)+L^{q^\prime}(0,T;W^{-1, q^{\prime}}(\Omega))$ and due to the above property (1) along with that $\phi \in C_c^{\infty}(\Omega_T)$, we have $\frac{\phi}{u^{\delta(x,t)}} \in L^q(0,T;W_0^{1, q}(\Omega)) \cap L^{\infty}(\Omega_T)$. By keeping this fact in mind, with a little abuse of notation we denote
\begin{equation*}
    \iint_{\Omega_T} \frac{\phi}{u^{\delta(x,t)}} d \nu:=\left\langle\nu, \frac{\phi}{u^{\delta(x,t)}}\right\rangle_{L^1(\Omega_T)+L^{q^\prime}(0,T;W^{-1, q^{\prime}}(\Omega)), L^q(0,T:W_0^{1, q} (\Omega))\cap L^{\infty}(\Omega_T)}.
\end{equation*}
   \end{remark}
\begin{remark}{\label{measapx}}
    There are many ways to approximate a measure in $\overline{\mathcal{M}}_0^p(\Omega_T)$ looking for existence of solutions for problem \cref{mainproblem}; we will make the following choice: let
\begin{equation*}
    0\leq \nu^{\varepsilon}=f^{\varepsilon}-\operatorname{div}\left(G^{\varepsilon}\right),
\end{equation*}
where $f^{\varepsilon} \in C_c^{\infty}(\Omega_T)$ is a sequence of functions which converges to $f$ weakly in $L^1(\Omega)$, $G^{\varepsilon} \in C_c^{\infty}(\Omega_T)$ is a sequence of functions which converges to $G$ strongly in $\left(L^{p^{\prime}}(\Omega_T)\right)^n$. Notice that this approximation can be easily obtained via a standard convolution argument. We also assume
\begin{equation*}
    \left\|\nu^{\varepsilon}\right\|_{L^1(\Omega_T)} \leq C\|\nu\|_{\mathcal{M}(\Omega_T)}.
\end{equation*}
\end{remark}
Before stating the existence results, we define the condition $\left(P_1\right)$ below.\\
\textbf{Condition $\left(P_1\right)$:} We define a strip around the parabolic boundary by $(\Omega_T)_\delta=\{(x,t)\in\Omega_T:\operatorname{dist}((x,t),\Gamma_T)<\delta\}$ for $\delta>0$, where $\Gamma_T=(\Omega\times\{t=0\})\cup(\partial\Omega\times(0,T))$. We say that a continuous function $\gamma: \overline{\Omega}_T \rightarrow(0, \infty)$, satisfies the condition $(P_1)$, if there exist $\gamma^* \geq 1$ and $\delta>0$ such that $\gamma(x,t) \leq \gamma^*$ in $(\Omega_T)_\delta$

Our main existence results are stated as follows: 
\begin{theorem}{\label{exis1}}(Variable singular exponent)
    Assume $u_0\in L^1(\Omega)$, $\gamma\in C(\overline{\Omega}_T)$ and is locally Lipschitz continuous with respect to $x$ variable in $\Omega_T$ and satisfies the condition $(P_1)$ for some $\gamma^*\geq 1$ and $\delta>0$. Let $1<p<2-\frac{n}{n+1}$, $\nu\in\overline{\mathcal{M}}_0^p(\Omega_T)$ and is non-singular with respect to Lebesgue measure. Further assume that $\,\mu \in \mathcal{M}(\Omega_T)$ is non-negative. Then the problem \cref{mainproblem} admits a weak solution $u\in L^p_{\mathrm{loc}}(0,T;W^{1,p}_{\mathrm{loc}}(\Omega))\cap L^1(\Omega_T)$ such that 
    \begin{enumerate}
        \item[(i)] If $\gamma^*=1$, then $u\in L^p(0,T;W^{1,p}_0(\Omega))\cap L^\infty(0,T;L^1(\Omega))$.
        \item[(ii)] If $\gamma_*>1$, then $u\in L^p(t_0,T;W^{1,p}_{\mathrm{loc}}(\Omega))\cap L^\infty(0,T;L^1(\Omega))$ such that $T_l(u)\in L^2(t_0,T;W^{1,2}_{\mathrm{loc}}(\Omega))$, $T_l^\frac{\gamma^*+1}{2}(u)\in L^2(0,T;W^{1,2}_0(\Omega))$ for every $l>0$ and $0<t_0<T$.
    \end{enumerate}However, the locally Lipschitz condition is not needed when $0\leq \nu\in L^1(\Omega_T)$.
\end{theorem}
    Our next result tells that when $\mu$ turns out to be a function and the function $\gamma$ is constant, then we can relax the condition on $\nu$.
    \begin{theorem}{\label{exis2}}(Constant singular exponent)
    Let  $\gamma(x,t)\equiv \gamma$ be a constant, $u_0\in L^{\gamma+1}(\Omega)$ and define $$q=\begin{cases}
    2, \text{ if } \gamma\geq 1\\ \frac{(\gamma+1)(n+2)}{(n+\gamma+1)}\text{ if } \gamma< 1.\end{cases}$$Suppose either $\mu\in L^{\gamma+1}\left(0,T;L^{\frac{n(\gamma+1)}{n+2\gamma}}\Omega)\right)$ or $\mu\in L^r(\Omega_T)$ where $r=\frac{(n+2)(\gamma+1)}{n+2(\gamma+1)}$ is a non-negative function in $\Omega_T$. $\nu\in\overline{\mathcal{M}}_0^q(\Omega_T)$ and $\nu$ is non-singular with respect to Lebesgue measure.
       Then the equation \cref{mainproblem} admits a weak solution $u\in L^q_{\mathrm{loc}}(0,T;W^{1,q}_{\mathrm{loc}}(\Omega))\cap L^1(\Omega_T)$ such that
\begin{enumerate}
    \item [(i)] If $0<\gamma\leq 1$, then $u\in L^q(0,T:W_0^{1,q}(\Omega))\cap L^\infty(0,T;L^1(\Omega))$. Indeed $u\in L^\infty(0,T;L^2(\Omega))$ if $\gamma=1$.
    \item[(ii)] If $\gamma>1$, then $u\in L^q(t_0,T;W_{\mathrm{loc}}^{1,q}(\Omega))\cap L^\infty(0,T;L^{\gamma+1}(\Omega))$, $\forall$ $0<t_0<T$, so that $u^\frac{\gamma+1}{2}\in L^q(0,T;W_0^{1,q}(\Omega))$. 
\end{enumerate}
    \end{theorem}
    \begin{remark}
        The results obtained in \cref{exis1}, \cref{exis2} holds for equation involving a local operator only. The result obtained in \cref{exis2} perfectly fits with the non-existence result obtained in \cref{nonexistence}.
    \end{remark}
    \begin{remark}{\label{boundarybeha}}
       Notice that the results in \cref{exis1} and \cref{exis2} implies in any case, $T_l^{\eta}(u)\in L^m(0,T;W^{1,m}_0(\Omega))$ for all $l>0$ and for some $\eta\geq 1$ and $m>1$. Hence one has that, for almost every $t \in(0, T)$ and for every $l>0$, $T_l^{\eta}(u(x, t)) \in W_0^{1, m}(\Omega)$. If $\eta=1$ then $T_l(u(x, t)) \in W_0^{1, m}(\Omega)$ is known to imply
       \begin{equation}{\label{mid}}
           \lim _{\epsilon \rightarrow 0} \frac{1}{\epsilon} \int_{\Omega_\epsilon} T_l(u)=0 \text{ for a.e. }t \in(0, T), \forall l>0,
       \end{equation} where $\Omega_\epsilon=\{x\in\Omega:\operatorname{dist}\{x,\partial\Omega\}<\epsilon\}$. Otherwise, if $\eta>1$, one can write
\begin{equation*}
    \frac{1}{\epsilon} \int_{\Omega_\epsilon} T_l(u) \leq \frac{1}{\epsilon}\left(\int_{\Omega_\epsilon} T_l^{\eta}(u)\right)^{\frac{1}{\eta}}\left|\Omega_\epsilon\right|^{\frac{1}{\eta^\prime}}=\left(\frac{\left|\Omega_\epsilon\right|}{\epsilon}\right)^{\frac{1}{\eta^\prime}}\left(\frac{1}{\epsilon} \int_{\Omega_\epsilon} T_l^{\eta}(u)\right)^{\frac{1}{\eta}},
\end{equation*}
and taking $\epsilon \rightarrow 0$ in the previous, one has \cref{mid}. This also gives one notion of boundary behavior for the solutions.
    \end{remark}
    \subsubsection{Regularity results} For $1\leq q,r\leq \infty$, define the sets 
    \begin{equation*}\begin{array}{l}
         A=\{(q,r): 1< q,r\leq \infty ,\frac{1}{r}+\frac{n}{2q}<1\},\smallskip\\B=\{(q,r): 1< q,r< \infty ,\frac{1}{r}+\frac{n}{2q}>1, \frac{1}{r}<\frac{n}{n-2}\frac{1}{q}-\frac{2}{n-2}\},\smallskip\\ C=\{(q,r): 1< q,r< \infty ,\frac{1}{r}+\frac{n}{2q}>1, \frac{1}{r}\geq \frac{n}{n-2}\frac{1}{q}-\frac{2}{n-2}\}.
    \end{array}
    \end{equation*}Further, we define the following conditions 
    \begin{enumerate}
        \item[$(C_1)$] Let $\gamma(x,t)\equiv \gamma$ is a constant. If $(q,r)\in B$, then $q> \left(\frac{2^*}{1-\gamma}\right)^\prime$, and if $(q,r)\in C$, then $r>\frac{2}{1+\gamma}$.\item[$(C_2)$] Let $\gamma$ be a function and satisfies condition $(P_1)$. Then if $(q,r)\in B$, then $q> \frac{n(\gamma^*+1)}{n+2\gamma^*}$, and if $(q,r)\in C$, then $r>{1+\gamma^*}$.\item[$(C_3)$] If $(q,r)\in B$, then $\frac{1}{q}\leq \frac{1}{2}+\frac{1}{n}$ and if $(q,r)\in C$, then $ 1<\frac{1}{r}+\frac{n}{2q}\leq 1+\frac{n}{4}$.
    \end{enumerate} We now state our regularity results below. 
 \begin{theorem}{\label{regu1}}(Constant singular exponent) Let $\gamma: \overline{\Omega}_T \rightarrow[1,\infty)$ be a constant, $u_0\in L^\infty(\Omega)$ and $\nu \in L^{r_1}(0,T;L^{q_1}(\Omega))$, $ \mu\in L^{r_2}(0,T;L^{q_2}(\Omega))$ for some $r_1,r_2, q_1,q_2 \geq 1$ be two non-negative functions. Suppose $u$ is the weak solution obtained in \cref{exis1}. Then the following conclusions hold:
    \begin{enumerate}
        \item If $(q_1,r_1),(q_2,r_2)\in A$, then $u\in L^\infty(\Omega_T)$. \item If $(q_1,r_1)\in B\text{ or } C,(q_2,r_2)\in A$, then $u\in L^\infty(0,T;L^{2\sigma_1}(\Omega))\cap L^{2\sigma_1}(0,T;L^{2^*\sigma_1}(\Omega))$ where \begin{equation*}
\sigma_1=\left\{\begin{array}{lll}
\frac{q_1(n-2 )(\gamma+1)}{2(n-2 q_1 )} & \text { if } (q_1,r_1)\in B, \smallskip\\
\frac{q_1 r_1 n(\gamma+1)}{2(n r_1+2q_1-2 q_1 r_1)} & \text { if } (q_1,r_1)\in C.
\end{array}\right.    
\end{equation*}
        \item If $(q_1,r_1)\in A,(q_2,r_2)\in B\text{ or }C$ and $(C_3)$ holds, then $u\in L^\infty(0,T;L^{2\sigma_2}(\Omega))\cap L^{2\sigma_2}(0,T;L^{2^*\sigma_2}(\Omega))$ where \begin{equation*}
\sigma_2=\left\{\begin{array}{lll}
\frac{q_2(n-2 )}{2(n-2 q_2 )} & \text { if } (q_2,r_2)\in B, \smallskip\\
\frac{q_2 r_2 n}{2(n r_2+2q_2-2 q_2 r_2)} & \text { if } (q_2,r_2)\in C.
\end{array}\right.    
\end{equation*}\item If $(q_1,r_1)\in B\text{ or } C,(q_2,r_2)\in B\text{ or }C$ and $(C_3)$ above holds, then $u\in L^\infty(0,T;L^{2\sigma}(\Omega))\cap L^{2\sigma}(0,T;L^{2^*\sigma}(\Omega))$ where $\sigma=\min\{\sigma_1,\sigma_2\}$.
    \end{enumerate}  
    \end{theorem}
    \begin{theorem}{\label{regu2}}(Constant singular exponent) Let $\gamma: \overline{\Omega}_T \rightarrow(0,1)$ be a constant, $u_0\in L^\infty(\Omega)$ and $\nu \in L^{r_1}(0,T;L^{q_1}(\Omega))$, $ \mu\in L^{r_2}(0,T;L^{q_2}(\Omega))$ for some $r_1,r_2, q_1,q_2 \geq 1$ be two non-negative functions. Suppose $u$ is the weak solution obtained in \cref{exis1}. Then the following conclusions hold:
    \begin{enumerate}
        \item If $(q_1,r_1),(q_2,r_2)\in A$, then $u\in L^\infty(\Omega_T)$. \item If $(q_1,r_1)\in B\text{ or } C,(q_2,r_2)\in A$ and $(C_1)$ holds, then $u\in L^\infty(0,T;L^{2\sigma_1}(\Omega))\cap L^{2\sigma_1}(0,T;L^{2^*\sigma_1}(\Omega))$ where \begin{equation*}
\sigma_1=\left\{\begin{array}{lll}
\frac{q_1(n-2 )(\gamma+1)}{2(n-2 q_1 )} & \text { if } (q_1,r_1)\in B, \smallskip\\
\frac{q_1 r_1 n(\gamma+1)}{2(n r_1+2q_1-2 q_1 r_1)} & \text { if } (q_1,r_1)\in C.
\end{array}\right.    
\end{equation*}
        \item If $(q_1,r_1)\in A,(q_2,r_2)\in B\text{ or }C$ and $(C_3)$ holds, then $u\in L^\infty(0,T;L^{2\sigma_2}(\Omega))\cap L^{2\sigma_2}(0,T;L^{2^*\sigma_2}(\Omega))$ where \begin{equation*}
\sigma_2=\left\{\begin{array}{lll}
\frac{q_2(n-2 )}{2(n-2 q_2 )} & \text { if } (q_2,r_2)\in B, \smallskip\\
\frac{q_2 r_2 n}{2(n r_2+2q_2-2 q_2 r_2)} & \text { if } (q_2,r_2)\in C.
\end{array}\right.    
\end{equation*}\item If $(q_1,r_1)\in B\text{ or } C,(q_2,r_2)\in B\text{ or }C$ and $(C_1)$, $(C_3)$ holds, then $u\in L^\infty(0,T;L^{2\sigma}(\Omega))\cap L^{2\sigma}(0,T;L^{2^*\sigma}(\Omega))$ where $\sigma=\min\{\sigma_1,\sigma_2\}$.
    \end{enumerate}
    \end{theorem}\begin{theorem}{\label{regu3}}(Variable singular exponent) Let $\gamma: \overline{\Omega}_T \rightarrow(0,\infty)$ be a continuous function satisfying the condition $(P_1)$ for some $\gamma^*\geq 1$ and $\delta>0$, $u_0\in L^\infty(\Omega)$ and $\nu \in L^{r_1}(0,T;L^{q_1}(\Omega))$, $ \mu\in L^{r_2}(0,T;L^{q_2}(\Omega))$ for some $r_1,r_2, q_1,q_2 \geq 1$ be two non-negative functions. Suppose $u$ is the weak solution obtained in \cref{exis1}. Then the following conclusions hold:
    \begin{enumerate}
        \item If $(q_1,r_1),(q_2,r_2)\in A$, then $u\in L^\infty(\Omega_T)$. \item If $(q_1,r_1)\in B\text{ or } C,(q_2,r_2)\in A$ and $(C_2)$ holds, then $u\in L^\infty(0,T;L^{2\sigma_1}(\Omega))\cap L^{2\sigma_1}(0,T;L^{2^*\sigma_1}(\Omega))$ where \begin{equation*}
\sigma_1=\left\{\begin{array}{lll}
\frac{q_1(n-2 )}{2(n-2 q_1 )} & \text { if } (q_1,r_1)\in B, \smallskip\\
\frac{q_1 r_1 n}{2(n r_1+2q_1-2 q_1 r_1)} & \text { if } (q_1,r_1)\in C.
\end{array}\right.    
\end{equation*}
        \item If $(q_1,r_1)\in A,(q_2,r_2)\in B\text{ or }C$ and condition $(C_3)$ holds, then $u\in L^\infty(0,T;L^{2\sigma_2}(\Omega))\cap L^{2\sigma_2}(0,T;L^{2^*\sigma_2}(\Omega))$ where \begin{equation*}
\sigma_2=\left\{\begin{array}{lll}
\frac{q_2(n-2 )}{2(n-2 q_2 )} & \text { if } (q_2,r_2)\in B, \smallskip\\
\frac{q_2 r_2 n}{2(n r_2+2q_2-2 q_2 r_2)} & \text { if } (q_2,r_2)\in C.
\end{array}\right.    
\end{equation*}\item If $(q_1,r_1)\in B\text{ or } C,(q_2,r_2)\in B\text{ or }C$ and conditions $(C_2)$, $(C_3)$ holds, then $u\in L^\infty(0,T;L^{2\sigma}(\Omega))\cap L^{2\sigma}(0,T;L^{2^*\sigma}(\Omega))$ where $\sigma=\min\{\sigma_1,\sigma_2\}$.
    \end{enumerate}
    \end{theorem}
    \subsubsection{Uniqueness}
    A weak solution $u$ of \cref{mainproblem} is a finite energy solution if $u\in L^2(0,T;W^{1,2}_0(\Omega))$ and $u_t\in L^2(0,T;W^{-1,2}(\Omega))+L^1(\Omega_T)$. We state the uniqueness result we obtained for the finite energy solutions as:
    \begin{theorem}{\label{uniqueness}}
        If $0\leq \nu\in L^1(\Omega_T)$ and $\mu\in\mathcal{M}(\Omega_T)$ is non-negative and is of form $L^1(\Omega_T)+L^2(0,T;W^{-1,2}(\Omega))$, then every finite energy solution to \cref{mainproblem} is unique.
    \end{theorem}
    \subsubsection{Asymptotic behavior}
    Following is the long time behavior for weak solutions to \cref{mainproblem}.\begin{theorem}{\label{asymptotic}}
      Let $\gamma$ be continuous in $\overline{\Omega}_T$ and satisfies the condition $(P_1)$. Let $0\leq f,g \in L^1(\Omega_T)$ be two functions. Further assume $u_0,v_0\in L^1(\Omega)$ and $u,v$ be the solutions to the problem \begin{equation}{\label{asy}}
           w_t -\Delta w+(-\Delta)_sw=\frac{f}{w^{\gamma(x,t)}}+g \text{ in } \Omega_T, \qquad w>0 \text{ in }\Omega_T, \qquad w\equiv 0\text{ in }\mathbb{R}^n\backslash\Omega\times(0,T);
        \end{equation}with initial data $u_0$ and $v_0$ respectively, obtained as in \cref{exis1}. Further assume that for each $T>0$, they have been obtained as the almost everywhere limit in $\Omega_T$ of solutions to the approximated problems\begin{equation}{\label{asym1}}
    \begin{array}{c}
         (u_k)_t-\Delta u_k+(-\Delta)^s u_k=\frac{T_k(f)}{\left(u_k+\frac{1}{k}\right)^{\gamma(x,t)}} +T_k(g)\text { in } \Omega_T, \smallskip\\u_k>0 \text{ in }\Omega_T,\quad u_k=0  \text { in }(\mathbb{R}^n \backslash \Omega) \times(0, T), \smallskip\\ u_k(\cdot, 0)=T_k(u_0) \text { in } \Omega ;
    \end{array}
\end{equation} and 
 \begin{equation}{\label{asym2}}
    \begin{array}{c}
         (v_k)_t-\Delta v_k+(-\Delta)^s v_k=\frac{T_k(f)}{\left(v_k+\frac{1}{k}\right)^{\gamma(x,t)}} +T_k(g) \text { in } \Omega_T, \smallskip\\v_k>0 \text{ in }\Omega_T,\quad v_k=0  \text { in }(\mathbb{R}^n \backslash \Omega) \times(0, T), \smallskip\\ v_k(\cdot, 0)=T_k(v_0) \text { in } \Omega ;
    \end{array}
\end{equation} respectively.  Then $u,v$ are indeed global weak solutions of \cref{asy}. 
        Further, there exist positive constants $\sigma, M$ such that for a.e. $t>0$
        \begin{equation*}
            \|u(t)-v(t)\|_{L^1(\omega)}\leq M\frac{\|u_0-v_0\|_{L^1(\Omega)}}{t^{n/2}e^{\sigma t}}\quad \text{ for each }\omega\subset\subset\Omega.
        \end{equation*}In particular for each $\omega\subset\subset\Omega$\begin{equation*}
            \|u(t)-v(t)\|_{L^1(\omega)}\to 0, \text{ as } t\to \infty.
        \end{equation*} 
    \end{theorem}
    \begin{remark}
        Note that \cref{asymptotic} implies \begin{equation*}
            \|u(t)-v\|_{L^1(\omega)}\to 0, \text{ as } t\to \infty, \text{ for each }\omega\subset\subset\Omega,
        \end{equation*} where $v$ is the solution to the elliptic problem \begin{equation*}
            -\Delta v+(-\Delta)_sv=\frac{f}{v^{\gamma(x)}}+g \text{ in } \Omega, \qquad v>0 \text{ in }\Omega, \qquad v\equiv 0\text{ in }\mathbb{R}^n\backslash\Omega;
        \end{equation*} obtained as in \cite{sanjitgarain}, provided $0\leq u_0\in L^1(\Omega)$, $0\leq f,g \in L^1(\Omega)$ 
        and $\gamma (x,t)\equiv \gamma(x)$.
    \end{remark}
 \section{Preliminaries for the existence and non-existence results}{\label{sec2}} Throughout this section, we consider $\gamma:\overline{\Omega}_T\to (0,\infty)$ is a continuous function unless otherwise mentioned. First of all, we need to show the existence of a weak solution to \cref{aproxproblem}. We prove the following preliminary lemma. 
 \begin{lemma}{\label{lemma1}}
 Let $f\in L^2(\Omega_T)$ and $ g\in L^2(\Omega)$. Then the initial boundary value problem
  \begin{equation}{\label{aproxproblem1}}
    \begin{array}{c}
         v_t-\Delta v+(-\Delta)^s v=f \text { in } \Omega_T, \smallskip\\\quad v=0  \text { in }(\mathbb{R}^n \backslash \Omega) \times(0, T), \smallskip\\ v(x, 0)=g \text { in } \Omega ;
    \end{array}
\end{equation} has a unique solution $v\in L^2(0,T;W^{1,2}_0(\Omega))$ such that $v_t\in L^2(0,T;W^{-1,2}(\Omega))$ and for all $\phi \in 
L^2(0, T;W_0^{1,2}(\Omega))$, we have the equality

\begin{equation*}
    \int_0^T\left\langle v_t, \phi\right\rangle d t+\int_0^T\int_{\Omega}\nabla v\cdot \nabla \phi\, d x d t+\int_0^T \int_{\mathbb{R}^n}\int_{\mathbb{R}^n} \frac{(v(x, t)-v(y, t))(\phi(x, t)-\phi(y, t))}{|x-y|^{n+2 s}} d x d y d t=\int_0^T\int_{\Omega}f\phi \,d x  d t
\end{equation*}
holds and $v(\cdot, t) \rightarrow g$ strongly in $L^2(\Omega)$, as $t \rightarrow 0$. Further, \begin{enumerate}
    \item if $f\in L^\infty(\Omega_T)\backslash\{0\}, g\in L^\infty(\Omega)$, then $v\in L^\infty(\Omega_T)$,
    \item if $f, g\geq 0$, then $v\geq 0$,
    \item if the above two conditions on $f,g$ hold then then for each $0<t_1<T$ 
and for each $\omega \subset \subset \Omega$ fixed, there exists $C:=C\left(t_1, n,s ,\omega\right)>0$ such that
\begin{equation*}
    v(x, t) \geq C\left(t_1, n, s, \omega\right) \text { in } \omega\times [t_1,T).
\end{equation*}
    \end{enumerate}
 \end{lemma}
\begin{proof}
    For almost every $t \in (0,T)$, we define the operator
\begin{equation*}
    \mathcal{A}_t: W^{1,2}_0(\Omega) \rightarrow W^{-1,2}(\Omega)
\end{equation*}by
\begin{equation*}
    \left\langle\mathcal{A}_t(v), \phi\right\rangle= \int_\Omega \nabla v\cdot\nabla \phi\, dx+ \int_{\mathbb{R}^{n}}\int_{\mathbb{R}^{n}} \frac{(v(x)-v(y)(\phi(x)-\phi(y))}{|x-y|^{n+2s}} d x dy.
\end{equation*}
It is easy to check that $\mathcal{A}_t(v) \in W^{-1,2}(\Omega)$ whenever $v \in W^{1,2}_0(\Omega)$. Additionally, $\mathcal{A}_t$ is a monotone operator. We now define $\mathcal{A}: W^{1,2}_0(\Omega) \times I\rightarrow W^{-1,2}(\Omega)$ to be the operator defined by
\begin{equation*}
\mathcal{A}(v, t)=\mathcal{A}_t(v).   
\end{equation*}
Observe that this is well-defined. We next show that the operator $\mathcal{A}$, together with the spaces
\begin{equation*}
V=W^{1,2}_0(\Omega), \quad \mathcal{V}=L^2(0,T; W^{1,2}_0(\Omega)) \quad \text { and } \quad H=L^2(\Omega),    
\end{equation*}
fits into the framework of \cref{evolu}. Since $\Omega$ is bounded, $W^{1,2}_0(\Omega)$ is dense and continuously embedded in $L^2(\Omega)$. This follows from the fact that smooth compactly supported functions are dense in both spaces. Note that $\mathcal{A}$ inherits the property of monotonicity from $\mathcal{A}_t$ since
\begin{equation*}
    \langle\mathcal{A}(u, t)-\mathcal{A}(v, t), u-v\rangle= \left\langle\mathcal{A}_t(u)-\mathcal{A}_t(v), u -v\right\rangle \geq 0.
\end{equation*} By \cref{embedding2} and H\"older inequality, it is easy to observe that
\begin{equation*}
|\langle\mathcal{A}(v, t), \phi\rangle| \leq C\|v\|_{W^{1,2}_0(\Omega)}\|\phi\|_{W^{1,2}_0(\Omega)}
,\end{equation*}
so that for sll $v\in{W^{1,2}_0(\Omega)}$, \begin{equation*}
    \|\mathcal{A}(v, t)\|_{W^{-1,2}(\Omega)}\leq C\|v\|_{W^{1,2}_0(\Omega)}.
\end{equation*}
Thus, in order to verify $(2)$ of \cref{evolu}, we are left with proving hemicontinuity. For this, fixed $t \in (0,T)$ and $\lambda, \lambda_0 \in \mathbb{R}$, we consider
\begin{equation*}
\langle\mathcal{A}(u+\lambda v, t), v\rangle-\langle\mathcal{A}(u+\lambda_0 v, t), v\rangle, \quad \text { for } u, v \in W^{1,2}_0(\Omega).    
\end{equation*}
In order to show that this differences goes to $0$ as $\lambda$ goes to $\lambda_0$, it is sufficient to write
\begin{equation*}
\begin{array}{rcl}
    \langle\mathcal{A}(u+\lambda v, t), v\rangle-\langle\mathcal{A}(u+\lambda_0 v, t), v\rangle&=& \langle\mathcal{A}_t(u+\lambda v)-\mathcal{A}_t(u+\lambda_0 v), v\rangle\\&=&(\lambda-\lambda_0)\left(\int_\Omega |\nabla v|^2 dx+ \int_{\mathbb{R}^{n}}\int_{\mathbb{R}^{n}} \frac{|v(x)-v(y)|^2}{|x-y|^{n+2s}} d x dy\right).\end{array}
\end{equation*}Therefore, for fix $u,v \in W^{1,2}_0(\Omega)$\begin{equation*}
    \left|  \langle\mathcal{A}(u+\lambda v, t), v\rangle-\langle\mathcal{A}(u+\lambda_0 v, t), v\rangle\right|\leq C|\lambda-\lambda_0|,
\end{equation*}
which proves that $\mathcal{A}$ is hemicontinuous for almost every $t \in (0,T)$.
\\Finally, as for hypothesis $(3)$ of \cref{evolu}, we observe by Poincar\'e inequality
\begin{equation*}
    \langle\mathcal{A}(v, t), v\rangle \geq \int_\Omega|\nabla v|^2dx\geq C\|v\|^2_{W^{1,2}_0(\Omega)}.
\end{equation*}
Thus one can apply \cref{evolu}, and for every $u_0 \in L^2(\Omega)$ obtain a unique solution
\begin{equation*}
v \in W_2(0,T)=\{\varphi \in L^2(0,T ; W^{1,2}_0(\Omega)): \varphi^{\prime} \in L^{2}(0,T ;W^{-1,2}(\Omega))\},    
\end{equation*}
to the problem
$v^{\prime}(t)+\mathcal{A}(v(t), t)=f(t)$ in $L^{2}(0,T ;W^{-1,2}(\Omega))$, with $v(0)=g$ in $L^2(\Omega)$. That is for all $\phi \in 
L^2(0, T;W_0^{1,2}(\Omega))$, we have the equality
\begin{equation*}
    \int_0^T\left\langle v_t, \phi\right\rangle d t+\int_0^T\int_{\Omega}\nabla v\cdot \nabla \phi\, d x d t+ \int_0^T \int_{\mathbb{R}^{n}}\int_{\mathbb{R}^{n}} \frac{(v(x, t)-v(y, t))(\phi(x, t)-\phi(y, t))}{|x-y|^{n+2 s}} d x d y d t=\int_0^T\int_{\Omega}f\phi \,d x  d t.
\end{equation*}
Observe that, we also have $v \in C([0,T] ; L^2(\Omega))$ and $v(\cdot,0)=g$ in $L^2(\Omega)$. Therefore $v(\cdot, t) \rightarrow g$ strongly in $L^2(\Omega)$, as $t \rightarrow 0$. Once the existence is proved for \cref{aproxproblem1}, we now show the properties mentioned. \smallskip\begin{enumerate}
    \item  Observing that $(v(x,t)-v(y,t))((v(x,t)-k)_+-(v(y,t)-k)_+)\geq 0$, for each $k$, we can show that for $(f,g)\in L^\infty(\Omega_T)\times L^\infty(\Omega)$, the weak solution $v$ satisfies \begin{equation*}
        \sup_{\Omega_T} v\leq C\|f\|_{L^\infty(\Omega_T)}+\|g\|_{L^\infty(\Omega)}.
    \end{equation*} The proof will follow exactly similar to that of [\citealp{bdd}, Theorem 4.2.1]. Finally repeating the same steps for $-v$, we get \begin{equation*}
         \inf_{\Omega_T} v\geq -C\|f\|_{L^\infty(\Omega_T)}-\|g\|_{L^\infty(\Omega)}.
    \end{equation*} Thus $v\in L^\infty(\Omega_T)$.
    \item Now we show that $v\geq 0$ 
provided $f$ and $g$ are non-negative. 
We write $v=v_+-v_-$, where $v_+=\operatorname{max}{\{v,0}\}
$ and $v_-=\operatorname{max}{\{-v,0}\}
$ and take $\phi=v_-\chi_{(0,\tilde{t})}, 0<\tilde{t}<T$, as a test function. 
Since $f\geq 0$, and $\phi \geq 0$, we have
\begin{equation}{\label{contradiction1}}
0\leq \int_0^{T}\int_{\Omega}f\phi \,d x d t=\int_0^{\tilde{t}}\int_{\Omega}fu_{-} d x d t.
\end{equation}
On the other hand, since $u$ is $0$ in $(\mathbb{R}^n\backslash \Omega)\times (0,T)$, we have that
\begin{equation*}
    \begin{array}{c}
\int_0^{T}\iint_{\mathbb{R}^{2 n} }\frac{(v(x,t)-v(y,t))(\phi(x,t)-\phi(y,t))}{|x-y|^{n+2s}}  d x d y d t\smallskip\\ =\int_0^{\tilde{t}}\iint_{\mathbb{R}^{2 n} \backslash(\Omega^c \times \Omega^c)}\frac{(v(x,t)-v(y,t))(v_{-}(x,t)-v_{-}(y,t))}{|x-y|^{n+2s}} d x d y d t \smallskip\\
=\int_0^{\tilde{t}} \int_{\Omega} \int_{\Omega}\frac{(v(x,t)-v(y,t))(v_{-}(x,t)-v_{-}(y,t))}{|x-y|^{n+2s}} d x d y d t+2 \int_0^{\tilde{t}}\int_{\Omega} \int_{\Omega^c}\frac{v(x,t) v_{-}(x,t)}{|x-y|^{n+2s}} d y d x d t .
    \end{array}
\end{equation*}
\smallskip Moreover, $(v_{+}(x,t)-v_{+}(y,t))(v_{-}(x,t)-v_{-}(y,t)) \leq 0$, and thus 
\begin{equation*}
\begin{array}{l}
\int_0^{\tilde{t}}\int_{\Omega} \int_{\Omega}\frac{(v(x,t)-v(y,t))(v_{-}(x,t)-v_{-}(y,t))}{|x-y|^{n+2s}} d x d y d t\smallskip\\
\leq-\int_0^{\tilde{t}}\int_{\Omega} \int_{\Omega}\frac{(v_{-}(x,t)-v_{-}(y,t))^2}{|x-y|^{n+2s}}  d x d y d t\leq 0 .
\end{array}
\end{equation*}
Further, we have
\begin{equation*}
\int_0^{\tilde{t}}\int_{\Omega} \int_{\Omega^c}\frac{v(x,t) v_{-}(x,t)}{|x-y|^{n+2s}} d y d x d t =-\int_0^{\tilde{t}}\int_{\Omega} \int_{\Omega^c}\frac{v^2_{-}(x,t)}{|x-y|^{n+2s}} d y d x d t \leq 0 .    
\end{equation*}
Therefore, we have shown that
\begin{equation*}
   \int_0^{T}\iint_{\mathbb{R}^{2 n} }\frac{(v(x,t)-v(y,t))(\phi(x,t)-\phi(y,t))}{|x-y|^{n+2s}}  d x d y d t\leq 0.
\end{equation*}
Moreover, 
\begin{equation*}
    \int_0^T\int_{\Omega}\nabla v\cdot \nabla \phi\, d x d t=
    -\int_0^{\tilde{t}}\int_{\Omega}| \nabla v_-|^2 d x d t\leq 0.
\end{equation*}
Now 
since $v\in C([0,T];L^2(\Omega))$, $v(\cdot,0)=g\geq 0$, so $v_-(\cdot,0)\equiv 0$, and we get
\begin{equation*}
    \begin{array}{rcl}
         \int_0^T\left\langle v_t, \phi\right\rangle d t
         &=&\int_0^{\tilde{t}}\int_\Omega  \frac{\partial v}{\partial t} v_- d x d t
        = \int_\Omega\int_{v(x,0)}^{v(x,\tilde{t})}\sigma_-d\sigma d x 
         =-\frac{1}{2}\int_\Omega (v_-(x,{\tilde{t}}))^2d x.
    \end{array}
\end{equation*}
Combining the above three inequalities, we get from \cref{contradiction1} that 
\begin{equation*}
\begin{array}{rcl}
0\leq    \int_0^T\int_{\Omega}f\phi\, d x d t&=& \int_0^T\left\langle v_t, \phi\right\rangle d t+\int_0^T\int_{\Omega}\nabla v\cdot \nabla \phi \,d x d t\smallskip\\&&+ \int_0^T \int_{\mathbb{R}^{n}}\int_{\mathbb{R}^{n}}\frac{(v(x, t)-v(y, t))(\phi(x, t)-\phi(y, t))}{|x-y|^{n+2 s}} d x d y d t\smallskip\\&&\leq-\frac{1}{2}\int_\Omega (v_-(x,{\tilde{t}}))^2d x \leq 0,
    \end{array}
\end{equation*}
and this gives that $||v_{-}(\cdot,\tilde{t})||_{L^2(\Omega)}=0$ 
for each $\tilde{t}>0$. 
Therefore $v_{-}\equiv 0$. So we conclude that $v\geq 0$.
\item Finally, we can use the weak Harnack inequality in \cite{weakharnack} to deduce the uniform positivity. One can see [\citealp{pap}, Proposition 2.18] for a detailed proof.
\end{enumerate}
\end{proof}
Now we show the existence of solutions to \cref{aproxproblem}. The proof, which is based on the Schauder fixed-point theorem, is quite standard, but we sketch it for completeness.
\begin{lemma}{\label{exisforaprox}}
    For any fixed $k \in \mathbb{N}$, there exists a unique non-negative solution $u_k \in L^2(0, T ; W_0^{1, 2}(\Omega)) \cap L^{\infty}(\Omega_T)$ such that $(u_k)_t \in L^{2}(0, T ; W^{-1, 2}(\Omega))$ to problem \cref{aproxproblem}. In particular $u_k\in C([0,T],L^2(\Omega))$.\\ Further, if $\nu$ is non-singular with respect to the Lebesgue measure, then 
    for each $t_0>0$ and $\omega\subset\subset\Omega$, $\exists C(\omega,t_0,n,s)$, a constant which does not depend on $k$,  such that $\forall k$, $u_k\geq C(\omega,t_0,n,s)$ in $\omega\times[t_0,T)$.
\end{lemma}
\begin{proof}Let $k \in \mathbb{N}$ be fixed, $v \in L^2(\Omega_T)$ and consider the following problem
   \begin{equation}{\label{aproxproblem11}}
    \begin{array}{c}
         w_t-\Delta w+(-\Delta)^s w=\frac{\nu_k}{\left(v_++\frac{1}{k}\right)^{\gamma(x,t)}} +\mu_k \text { in } \Omega_T, \smallskip\\w>0 \text{ in }\Omega_T,\quad w=0  \text { in }(\mathbb{R}^n \backslash \Omega) \times(0, T), \smallskip\\ w(x, 0)=u_{0k}(x) \text { in } \Omega .
    \end{array}
\end{equation}
As $\nu_k,\mu_k$ are non-negative and belongs to $L^\infty(\Omega_T)$, $u_{0k}\in L^\infty(\Omega)$ and $\gamma>0$, it follows from \cref{lemma1} that problem \cref{aproxproblem11} admits a unique solution $0\leq w \in L^2(0, T ; W_0^{1, 2}(\Omega)) \cap C([0, T] ; L^2(\Omega))$ such that $w_t \in L^{2}(0, T ; W^{-1, 2}(\Omega))$ for every fixed $v \in L^2(\Omega_T)$. Furthermore $w$ belongs to $L^{\infty}(\Omega_T)$.
Our aim is to prove the existence of a fixed point for the map
\begin{equation*}
G: L^2(\Omega_T) \rightarrow L^2(\Omega_T)    
\end{equation*}
which for any $v \in L^2(\Omega_T)$ gives the weak solution $w$ to \cref{aproxproblem11}. We take $w$ as test function in the weak formulation of \cref{aproxproblem11} obtaining
\begin{equation}{\label{mid1}}
   \int_0^T\langle w_t, w\rangle d t+\iint_{\Omega_T}|\nabla w|^2 + \int_0^T \iint_{\mathbb{R}^{2n}} \frac{|w(x, t)-w(y, t)|^2}{|x-y|^{n+2 s}} d x d y d t=\iint_{\Omega_T}\frac{\nu_kw}{\left(v_++\frac{1}{k}\right)^{\gamma(x,t)}}+\iint_{\Omega_T} \mu_k w    .
\end{equation}
By the classical integration by parts formula, one has
\begin{equation*}
\frac{1}{2} \int_{\Omega} w^2(x, T)-\frac{1}{2} \int_{\Omega} u_{0k}^2(x)\leq \int_0^T\langle w_t, w\rangle dt+\int_0^T \iint_{\mathbb{R}^{2n}} \frac{|w(x, t)-w(y, t)|^2}{|x-y|^{n+2 s}} d x d y d t.
\end{equation*}
For the right hand side of \cref{mid1} by the Hölder inequality
\begin{equation*}
    \iint_{\Omega_T}\frac{\nu_kw}{\left(v_++\frac{1}{k}\right)^{\gamma(x,t)}}+\iint_{\Omega_T} \mu_k w    \leq C \|\nu_k\|_{L^{\infty}(\Omega_T)}k^{\gamma_1}\left(\iint_{\Omega_T}|w|^2\right)^{\frac{1}{2}}+C\|\mu_k\|_{L^{\infty}(\Omega_T)}\left(\iint_{\Omega_T}|w|^2\right)^{\frac{1}{2}},
\end{equation*} where $\gamma_1=\max_{\overline{\Omega}_T}\gamma$.
Therefore, dropping a positive term
\begin{equation*}
 \iint_{\Omega_T}|\nabla w|^2dxdt \leq C\left(\iint_{\Omega_T}|w|^2dxdt\right)^{\frac{1}{2}}+\frac{1}{2} \int_{\Omega} u_{0_k}^2 dx,
    \end{equation*}
    for some $C\equiv C(k,\nu_k,\mu_k,\Omega_T)$.
Poincaré inequality (\cref{p}) implies for some constant $C$,
\begin{equation}{\label{mid2}}
 \iint_{\Omega_T}| w|^2dxdt \leq C\left(\iint_{\Omega_T}|w|^2dxdt\right)^{\frac{1}{2}}+\frac{1}{2} \int_{\Omega} u_{0_k}^2 dx,
    \end{equation}
which implies by Young's inequality
\begin{equation}{\label{mid21}}
\left(\iint_{\Omega_T}|w|^2dxdt\right)^{\frac{1}{2}} \leq C.
\end{equation}
Note that, the constant $C$ in \cref{mid21} is independent of $v$, and so the ball $B:=B_C(0)$ of $L^2(\Omega_T)$ of radius $C$ is invariant for $G$.

Now we check the continuity of the map $G$. Let $v_m$ be a sequence of functions converging to $v$ in $L^2(\Omega_T)$.
By the dominated convergence theorem one has that \begin{equation}{\label{mid3}}\frac{\nu_k}{\left((v_m)_++\frac{1}{k}\right)^{\gamma(x,t)}} \rightarrow \frac{\nu_k}{\left(v_++\frac{1}{k}\right)^{\gamma(x,t)}}  \text{ in }L^2(\Omega_T).\end{equation} Let $w_m:=G(v_m)$ and $w:=G(v)$. Taking $w_m-w$ as a test function in the weak formulation of both $w_m$ and $w$, and then subtracting, we get

\begin{equation*}
    \begin{array}{l}
        \quad  \int_0^T\langle (w_m-w)_t, (w_m-w)\rangle d t+\iint_{\Omega_T}|\nabla (w_m-w)|^2 d x d t\\\quad+ \int_0^T \iint_{\mathbb{R}^{2n}} \frac{|(w_m-w)(x, t)-(w_m-w)(y, t)|^2}{|x-y|^{n+2 s}} d x d y d t\\=\iint_{\Omega_T}\left(\frac{\nu_k}{\left((v_m)_++\frac{1}{k}\right)^{\gamma(x,t)}}-\frac{\nu_k}{\left(v_++\frac{1}{k}\right)^{\gamma(x,t)}}\right)    (w_m-w)dxdt\\\leq \left\|\frac{\nu_k}{\left((v_m)_++\frac{1}{k}\right)^{\gamma(x,t)}}-\frac{\nu_k}{\left(v_++\frac{1}{k}\right)^{\gamma(x,t)}}\right\|_{L^2(\Omega_T)}\|w_m-w\|_{L^2(\Omega)}.
    \end{array}
\end{equation*} Now as $w_m(\cdot,0)=w(\cdot,0)=u_{0k}$, so dropping positive terms and using Poincar\'e inequality we get
\begin{equation*}
    \|w_m-w\|^2_{L^2(\Omega)}\leq C\iint_{\Omega_T}|\nabla (w_m-w)|^2 d x d t\leq C\left\|\frac{\nu_k}{\left((v_m)_++\frac{1}{k}\right)^{\gamma(x,t)}}-\frac{\nu_k}{\left(v_++\frac{1}{k}\right)^{\gamma(x,t)}}\right\|_{L^2(\Omega_T)}\|w_m-w\|_{L^2(\Omega)}.
\end{equation*}
This, along with \cref{mid3} gives that $G(v_m)$ converges to $G(v)$ in $L^2(\Omega_T)$.

Lastly, we need $G(B)$ to be relatively compact. Let $v_m$ be a bounded sequence in $L^2(\Omega_T)$, and let $w_m=G(v_m)$. Reasoning as to obtain \cref{mid2}, we have
\begin{equation*}
 \iint_{\Omega_T}| w_m|^2dxdt \leq C\left(\iint_{\Omega_T}|w_m|^2dxdt\right)^{\frac{1}{2}}+\frac{1}{2} \int_{\Omega} u_{0_k}^2 dx,
    \end{equation*}
where $w_m:=G(v_m)$ and $C$ is clearly independent of $v_m$. Recalling \cref{mid21} this means that $w_m$ is bounded with respect to $m$ in $L^2(0, T ; W_0^{1, 2}(\Omega))$. We also deduce from the equation that $(w_m)_t$ is bounded with respect to $m$ in $L^{2}(0, T ; W^{-1, 2}(\Omega))$. Hence $w_m$ admits a strongly convergent subsequence in $L^2(\Omega_T)$ (see [\citealp{conv}, Corollary 4]). This concludes the proof of existence by Schauder fixed-point theorem.

Now we show, the solution is unique. For this let $u_k$ and $v_k$ be two solutions of \cref{aproxproblem}. Choosing $(u_k-v_k)_+$ as a test function in the weak formulation of \cref{aproxproblem} and take their difference to get
 \begin{equation}{\label{eeqn11}}
    \begin{array}{l}
        \quad  \int_0^T\langle (u_k-v_k)_t, (u_k-v_k)_+\rangle d t+\iint_{\Omega_T}|\nabla (u_k-v_k)_+|^2 d x d t\\\quad+ \int_0^T \iint_{\mathbb{R}^{2n}} \frac{((u_k-v_k)(x, t)-(u_k-v_k)(y, t))((u_k-v_k)_+(x,t)-(u_k-v_k)_+(y,t))}{|x-y|^{n+2 s}} d x d y d t\smallskip\\=\iint_{\Omega_T}\nu_k\left(\frac{1}{\left(u_k+\frac{1}{k}\right)^{\gamma(x,t)}}-\frac{1}{\left(v_k+\frac{1}{k}\right)^{\gamma(x,t)}}  \right)  (u_k-v_k)_+dxdt\leq 0.
    \end{array}
\end{equation}Now for the nonlocal term in the left, observe that $(w(x,t)-w(y,t))(w(x,t)_+-w(y,t)_+)\geq 0$, for any $w$ and thus the term is non-negative. As $u_k(\cdot,0)=v_k(\cdot, 0)$, so 
\begin{equation*}
    \begin{array}{rcl}
         \int_0^T\left\langle (u_k-v_k)_t, (u_k-v_k)_+\right\rangle d t
         &=&\int_0^{T}\int_\Omega  \frac{\partial (u_k-v_k)}{\partial t} (u_k-v_k)_+d x d t
       \smallskip\\& =& \int_\Omega\int_{0}^{(u_k-v_k)(x,T)}\sigma_+d\sigma d x 
         =\frac{1}{2}\int_\Omega ((u_k-v_k)_+(x,T))^2d x\geq 0.
    \end{array}
\end{equation*} Therefore we obtain from \cref{eeqn11} that 
\begin{equation*}
    0\leq \iint_{\Omega_T}|\nabla (u_k-v_k)_+|^2 d x d t\leq 0,
\end{equation*} which by Poincar\'e inequality implies $\iint_{\Omega_T}|(u_k-v_k)_+|^2 d x d t=0$, which gives $ u_k\leq v_k$ in $\Omega_T$. Similarly one can obtain $v_k\leq u_k$ and thus the solution is unique.

Finally, if $\nu$ is non-singular (with respect to the Lebesgue measure) non-negative bounded
Radon measure on $\Omega_T$, then
\begin{equation*}
    \nu=\nu_a+\nu_s,
\end{equation*}
where $0\neq \nu_a \ll \mathcal{L}$ and $\nu_s\perp \mathcal{L}$. By Radon-Nikodym theorem [\citealp{royden}, Page 382], there exists a non-negative Lebesgue measurable function $f$ such that for every measurable set $E \subset \Omega_T$,
\begin{equation*}
\nu_a(E)=\iint_E f d x dt.    
\end{equation*}
Furthermore, as $\nu$ is bounded then $f \in L^1(\Omega_T)$. Therefore we take $\nu_k=T_k(f)+(\nu_s)_k$ in \cref{aproxproblem}, where $\{(\nu_s)_k\}_k$ is a sequence of non-negative $L^\infty(\Omega_T)$ functions (which are bounded unifomly in $L^1(\Omega_T)$) and converges in narrow topology to $\nu_s$.\\Consider now $w$ to be unique non-negative solution (obtained as in the above step) to the problem   \begin{equation}{\label{eqn}}
    \begin{array}{c}
         w_t-\Delta w+(-\Delta)^s w=\frac{T_1(f)}{(w+1)^{\gamma(x,t)}} +\mu_k \text { in } \Omega_T, \smallskip\\w>0 \text{ in }\Omega_T,\quad w=0  \text { in }(\mathbb{R}^n \backslash \Omega) \times(0, T), \smallskip\\ w(x, 0)=0
         \text { in } \Omega .
    \end{array}
\end{equation} Take $(w-u_k)_+$ as a test function in \cref{eqn} and \cref{aproxproblem} respectively and then subtract to get \begin{equation}{\label{eqn11}}
    \begin{array}{l}
        \quad  \int_0^T\langle (w-u_k)_t, (w-u_k)_+\rangle d t+\iint_{\Omega_T}|\nabla (w-u_k)_+|^2 d x d t\\\quad+ \int_0^T \iint_{\mathbb{R}^{2n}} \frac{((w-u_k)(x, t)-(w-u_k)(y, t))((w-u_k)_+(x,t)-(w-u_k)_+(y,t))}{|x-y|^{n+2 s}} d x d y d t\\=\iint_{\Omega_T}\left(\frac{T_1(f)}{\left(w+1\right)^{\gamma(x,t)}}-\frac{T_k(f)+(\nu_s)_k}{\left(u_k+\frac{1}{k}\right)^{\gamma(x,t)}}-\mu_k\right)    (w-u_k)_+dxdt\leq 0.
    \end{array}
\end{equation}Observe that the nonlocal term in the left is non-negative. As $w(\cdot,0)=0
\leq u_{0k}=u_k(\cdot, 0)$, so taking $v=(w-u_k)$, we get $v_+(\cdot,0)\equiv 0$, and hence
\begin{equation*}
    \begin{array}{rcl}
         \int_0^T\left\langle v_t, v_+\right\rangle d t
         &=&\int_0^{T}\int_\Omega  \frac{\partial v}{\partial t} v_+d x d t
        = \int_\Omega\int_{v(x,0)}^{v(x,T)}\sigma_+d\sigma d x 
         =\frac{1}{2}\int_\Omega (v_+(x,T))^2d x\geq 0.
    \end{array}
\end{equation*} Therefore we obtain from \cref{eqn11} that 
\begin{equation*}
    0\leq \int_{\Omega_T}|\nabla (w-u_k)_+|^2 d x d t\leq 0,
\end{equation*} which by Poincar\'e inequality implies $\int_{\Omega_T}|(w-u_k)_+|^2 d x d t=0$, which gives $w\leq u_k$ in $\Omega_T$. As choice of $k$ was arbitrary, one can now use property $(3)$ of \cref{lemma1} to deduce the conclusion.
\end{proof}
For fixed $l>0$ we will make use of the truncation functions $T_l$ and $G_l$ defined, respectively as $T_l(r)= \max (-l, \min (r, l))$, and $G_l(r)=r-T_l(r)$. For $\eta, \delta>0$, define
\begin{equation}{\label{T_k}}
\tilde{T}_{l, \eta}(r):=\int_0^r T_l^\eta(\theta) d \theta    
\end{equation}
and\begin{equation}{\label{beta}}
\beta_\delta(r):= \begin{cases}1 & r \geq 2\delta \\ \frac{r}{\delta}-1 & \delta<r<2 \delta \\ 0 & r \leq  \delta.\end{cases}
\end{equation}\textbf{Preliminary for the proof of \cref{exis1}}
\begin{lemma}{\label{preli1}}(Uniform boundedness)
Let $u_0\in L^1(\Omega)$. Assume that the function $\gamma: \overline{\Omega}_T \rightarrow(0, \infty)$ is continuous which satisfies the condition $(P_1)$ for some $\gamma^* \geq 1$ and for some $\delta>0$. Let $k \in \mathbb{N}$ and assume that the solution of \cref{aproxproblem} obtained in \cref{exisforaprox} is denoted by $u_k$. Let $\nu,\mu$ be non-negative bounded Radon measures and $\nu$ be non-singular with respect to the Lebesgue measure. If the sequences of non-negative functions $\{\nu_k\}_{k \in \mathbb{N}},\{\mu_k\}_{k \in \mathbb{N}} \subset L^\infty(\Omega_T)$ are uniformly bounded in $L^1(\Omega_T)$, then the following conclusions hold:
\begin{enumerate}
    \item If $\gamma^*=1$, then $\{u_k\}_{k \in \mathbb{N}}$ is bounded in $L^q(0,T;W_0^{1, q}(\Omega))\cap L^\infty(0,T;L^1(\Omega))$ for every $1<q<2-\frac{n}{n+1}$,
\item If $\gamma^*>1$, then the sequence $\{u_k\}_{k \in \mathbb{N}}$ is uniformly bounded in $L^q(t_0,T;W_{\mathrm{loc}}^{1, q}(\Omega)) \cap L^\infty(0,T;L^1(\Omega))$ for every $1<q<2-\frac{n}{n+1}$ and $0<t_0<T$. Moreover, the sequences $\{G_l(u_k)\}_{k \in \mathbb{N}},\{T_l(u_k)\}_{k \in \mathbb{N}}$ and $\{T_l^{\frac{\gamma^*+1}{2}}(u_k)\}_{k \in \mathbb{N}}$ are uniformly bounded in $L^q(0,T;W_0^{1, q}(\Omega)), L^2(t_0,T;W_{\mathrm{loc}}^{1,2}(\Omega))$ and in $L^2(0,T;W_0^{1, 2}(\Omega))$ respectively for every fixed $l>0$, $0<t_0<T$ and $1<q<2-\frac{n}{n+1}$.
\end{enumerate}
\end{lemma}
\begin{proof}
     As $u_k\in  L^2(0,T; W^{1,2}_0(\Omega))$, so for fix $\tau\in(0,T]$ we choose $T_1^{\gamma^*}(u_k) \chi_{(0, \tau)}$ as a test function in \cref{aproxproblem}, to get
     
\begin{equation}{\label{imp1}}
    \begin{array}{l}
      \quad  \int_0^\tau \int_{\Omega}(u_k)_tT_1^{\gamma^*}(u_k)d x d t +\gamma^*\int_0^\tau\int_\Omega |\nabla T_1(u_k)|^2T_1(u_k)^{\gamma^*-1}
      d x d t\smallskip\\
 +\underbrace{\int_0^\tau \int_{\mathbb{R}^n}\int_{\mathbb{R}^n} \frac{(u_k(x, t)-u_k(y, t))(T_1^{\gamma^*}(u_k)(x,t)-T_1^{\gamma^*}(u_k)(y,t))}{|x-y|^{n+2 s}} d x d y d t}_{\geq 0}\smallskip\\
= \int_0^\tau\int_{\Omega} \left(\frac{\nu_k}{\left(u_k+\frac{1}{k}\right)^{\gamma(x,t)}} +\mu_k\right)T_1^{\gamma^*}(u_k)
d x d t .
    \end{array}
\end{equation}
Note that $T_1^{\gamma^*}(r)$ is an increasing function in $r$ and hence the nonlocal term is non-negative. Now it holds\begin{equation}{\label{imp2}}
\begin{array}{l}
\quad\int_0^{\tau}\int_\Omega(\tilde{T}_{1, \gamma^*}(u_k))_t +\frac{4\gamma^*}{(\gamma^*+1)^2}\int_0^\tau\int_\Omega |\nabla T_1^{\frac{\gamma^*+1}{2}}(u_k)|^2d x d t\smallskip\\\leq \iint_{\Omega_T} \frac{\nu_kT_1^{\gamma^*}(u_k)}{\left(u_k+\frac{1}{k}\right)^{\gamma(x,t)}} d x d t+\|\mu_k\|_{L^1(\Omega_T)} \smallskip\\
=\iint_{(\Omega_T)_\delta} \frac{\nu_kT_1^{\gamma^*}(u_k)}{\left(u_k+\frac{1}{k}\right)^{\gamma(x,t)}} d x d t+\iint_{(\omega_{{T}})_\delta} \frac{\nu_kT_1^{\gamma^*}(u_k)}{\left(u_k+\frac{1}{k}\right)^{\gamma(x,t)}} d x d t +\|\mu_k\|_{L^1(\Omega_T)} \smallskip \\
 \leq \iint_{(\Omega_T)_\delta} \frac{\nu_kT_1^{\gamma^*}(u_k)}{\left(u_k+\frac{1}{k}\right)^{\gamma(x,t)}} d x d t+\iint_{(\omega_{{T}})_\delta}  \frac{\nu_kT_1^{\gamma^*}(u_k)}{C_{(\omega_{{T}})_\delta}^{\gamma(x,t)}} d x d t+\|\mu_k\|_{L^1(\Omega_T)} \smallskip\\
 \leq \iint_{(\Omega_T)_\delta\cap\left\{u_k \leq 1\right\}}  \nu_k u_k^{\gamma^*-\gamma(x,t)} d x d t +\iint_{(\Omega_T)_\delta\cap\left\{u_k \geq 1\right\}}  \frac{\nu_kT_1^{\gamma^*}(u_k)}{\left(u_k+\frac{1}{k}\right)^{\gamma(x,t)}} d x d t 
\smallskip\\\quad+\iint_{(\omega_{{T}})_\delta}  \frac{\nu_k}{C_{(\omega_{{T}})_\delta}^{\gamma(x,t)}} d x d t+\|\mu_k\|_{L^1(\Omega_T)}  \\\leq\|\mu_k\|_{L^1(\Omega_T)}+2\|\nu_k\|_{L^1(\Omega_T)}+\left\|C_{(\omega_{{T}})_\delta}^{-\gamma(\cdot)}\right\|_{L^{\infty}(\Omega_T)}\|\nu_k\|_{L^1(\Omega_T)}\leq C,
\end{array}
\end{equation}
where $(\omega_{{T}})_\delta=\Omega_T\backslash(\Omega_T)_\delta$, $\tilde{T}_{1, \gamma^*}(r)$ is defined in \cref{T_k}. Here, we have used the uniform positivity condition inside compact sets, as obtained in \cref{exisforaprox} in the third line. Further, observe that $\tilde{T}_{1, \eta}(r) \geq r-1$, to get
\begin{equation}{\label{imp3}}
    \int_{\Omega} u_k(x, \tau) +\frac{4\gamma^*}{(\gamma^*+1)^2}\int_0^\tau\int_\Omega |\nabla T_1^{\frac{\gamma^*+1}{2}}(u_k)|^2d x d t\leq C+|\Omega|+\int_{\Omega} \tilde{T}_{1, \gamma^*}(u_0)\leq C+\int_\Omega u_0dx.
\end{equation}
The right hand side of \cref{imp3} is bounded with respect to $\tau$; hence taking the supremum on $\tau$ we obtain that
\begin{equation}{\label{imp4}}
\left\|u_k\right\|_{L^{\infty}(0, T ; L^1(\Omega))} +\int_0^T\int_\Omega |\nabla T_1^{\frac{\gamma^*+1}{2}}(u_k)|^2d x d t\leq C    .
\end{equation}
Let $l$ be a fixed integer and note that $T_1(G_l(r))$ is defined by
\begin{equation}{\label{TKKK}}
    T_1(G_l(r))=\begin{cases}
        1 & \text { if } \quad r>l+1 \\
r-l & \text { if } \quad l\leq r \leq l+1 \\
r+l & \text { if } \quad-l-1<r<-l \\
-1 & \text { if } \quad r \leq-l-1 \\
0 & \text { if } \quad-l<r<l.
    \end{cases}
\end{equation}
\textbf{Case 1.} We first assume $\gamma^*=1$. Now take $T_1(G_l(u_k))$ as a test function in \cref{aproxproblem}, to get 
\begin{equation*}{\label{imp5}}
    \begin{array}{l}
      \quad  \int_0^T\int_{\Omega}(u_k)_tT_1(G_l(u_k))d x d t +\iint_{\Omega_T\cap \{l\leq u_k\leq l+1\}} |\nabla u_k|^2d x d t\smallskip\\
 +\underbrace{\int_0^T\int_{\mathbb{R}^n}\int_{\mathbb{R}^n} \frac{(u_k(x, t)-u_k(y, t))(T_1(G_l(u_k))(x,t)-T_1(G_l(u_k))(y,t))}{|x-y|^{n+2 s}} d x d y d t}_{\geq 0}\\
= \int_0^T\int_{\Omega} \left(\frac{\nu_k}{\left(u_k+\frac{1}{k}\right)^{\gamma(x,t)}} +\mu_k\right)T_1(G_l(u_k))
d x d t .
    \end{array}
\end{equation*}
Note that the function $T_1(G_l(r))$ is increasing in $r$ and thus the nonlocal integral is non-negative. For the right side, as $u_k\geq0$ implies $G_l(u_k)\leq u_k$ and $T_1$ is an increasing function, we can estimate similarly like \cref{imp2} to get
\begin{equation*}{\label{imp6}}
\begin{array}{rcl}
    \int_\Omega L(u_k(x,T))dx
    +\iint_{\Omega_T\cap \{l\leq u_k\leq l+1\}} |\nabla u_k|^2d x d t&\leq& C+\int_\Omega L(u_k(x,0))\\&\leq& C+\int_\Omega\int_0^{u_0}T_1(G_l(\theta))d\theta\leq C+\|u_0\|_{L^1(\Omega)},
\end{array}
\end{equation*}
where \begin{equation}{\label{GLL}} L(r)=\int_0^rT_1(G_l(\theta))d\theta.\end{equation} So using \cref{imp4}, we finally get 
\begin{equation*}{\label{imp7}}
    \iint_{B_l} |\nabla u_k|^2d x d t\leq C,
\end{equation*} where $B_l=\Omega_T\cap \{l\leq u_k\leq l+1\}$ and $C$ is a constant independent of $k$ and $l$.
We now follow [\citealp{2}, Theorem 4]. Let $1<q<2-\frac{n}{n+1}, r=\frac{n+1}{n}q$. We have
\begin{equation*}
\iint_{B_l}|\nabla u_k|^q \leq c(\operatorname{meas} B_l)^{1-q / 2} \leq c\left(\iint_{B_l}|u_k|^r\right)^{(2-q) / 2} \frac{1}{l^{r(2-q) / 2}},
\end{equation*} so that by H\"older inequality
\begin{equation}{\label{imp8}}{\begin{array}{rcl}
     \iint_{\Omega_T}|\nabla u_k|^q & \leq& \iint_{\Omega_T}|\nabla T_1(u_k)|^q dxdt+c\sum_{l=1}^\infty\left(\iint_{B_l}|u_k|^r\right)^{(2-q) / 2} \frac{1}{l^{r(2-q) / 2}}\\
& \stackrel{\cref{imp4}}{\leq}& C+c\left(\iint_{\Omega_T}|u_k|^r\right)^{(2-q) / 2}\left(\sum_{l=1}^{\infty} \frac{1}{l^{r(2-q) / q}}\right)^{q / 2} .
\end{array}}
\end{equation}
Applying Hölder's inequality yields
\begin{equation*}
    \|u_k\|_{L^{r}(\Omega)} \leq\|u_k\|_{L^1(\Omega)}^\theta\|u_k\|_{L^{q^*}(\Omega)}^{1-\theta} \stackrel{\cref{imp4}}{\leq} c\|u_k\|_{L^{q^*}(\Omega)}^{1-\theta},
\end{equation*}
where $1-\theta=\left((r-1) /(q^*-1)\right) \cdot\left(q^* / r\right)$.
The above inequality leads to
\begin{equation*}
    \|u_k\|_{L^{r}(0, T: L^{r}(\Omega))}^{r} \leq c \int_0^T\|u_k\|_{L^{q^*}(\Omega)}^{q^*(r-1) /(q^*-1)} =c\|u_k\|_{L^{q}(0, T: L^{q^*}(\Omega))}^q
\end{equation*}
if $r$ is such that $q^*\frac{r-1}{(q^*-1)}=q$, that is, $r=\frac{n+1}{n} q$.
Using \cref{imp8}, Sobolev embedding theorem now implies that

\begin{equation*}
\begin{array}{rcl}
    \|u_k\|_{L^{q}(0, T: L^{q^*}(\Omega))}^q= \int_0^T\left(\int_{\Omega}|u_k|^{q^*}\right)^{q / q^*} d t 
&\leq&  c\int_0^T\int_{\Omega} |\nabla u_k|^q d t 
\\&\leq&  C+c\|u_k\|_{ L^q(0,T;L^{q^*}(\Omega))}^{q(2-q) /2}  \times\left(\sum_{l=1}^{\infty} \frac{1}{l^{(n+1)(2-q) / n}}\right)^{q / 2}.\end{array}
\end{equation*}
From the previous bound on $q$ we have, using Young's inequality, the a priori estimate
\begin{equation*}{\label{imp9}}
\|u_k\|_{L^q(0, T; L^{q^*}(\Omega))} \leq c,    
\end{equation*}
and then the estimate
\begin{equation*}{\label{imp10}}
\|u_k\|_{L^q(0, T ; W_0^{1, q}(\Omega))} \leq C.  
\end{equation*}This along with \cref{imp4} concludes the proof for the case $\gamma^*=1$.
\smallskip\\\textbf{Case 2.} Now assume $\gamma^*>1$. For fix $l>0$ and $\tau\in(0,T]$, take $T_1(G_l(u_k))\chi_{(0,\tau)}$ as a test function in \cref{aproxproblem} to get
\begin{equation}{\label{imp11}}
    \begin{array}{l}
      \quad  \int_0^\tau\int_{\Omega}(u_k)_tT_1(G_l(u_k))d x d t +\iint_{\Omega_\tau\cap \{l\leq u_k\leq l+1\}} |\nabla u_k|^2d x d t\smallskip\\
+\underbrace{\int_0^\tau\int_{\mathbb{R}^n}\int_{\mathbb{R}^n} \frac{(u_k(x, t)-u_k(y, t))(T_1(G_l(u_k))(x,t)-T_1(G_l(u_k))(y,t))}{|x-y|^{n+2 s}} d x d y d t}_{\geq 0}\\
\leq\int_0^T\int_{\Omega} \left(\frac{\nu_k}{\left(u_k+\frac{1}{k}\right)^{\gamma(x,t)}} +\mu_k\right)T_1(G_l(u_k))
d x d t 
\leq  \|l^{-\gamma(\cdot)}\|_{L^\infty(\Omega_T)}\|\nu_k\|_{L^1(\Omega_T)}+\|\mu_k\|_{L^1(\Omega_T)},
    \end{array}
\end{equation}where $\Omega_\tau=\Omega\times(0,\tau)$. Note that in the last line, we have used the fact that each $u_k$ is non-negative, and hence from \cref{TKKK}, $T_1(G_l(u_k))\equiv 0$ if $u_k\leq l$. Now use \cref{GLL} and see that non-negativity of $u_k$ gives for each $\tau\in[0,T]$, \begin{equation*}
L(u_k(x,\tau))=\int_0^{u_k(x,\tau)}T_1(G_l(\theta))d\theta=\begin{cases}0&\text{ if }u_k(x,\tau)<l,\smallskip\\\frac{1}{2}(u_k-l)^2&\text{ if }l\leq u_k(x,\tau)\leq l+1\smallskip\\\frac{1}{2}+(u_k(x,\tau)-l-1)&\text{ if } u_k\geq l+1.\end{cases} 
\end{equation*}This clearly gives 
 $ L(u_k(x,\tau))\geq G_l(u_k)-1$, and we can estimate from \cref{imp11} as
 \begin{equation}{\label{imp12}}
 \begin{array}{c}
     \int_\Omega G_l(u_k(x,\tau))-|\Omega|\leq \int_\Omega L(u_k(x,\tau))+\iint_{\Omega_\tau\cap \{l\leq u_k\leq l+1\}} |\nabla u_k|^2d x d t\leq C(l)\\\implies \int_\Omega G_l(u_k(x,\tau))+\iint_{\Omega_\tau\cap \{l\leq u_k\leq l+1\}} |\nabla u_k|^2d x d t\leq C(l)+|\Omega|.\end{array}
 \end{equation}Note that the right side of \cref{imp12} is also bounded with respect to $\tau$; hence taking supremum on $\tau$ we obtain that
\begin{equation}{\label{imp13}}
\left\|G_l(u_k)\right\|_{L^{\infty}(0, T ; L^1(\Omega))} +\iint_{\Omega_T\cap \{l\leq u_k\leq l+1\}} |\nabla G_l(u_k)|^2d x d t\leq C (l),   
\end{equation}where $C(l)$ is a positive constant which depends on $l$ but is independent of $k$. Now for fix $l>0$, denote $v_k=G_l(u_k)$. Take $T_1(G_m(v_k))$ as a test function in \cref{aproxproblem} to get \begin{equation}{\label{imp131}}
    \begin{array}{l}
      \quad  \int_0^T\int_{\Omega}(u_k)_tT_1(G_m(v_k))d x d t +\iint_{\Omega_T\cap \{m\leq v_k\leq m+1\}} |\nabla v_k|^2d x d t\smallskip\\
 +{\int_0^T\int_{\mathbb{R}^n}\int_{\mathbb{R}^n} \frac{(u_k(x, t)-u_k(y, t))(T_1(G_m(v_k))(x,t)-T_1(G_m(v_k))(y,t))}{|x-y|^{n+2 s}} d x d y d t}\smallskip\\
= \int_0^T\int_{\Omega} \left(\frac{\nu_k}{\left(u_k+\frac{1}{k}\right)^{\gamma(x,t)}} +\mu_k\right)T_1(G_m(v_k))
d x d t .
    \end{array}
\end{equation} We estimate each terms of \cref{imp131}. First note that $v_k\equiv 0$ if $u_k\leq l$ and hence by \cref{TKKK}, $T_1(G_m(v_k))\equiv 0$ if $u_k\leq l$. Therefore $u_k\geq l$ for all $k$ in $\operatorname{supp}T_1(G_m(v_k))$ for each $m>0$. So one can estimate as
\begin{equation*}{\label{imp14}}
    \int_0^T\int_{\Omega} \left(\frac{\nu_k}{\left(u_k+\frac{1}{k}\right)^{\gamma(x,t)}} +\mu_k\right)T_1(G_m(v_k))
d x d t \leq \|l^{-\gamma(\cdot)}\|_{L^\infty(\Omega_T)}\|\nu_k\|_{L^1(\Omega_T)}+\|\mu_k\|_{L^1(\Omega_T)}.
\end{equation*} As the function $T_1(G_m(r))$ is increasing in $r$ and $u_k(x,t)\geq u_k(y,t)$ implies $v_k(x,t)=G_l(u_k(x,t))\geq G_l(u_k(y,t))=v_k(y,t)$, so the nonlocal term is non-negative. Finally, denoting $\tilde{L}(r)=\int_0^rT_1(G_m(G_l(\theta)))d\theta,$ we have
\begin{equation*}
    \int_0^T\int_{\Omega}(u_k)_tT_1(G_m(v_k))d x d t =\int_\Omega\tilde{L}(u_k(x,T))-\int_\Omega\tilde{L}(u_k(x,0))\geq -\int_\Omega\tilde{L}(u_k(x,0)).
\end{equation*} Thus from \cref{imp131} using \cref{imp4}, we obtain \begin{equation*}
    \iint_{\Omega_T\cap \{m\leq v_k\leq m+1\}} |\nabla v_k|^2d x d t\leq C(l)+\int_\Omega\tilde{L}(u_k(x,0))\leq C(l)+\int_\Omega u_0dx\leq C(l)+\|u_0\|_{L^1(\Omega)},
\end{equation*} where $C(l)$ is a constant which depends on $l$ (fixed after choice) but does not depend on $k$ and $m$. Denoting now $B_m=\Omega_T\cap \{m\leq v_k\leq m+1\}=\{l+m\leq u_k\leq l+m+1\}$, we have for $1<q<2-\frac{n}{n+1}, r=\frac{n+1}{n}q$,
\begin{equation*}
\iint_{B_m}|\nabla v_k|^q \leq C(l)(\operatorname{meas} B_m)^{1-q / 2} \leq C(l)\left(\iint_{B_m}|v_k|^r\right)^{(2-q) / 2} \frac{1}{m^{r(2-q) / 2}},
\end{equation*} so that 
\begin{equation}{\label{imp15}}{\begin{array}{rcl}
     \iint_{\Omega_T}|\nabla v_k|^q =\iint_{\Omega_T\cap \{u_k\geq l\}}|\nabla v_k|^q & \leq& \iint_{\Omega_T\cap \{l\leq u_k\leq l+1\}}|\nabla v_k|^q dxdt+c\sum_{m=1}^\infty\left(\iint_{B_m}|v_k|^r\right)^{(2-q) / 2} \frac{1}{m^{r(2-q) / 2}}\\
& \stackrel{\cref{imp13}}{\leq}& C(l)+C(l)\left(\iint_{\Omega_T}|v_k|^r\right)^{(2-q) / 2}\left(\sum_{m=1}^{\infty} \frac{1}{m^{r(2-q) / q}}\right)^{q / 2} .
\end{array}}
\end{equation} Equation \cref{imp15} is indeed same as \cref{imp8} with $u_k$ replaced by $v_k$. One can follow similar steps like the previous case from now on and use \cref{imp13}, to deduce 
\begin{equation}{\label{imp16}}
\|G_l(u_k)\|_{L^q(0, T; W^{1,q}_0(\Omega))} \leq C(l),    
\end{equation} for all $1<q<2-\frac{n}{n+1}$ and for each $l>0$ fixed.\smallskip\\ To get the estimates for $T_l(u_k)$, fix $l>0$ and take $T_l^{\gamma^*}(u_k) \chi_{(0, \tau)}$ as a test function in \cref{aproxproblem}, and then estimate exactly same like \cref{imp1,imp2}, to get\begin{equation*}{\label{imp161}}
\begin{array}{l}
\quad\int_0^{\tau}\int_\Omega(\tilde{T}_{l, \gamma^*}(u_k))_t +\frac{4\gamma^*}{(\gamma^*+1)^2}\int_0^\tau\int_\Omega |\nabla T_l^{\frac{\gamma^*+1}{2}}(u_k)|^2d x d t  \\\leq l^{\gamma^*}\|\mu_k\|_{L^1(\Omega_T)}+(1+l^{\gamma^*})\|\nu_k\|_{L^1(\Omega_T)}+l^{\gamma^*}\left\|C_{(\omega_{{T}})_\delta}^{-\gamma(\cdot)}\right\|_{L^{\infty}(\Omega_T)}\|\nu_k\|_{L^1(\Omega_T)}\leq C,
\end{array}
\end{equation*}
where $\tilde{T}_{l, \gamma^*}(r)$ is defined in \cref{T_k}. Furthermore, observing that $\tilde{T}_{l, \eta}(r) \geq C T_l(r)^{\eta+1}$, one gets
\begin{equation}{\label{imp17}}
    C\int_{\Omega} T^{\gamma^*+1}_l(u_k(x, \tau)) +\frac{4\gamma^*}{(\gamma^*+1)^2}\int_0^\tau\int_\Omega |\nabla T_l^{\frac{\gamma^*+1}{2}}(u_k)|^2d x d t\leq C(l)+\int_{\Omega} \tilde{T}_{l, \gamma^*}(u_0)\leq C(l)+l^{\gamma^*}\int_\Omega u_0dx.
\end{equation}
Every term on the right hand side of \cref{imp17} is also bounded with respect to $\tau$; hence taking the supremum on $\tau$ we obtain that
\begin{equation}{\label{imp18}}
\left\|T_l^{\gamma^*+1}(u_k)\right\|_{L^{\infty}(0, T ; L^2(\Omega))} +\int_0^T\int_\Omega |\nabla T_l^{\frac{\gamma^*+1}{2}}(u_k)|^2d x d t\leq C   (l) ,
\end{equation} for each $l>0$. Further, as $\nu$ is non-singular with respect to the Lebesgue measure, then 
    for each $t_0>0$ and $\omega\subset\subset\Omega$, $\exists C(\omega,t_0,n,s)$, a constant which does not depend on $k$,  such that $\forall k$, $u_k\geq C(\omega,t_0,n,s)$ in $\omega\times[t_0,T)$. Therefore one gets by \cref{imp18}
    \begin{equation}{\label{imp19}}
      C_{(\omega,t_0)}^{\gamma^*-1} \int_{t_0}^T\int_\omega |\nabla T_l(u_k)|^2d x d t \leq \int_{t_0}^T\int_\omega u_k^{\gamma^*-1}|\nabla T_l(u_k)|^2d x d t= \int_{t_0}^T\int_\omega  T_l^{\gamma^*-1}(u_k)|\nabla T_l(u_k)|^2d x d t \leq C(l).
    \end{equation} \cref{imp19} implies that the sequence $\{T_l(u_k)\}_{k \in \mathbb{N}}$ is uniformly bounded in $L^2(t_0,T;W_{\mathrm{loc}}^{1,2}(\Omega))$ for each $l>0$ and $0<t_0<T$. Finally note that $u_k=T_l(u_k)+G_l(u_k)$ for each fix $l>0$ and use \cref{imp4}, \cref{imp16}, \cref{imp19} to deduce that the sequence $\{u_k\}_{k \in \mathbb{N}}$ is uniformly bounded in $L^q(t_0,T;W_{\mathrm{loc}}^{1, q}(\Omega)) \cap L^\infty(0,T;L^1(\Omega))$ for every $1<q<2-\frac{n}{n+1}$ and $0<t_0<T$. Moreover, the sequences $\{G_l(u_k)\}_{k \in \mathbb{N}}$ and $\{T_l^{\frac{\gamma^*+1}{2}}(u_k)\}_{k \in \mathbb{N}}$ are uniformly bounded in $L^q(0,T;W_0^{1, q}(\Omega))$ and in $L^2(0,T;W_0^{1, 2}(\Omega))$ respectively for every fixed $l>0$ and $1<q<2-\frac{n}{n+1}$.
\end{proof}
\textbf{Preliminary for the proof of \cref{exis2}}
\begin{lemma}{\label{preli2}}
    Let  $\gamma(x,t)\equiv \gamma$ be a constant, $u_0\in L^{\gamma+1}(\Omega)$ and define $$q=\begin{cases}
    2, \text{ if } \gamma\geq 1\\ \frac{(\gamma+1)(n+2)}{(n+\gamma+1)}\text{ if } \gamma< 1.\end{cases}$$  Let $\nu$ be a non-negative bounded Radon measure, which is non-singular with respect to the Lebesgue measure. Suppose either $\mu\in L^{\gamma+1}\left(0,T;L^{\frac{n(\gamma+1)}{n+2\gamma}}\Omega)\right)$ or $\mu\in L^r(\Omega_T)$ where $r=\frac{(n+2)(\gamma+1)}{n+2(\gamma+1)}$ is a non-negative function in $\Omega_T$. Let $k \in \mathbb{N}$, $\mu_k=T_k(\mu)$ and assume that the solution of \cref{aproxproblem} obtained in \cref{exisforaprox} is denoted by $u_k$. If the sequence of non-negative functions $\{\nu_k\}_{k \in \mathbb{N}} \subset L^\infty(\Omega_T)$ are uniformly bounded in $L^1(\Omega_T)$, then the following conclusions hold:
\begin{enumerate}
    \item If $\gamma< 1$, then the sequence $\{u_k\}_{k \in \mathbb{N}}$ is bounded in $L^q(0,T;W_0^{1, q}(\Omega))\cap L^\infty(0,T;L^1(\Omega))$.
    \item If $\gamma\geq 1$, then the sequence $\{u^{\frac{\gamma+1}{2}}_k\}_{k \in \mathbb{N}}$ is uniformly bounded in $L^q(0,T;W_0^{1,q}(\Omega))\cap L^\infty(0,T;L^{2}(\Omega))$. Further, the sequence $\{u_k\}_{k \in \mathbb{N}}$ is bounded in $ L^q(t_0,T;W_{\mathrm{loc}}^{1,q}(\Omega))\cap L^\infty(0,T; L^{\gamma+1}(\Omega))$ for all $0<t_0<T$. 
\end{enumerate}  
\end{lemma}
\begin{proof}
    \textbf{Case 1.} We first assume $\gamma<1$. As $u_k\in  L^2(0,T; W^{1,2}_0(\Omega))$, so for fix $\tau\in(0,T]$ we choose $T_1
    (u_k) \chi_{(0, \tau)}$ as a test function in \cref{aproxproblem}, to get
\begin{equation*}{\label{impo1}}
    \begin{array}{l}
      \quad  \int_0^\tau \int_{\Omega}(u_k)_tT_1
      (u_k)d x d t +
      \int_0^\tau\int_\Omega |\nabla T_1(u_k)|^2
      d x d t\smallskip\\
 +\underbrace{\int_0^\tau \int_{\mathbb{R}^n}\int_{\mathbb{R}^n} \frac{(u_k(x, t)-u_k(y, t))(T_1
 (u_k)(x,t)-T_1
 (u_k)(y,t))}{|x-y|^{n+2 s}} d x d y d t}_{\geq 0}\smallskip\\
= \int_0^\tau\int_{\Omega} \left(\frac{\nu_k}{\left(u_k+\frac{1}{k}\right)^{\gamma}} +\mu_k\right)T_1(u_k)
\leq \iint_{\Omega_T} \frac{\nu_kT_1^\gamma(u_k)}{\left(u_k+\frac{1}{k}\right)^{\gamma}} +\|\mu_k\|_{L^1(\Omega_T)} \leq \|\nu_k\|_{L^1(\Omega_T)}+\|\mu_k\|_{L^1(\Omega_T)}.
    \end{array}
\end{equation*}
Note that $T_1
(r)$ is an increasing function in $r$ and hence the nonlocal term is non-negative. It holds
\begin{equation*}{\label{impo2}}
\begin{array}{l}
\quad\int_0^{\tau}\int_\Omega(\tilde{T}_{1,1
}(u_k))_t +
\int_0^\tau\int_\Omega |\nabla T_1
(u_k)|^2
d x d t\leq \|\nu_k\|_{L^1(\Omega_T)}+\|\mu_k\|_{L^1(\Omega_T)},
\end{array}
\end{equation*}
where $\tilde{T}_{1, \eta}(r)$ is defined in \cref{T_k}. Finally, observing that $\tilde{T}_{1, \eta}(r) \geq r-1$, one gets
\begin{equation}{\label{impo3}}
    \int_{\Omega} u_k(x, \tau) +
    \int_0^\tau\int_\Omega |\nabla T_1
    (u_k)|^2d x d t\leq C+|\Omega|+\int_{\Omega} \tilde{T}_{1, 1
    }(u_0)\leq C+\int_\Omega u_0dx.
\end{equation}
The right side of \cref{impo3} is also bounded with respect to $\tau$; hence taking the supremum on $\tau$ we obtain that
\begin{equation}{\label{impo4}}
\left\|u_k\right\|_{L^{\infty}(0, T ; L^1(\Omega))} +\int_0^T\int_\Omega |\nabla T_1
(u_k)|^2d x d t\leq C    .
\end{equation}
Now for $0<\ep<1/k$ and $\tau\in(0,T]$ fixed, take $((u_k+\ep)^\gamma-\ep^\gamma)\chi_{(0,\tau)}$ as a test function in \cref{aproxproblem}, to get 
\begin{equation}{\label{impor1}}
    \begin{array}{l}
      \quad  \int_0^\tau \int_{\Omega}(u_k)_t((u_k+\ep)^\gamma-\ep^\gamma)d x d t +\frac{4\gamma}{(\gamma+1)^2}\int_0^\tau\int_\Omega |\nabla (u_k+\ep)^{\frac{\gamma+1}{2}}|^2
      d x d t\smallskip\\
 +\underbrace{\int_0^\tau \int_{\mathbb{R}^n}\int_{\mathbb{R}^n} \frac{(u_k(x, t)-u_k(y, t))((u_k+\ep)^\gamma(x,t)-(u_k+\ep)^\gamma(y,t))}{|x-y|^{n+2 s}} d x d y d t}_{\geq 0}\smallskip\\
= \int_0^\tau\int_{\Omega} \left(\frac{\nu_k}{\left(u_k+\frac{1}{k}\right)^{\gamma}} +\mu_k\right)((u_k+\ep)^\gamma-\ep^\gamma)
d x d t \leq \|\nu_k\|_{L^1(\Omega_T)}+\int_0^T\int_\Omega \mu (u_k+\ep)^\gamma dx dt.
    \end{array}
\end{equation}Define for $0<r \in \mathbb{R}$,
$\psi_\ep(r)=\int_0^r\left[(
v+\ep)^{\gamma}-\ep^\gamma\right]  dv  $.
As $\gamma<1$, it results
 $   \frac{(r+\ep)^{\gamma+1}-\ep^{\gamma+1}}{\gamma+1}-r\leq \psi_\ep(r)\leq \frac{r^{\gamma+1}}{\gamma+1}.$
As $u_k\in C([0,T];L^2(\Omega))$ for each $k$, from \cref{impor1}, we thus obtain
\begin{equation}{\label{impor2}}
    \begin{array}{l}
    \int_\Omega \psi(u_k(x,\tau))dx+\frac{4\gamma}{(\gamma+1)^2}\int_0^\tau\int_\Omega |\nabla (u_k+\ep)^{\frac{\gamma+1}{2}}|^2
      d x d t\leq C+\int_0^T\int_\Omega \mu (u_k+\ep)^\gamma dx dt + \int_\Omega \psi(u_k(x,0))dx\smallskip\\
      \implies \frac{1}{\gamma+1} \int_{\Omega}(u_k+\ep)^{\gamma+1}(x,\tau) d x +\frac{4\gamma}{(\gamma+1)^2}\int_0^\tau\int_\Omega |\nabla (u_k+\ep)^{\frac{\gamma+1}{2}}|^2
      d x d t\smallskip\\\qquad\leq C+\int_0^T\int_\Omega \mu (u_k+\ep)^\gamma+\frac{1}{\gamma+1}\int_\Omega u_0^{\gamma+1}dx +\frac{\ep^{\gamma+1}}{\gamma+1}|\Omega|+ \int_{\Omega}u_k(x,\tau) d x \smallskip\\\qquad{\leq} C+\int_0^T\int_\Omega \mu (u_k+\ep)^{\gamma}+C\|u_0\|_{L^{\gamma+1}(\Omega)}+\sup _{0 \leq t \leq T} \int_{\Omega}\left|u_k(x, t)\right| d x .
    \end{array}
\end{equation}Using \cref{impo4}, 
taking supremum on $\tau$ in \cref{impor2}, we obtain that
\begin{equation}{\label{impor6}}
\sup _{0 \leq \tau \leq T} \int_{\Omega}\left|(u_k+\ep)(x, \tau)\right|^{{\gamma}+1} d x +\int_0^T\int_\Omega |\nabla( u_k+\ep)^{\frac{\gamma+1}{2}}|^2d x d t\leq C  +  \int_0^T\int_\Omega \mu (u_k+\ep)^{\gamma}.
\end{equation}
\textbf{Case 1A: $\mu \in L^{\gamma+1}\left(0, T ; L^{\left(\frac{n(\gamma+1)}{n+2\gamma}\right)}(\Omega)\right)$}\\By Hölder inequality in space variable, then Sobolev inequality and again applying of H\"older inequality we get
\begin{equation*}
    \begin{array}{rcl}
         \int_0^T\int_{\Omega} \mu (u_k+\ep)^{\gamma} d x d t&=& \int_0^T\int_{\Omega} \mu \left((u_k+\ep)^{\frac{\gamma+1}{2}}\right)^{\frac{2\gamma}{\gamma+1}} d x d t\\&\leq &C\int_0^T\left(\int_{\Omega}|\mu(x, t)|^{\frac{n(\gamma+1)}{n+2\gamma}} d x\right)^{\frac{n+2\gamma}{n(\gamma+1)}}\left(\int_{\Omega}|(u_k+\ep)^{\frac{\gamma+1}{2}}(x, t)|^{2^*} d x\right)^{\frac{(n-2)\gamma}{n(\gamma+1)}} d t
        \\&\leq&C\int_0^T\|\mu(\cdot,t)\|_{L^{\frac{n(\gamma+1)}{n+2\gamma}}{(\Omega)}}\left(\int_{\Omega}\left|(u_k+\ep)\right|^{{\gamma}+1} d x+\int_{\Omega}|\nabla (u_k+\ep)^{\frac{\gamma+1}{2}}|^{2} d x\right)^{\frac{\gamma}{\gamma+1}} d t\smallskip\\&\leq& C\left(\int_0^T\|\mu(\cdot,t)\|_{L^{\frac{n(\gamma+1)}{n+2\gamma}}{(\Omega)}}^{\gamma+1} d t\right)^{\frac{1}{\gamma+1}}\times\left(\iint_{\Omega_T}\left|(u_k+\ep)\right|^{{\gamma}+1} +|\nabla (u_k+\ep)^{\frac{\gamma+1}{2}}|^{2} \right)^{\frac{\gamma}{\gamma+1}}\smallskip\\&\leq& C\left(\int_0^T\|\mu(\cdot,t)\|_{L^{\frac{n(\gamma+1)}{n+2\gamma}}{(\Omega)}}^{\gamma+1} d t\right)^{\frac{1}{\gamma+1}}\smallskip\\&&\times\left(T\sup _{0 \leq t \leq T}\int_\Omega\left|(u_k+\ep)(x,t)\right|^{{\gamma}+1}dx +\iint_{\Omega_T}|\nabla (u_k+\ep)^{\frac{\gamma+1}{2}}|^{2} dxdt\right)^{\frac{\gamma}{\gamma+1}}.
    \end{array}
\end{equation*}
Now since $\gamma<\gamma+1$, so using Young's inequality in the above estimate, we get from \cref{impor6} that
\begin{equation*}
    \begin{array}{l}
  \quad\sup _{0 \leq t \leq T} \int_{\Omega}\left|(u_k+\ep)(x, t)\right|^{{\gamma}+1} d x +\int_0^T\int_\Omega |\nabla( u_k+\ep)^{\frac{\gamma+1}{2}}|^2d x d t
 \smallskip\\\leq C\left\|\mu\right\|_{L^{\gamma+1}\left(0, T ; L^{\frac{n(\gamma+1)}{n+2\gamma}}(\Omega)\right)}\times\left(\varepsilon\left(\sup _{0 \leq t \leq T}\int_\Omega\left|(u_k+\ep)(x,t)\right|^{{\gamma}+1}dx+\int_0^T\int_\Omega |\nabla u_k^{\frac{{\gamma}+1}{2}}|^2d x d t\right)+C(\varepsilon)\right)+C_2.
 \end{array}
\end{equation*}
Choosing $\varepsilon$ small enough such that $
    \varepsilon C\left\|\mu\right\|_{L^{\gamma+1}\left(0, T ; L^{\frac{n(\gamma+1)}{n+2\gamma}}(\Omega)\right)}=1,
$ we get
\begin{equation}{\label{impor7}}
 \sup _{0 \leq t \leq T} \int_{\Omega}\left|(u_k+\ep)(x, t)\right|^{{\gamma}+1} d x +\int_0^T\int_\Omega |\nabla( u_k+\ep)^{\frac{\gamma+1}{2}}|^2d x d t\leq C.
   \end{equation}
\smallskip\\\textbf{Case 1B: $\mu\in L^r(\Omega_T)$, $r=\frac{(n+2)(\gamma+1)}{n+2(\gamma+1)}$ }\\ For this case, we apply Hölder inequality with exponents ${r}$ and ${r}^\prime$ to get
\begin{equation}{\label{fusobolevsecond11}}
    \begin{array}{rcl}
         \int_0^T\int_{\Omega} \mu (u_k+\ep)^{\gamma} d x d t\leq C\|\mu\|_{L^{{r}}\left(\Omega_T\right)}\left[\iint_{\Omega_T} |(u_k+\ep)|^{\gamma{r}^{\prime}} d x d t\right]^{\frac{1}{
         {r}^{\prime}}}.
         \end{array}
         \end{equation}
We note that $\gamma{r}^\prime=\frac{(n+2)(\gamma+1)}{n}$. Using the Hölder inequality with exponent $\frac{n}{n-2 }$ and $\frac{n}{2 }$ we get
\begin{equation*}
    \begin{array}{rcl}
\iint_{\Omega_T}\left|(u_k+\ep)\right|^{\frac{({\gamma}+1)(n+2 )}{n}} d x d t&= &\iint_{\Omega _T}\left|(u_k+\ep)\right|^{2\frac{{\gamma}+1}{2}}\left|(u_k+\ep)\right|^{\frac{4 }{n} \frac{{\gamma}+1}{2}} d x d t \smallskip\\
 &\leq& \int_0^T\left(\int_{\Omega}\left|(u_k+\ep)(x, t)\right|^{2^* \frac{{\gamma}+1}{2}} d x\right)^{\frac{n-2}{n}}\left(\int_{\Omega}\left|(u_k+\ep)(x, t)\right|^{{\gamma}+1} d x\right)^{\frac{2 }{n}} d t \smallskip\\
 &=&\int_0^T\|(u_k+\ep)^{\frac{{\gamma}+1}{2}}\|_{L^{2 ^*}(\Omega)}^2\left(\int_{\Omega}\left|(u_k+\ep)(x, t)\right|^{{\gamma}+1} d x\right)^{\frac{2 }{n}} d t .
\end{array}
\end{equation*}
Now taking supremum over $t\in [0, T]$ and applying the Sobolev embedding as that of \cref{Sobolev embedding}, we can write
\begin{equation}{\label{veryimp2}}
\begin{array}{rcl}
\iint_{\Omega_T}\left|(u_k+\ep)\right|^{\frac{({\gamma}+1)(n+2 )}{n}} d x d t  &\leq& C(n)\left(\sup _{0 \leq t \leq T} \int_{\Omega}\left|(u_k+\ep)(x, t)\right|^{{\gamma}+1} d x\right)^{\frac{2 }{n}} \smallskip\\&&\times\left(\iint_{\Omega_T}\left( (u_k+\ep)^{\gamma+1}+ |\nabla (u_k+\ep)^{\frac{{\gamma}+1}{2}}|^2 \right)d x d t \right)\smallskip\\&\leq& C(n)\left(\sup _{0 \leq t \leq T} \int_{\Omega}\left|(u_k+\ep)(x, t)\right|^{{\gamma}+1} d x\right)^{\frac{2 }{n}} \smallskip\\&&\times\left(T\sup _{0 \leq t \leq T}\int_\Omega |(u_k+\ep)(x,t)|^{\gamma+1}+\iint_{\Omega_T} |\nabla (u_k+\ep)^{\frac{{\gamma}+1}{2}}|^2 d x d t \right).\end{array}
\end{equation}
Using \cref{impor6} and \cref{fusobolevsecond11}, from the above inequality we get by convexity argument
\begin{equation*}
\begin{array}{rcl}
\iint_{\Omega_T}\left|(u_k+\ep)\right| ^{\frac{({\gamma}+1)(n+2 )}{n}} d x d t &
 \leq &C(n)\left(C_1\iint_{\Omega_T} \mu (u_k+\ep)^{{\gamma}} d x d t+C_2\right)^{\frac{n+2 }{n}} \\
&\leq& C(n) 2^{\frac{2 }{n}}\left(\left(C_1\iint_{\Omega_T} \mu (u_k+\ep)^{{\gamma}} d x d t\right)^{\frac{n+2 }{n}}+C_2^{\frac{n+2}{n}}\right)\\&\leq&C(n) 2^{\frac{2 }{n}}\left(\left(C_1C\|\mu\|_{L^{{r}}(\Omega_T)}\right)^{\frac{n+2 }{n}}\left(\iint_{\Omega_T} |(u_k+\ep)|^{\gamma{r}^{\prime}} d x d t\right)^{\frac{n+2 }{n{r}^\prime}}+C_2^{\frac{n+2}{n}}\right).
\end{array}
\end{equation*}
Now as $\frac{n+2}{n{r}^\prime}=\frac{\gamma}{\gamma+1}<1$, therefore using Young's inequality we get
\begin{equation}{\label{veryimp}}
\iint_{\Omega_T}\left|(u_k+\ep)\right|^{\gamma{r}^\prime}d x d t=\iint_{\Omega_T}\left|(u_k+\ep)\right| ^{\frac{({\gamma}+1)(n+2 )}{n}} d x d t \leq C,
\end{equation}
where $C$ is a positive constant independent of $k$. The above boundedness along with \cref{impor6} and \cref{fusobolevsecond11} gives 
\begin{equation}{\label{impor9}}
\sup _{0 \leq t \leq T} \int_{\Omega}\left|(u_k+\ep)(x, t)\right|^{{\gamma}+1} d x +\int_0^T\int_\Omega |\nabla( u_k+\ep)^{\frac{\gamma+1}{2}}|^2d x d t\leq C  .
\end{equation}Let $1<p<2$ will be specified later. By H\"older inequality, we have
\begin{equation}{\label{3.15}}
    \begin{array}{l}
        \quad \iint_{\Omega_T} |\nabla u_k|^{p} d x d t= \iint_{\Omega_T} \frac{|\nabla u_k|^{p}(u_k+\ep)^{\frac{(1-\gamma)p}{2}}}{(u_k+\ep)^{\frac{(1-\gamma)p}{2}}} d x d t\smallskip\\\leq \left(\iint_{\Omega_T} \frac{|\nabla u_k|^{2}}{(u_k+\ep)^{{(1-\gamma)}}} d x d t\right)^{\frac{p}{2}}\times\left(\iint_{\Omega_T} (u_k+\ep)^{\frac{(1-\gamma)p}{2-p}}d x d t\right)^{\frac{2-p}{2}}\\\stackrel{\cref{impor7},\cref{impor9}}{\leq}C\left(\iint_{\Omega_T} (u_k+\ep)^{\frac{(1-\gamma)p}{2-p}} d x d t\right)^{\frac{2-p}{2}}.
    \end{array}
\end{equation}
We now choose $p$ to be such that
\begin{equation*}
\frac{(1-\gamma)p}{2-p}=\frac{(\gamma+1)(n+2)}{n}
\text{ i.e. } p=\frac{(\gamma+1)(n+2 )}{n+1+\gamma}=q.
\end{equation*}
Note that $1<q<2$. Note also that the way \cref{veryimp2} is established does not use the fact that $\mu\in L^r(\Omega_T)$ of case $1B.$ Therefore \cref{impor7} can be used to deduce for case $1A$, the estimate 
\begin{equation}{\label{veryimp3}}
    \iint_{\Omega_T}\left|(u_k+\ep)\right| ^{\frac{({\gamma}+1)(n+2 )}{n}} d x d t \leq C.
\end{equation}Then we get by \cref{veryimp,3.15,veryimp3} along with the choice of $p$,
\begin{equation*}
\begin{array}{rcl}
   \iint_{\Omega_T} |\nabla u_k|^{q} d x d t\leq C,
    \end{array}
\end{equation*}
where $C$ is a positive constant independent of $k$. Thus merging \cref{impo4}, the sequence $\left\{u_k\right\}_k$ is uniformly bounded in $
L^{q}(0, T ; W_0^{1, q}(\Omega))\cap L^\infty(0,T;L^1(\Omega))$.\smallskip\\
\textbf{Case 2.} Consider now $\gamma\geq 1$. As each $u_k$ is bounded, take $u_k^\gamma\chi(0,\tau)$, $\tau\in(0,T]$ as a test function in \cref{aproxproblem} to get \begin{equation}{\label{impo5}}
    \begin{array}{l}
      \quad  \int_0^\tau \int_{\Omega}(u_k)_tu_k^\gamma d x d t +\gamma\int_0^\tau\int_\Omega |\nabla u_k|^2u_k^{\gamma-1}
      d x d t\smallskip\\
 +\underbrace{\int_0^\tau \int_{\mathbb{R}^n}\int_{\mathbb{R}^n} \frac{(u_k(x, t)-u_k(y, t))(u_k^\gamma(x,t)-u_k^\gamma(y,t))}{|x-y|^{n+2 s}} d x d y d t}_{\geq 0}\smallskip\\
= \int_0^\tau\int_{\Omega} \left(\frac{\nu_k}{\left(u_k+\frac{1}{k}\right)^{\gamma}} +\mu_k\right)u_k^\gamma
d x d t \leq \|\nu_k\|_{L^1(\Omega_T)}+\int_0^T\int_\Omega \mu u^\gamma_kdxdt.
    \end{array}
\end{equation}\cref{impo5} can be further estimated as \begin{equation}{\label{impo67}}
    \begin{array}{rcl}
      \quad  \frac{1}{\gamma+1} \int_{\Omega}u_k^{\gamma+1}(x,\tau) d x +\gamma\int_0^\tau\int_\Omega |\nabla u_k|^2u_k^{\gamma-1}
      d x d t&\leq& C+\int_0^T\int_\Omega \mu u^\gamma_k+\frac{1}{\gamma+1} \int_{\Omega}u_k^{\gamma+1}(x,0) d x \smallskip\\&\leq &C+\int_0^T\int_\Omega \mu u^\gamma_k+C\|u_0\|_{L^{\gamma+1}(\Omega)}.
    \end{array}
\end{equation}Every term on the right hand side of \cref{impo67} is also bounded with respect to $\tau$; hence taking the supremum on $\tau$ we obtain that
\begin{equation}{\label{impo6}}
\sup _{0 \leq \tau \leq T} \int_{\Omega}\left|u_k(x, \tau)\right|^{{\gamma}+1} d x +\int_0^T\int_\Omega |\nabla u_k^{\frac{\gamma+1}{2}}|^2d x d t\leq C  +  \int_0^T\int_\Omega \mu u^\gamma_k.
\end{equation}Equation \cref{impo6} is the same as that of \cref{impor6} except $(u_k+\ep) $ is replaced with $u_k$. One can now follow exactly same steps as that of the previous case (steps to deduce \cref{impor7,impor9}) to get \begin{equation*}{\label{impor64}}
\sup _{0 \leq t \leq T} \int_{\Omega}\left|u_k(x, t)\right|^{{\gamma}+1} d x +\int_0^T\int_\Omega |\nabla u_k^{\frac{\gamma+1}{2}}|^2d x d t\leq C  ,
\end{equation*}
where $C$ is a positive constant independent of $k$. The above boundedness gives that the sequence $\left\{u_k^{\frac{\gamma+1}{2}}\right\}_k$ is uniformly bounded in $
L^2(0, T ; W_0^{1,2}(\Omega)) \cap L^{\infty}(0, T ; L^{2}(\Omega))$ and $\left\{u_k\right\}_k$ is uniformly bounded in $
L^{\infty}(0, T ; L^{\gamma+1}(\Omega))$. 

Further, as $\nu$ is non-singular with respect to the Lebesgue measure, then 
    for each $t_0>0$ and $\omega\subset\subset\Omega$, $\exists C(\omega,t_0,n,s)$, a constant which does not depend on $k$,  such that $\forall k$, $u_k\geq C(\omega,t_0,n,s)$ in $\omega\times[t_0,T)$. Therefore 
    \begin{equation*}{\label{imp191}}
      C_{(\omega,t_0)}^{\gamma^*-1} \int_{t_0}^T\int_\omega |\nabla u_k|^2d x d t \leq \int_{t_0}^T\int_\omega u_k^{\gamma-1}|\nabla u_k|^2d x d t \leq C,
    \end{equation*} which implies that the sequence $\{u_k\}_{k \in \mathbb{N}}$ is bounded in $ L^2(t_0,T;W_{\mathrm{loc}}^{1,2}(\Omega))$ for all $0<t_0<T$.
\end{proof}
The estimates on $u_k$ imply that the singular term is locally bounded in $L^1(\Omega_T)$. Before going to it, let \begin{equation*}
    W_2(0,T):=\left\{v \in L^2(0,T ; W^{1,2}_0(\Omega)): v^{\prime} \in L^{2}(0,T ; W^{-1,2}(\Omega))\right\} \subset C([0,T]; L^2(\Omega)).
\end{equation*}
Then, by [\citealp{showalter2013monotone}, Corollary 1.1, Chapter III], for every $u, v \in W_2(0,T)$ the scalar product $t \mapsto(u(t), v(t))_{L^2(\Omega)}$ is an absolutely continuous function and there holds
\begin{equation}{\label{first}}
\frac{d}{d t}(u(t), v(t))_{L^2(\Omega)}=\left\langle u^{\prime}(t), v(t)\right\rangle+\left\langle v^{\prime}(t), u(t)\right\rangle, \quad \text { for a. e. } t \in (0,T).    
\end{equation}Now, we have
\begin{corollary}{\label{coRo}} Let $\nu$ be a non-negative bounded Radon measure in $\Omega_T$, which is non-singular with respect to the Lebesgue measure. Under the assumptions of \cref{preli1} and \cref{preli2} one has that
\begin{equation*}
    \iint_{\Omega_T} \frac{\nu_k}{\left(u_k+\frac{1}{k}\right)^{\gamma(x,t)}}\phi\,dxdt\leq C,
\end{equation*}for every non-negative $\phi \in C_c^\infty(\Omega_T)$, with $C$ not depending on $k$.\end{corollary}
\begin{proof}
Let us take $0 \leq \phi \in C_c^\infty(\Omega_T)$ as a test function in \cref{aproxproblem} obtaining
\begin{equation*}
\begin{array}{rcl}
     \iint_{\Omega_T} \frac{\nu_k}{\left(u_k+\frac{1}{k}\right)^{\gamma(x,t)}}\phi\,dxdt&\leq& \int_0^T\langle (u_k)^\prime(t),\phi(t)\rangle dt+\int_0^T\int_\Omega \nabla u_k\cdot \nabla \phi\smallskip\\&&+\int_0^T \int_{\mathbb{R}^n}\int_{\mathbb{R}^n} \frac{(u_k(x, t)-u_k(y, t))(\phi(x, t)-\phi(y, t))}{|x-y|^{n+2 s}} d x d y d t\smallskip\\&\stackrel{\cref{first}}{\leq}& \int_0^T \frac{d}{d t}(u_k(t), v(t))_{L^2(\Omega)}dt-\int_0^T\int_\Omega u_k\phi_t dxdt +\int_0^T\int_\Omega|\nabla u_k||\nabla \phi| dxdt\smallskip\\&&+\int_0^T \int_{\mathbb{R}^n}\int_{\mathbb{R}^n} \frac{(u_k(x, t)-u_k(y, t))(\phi(x, t)-\phi(y, t))}{|x-y|^{n+2 s}} d x d y d t.
     \end{array}
\end{equation*}For the nonlocal term, note that $\{u_k\}_k$ is uniformly bounded in $L^\infty(0,T;L^1(\Omega))$, and use \cref{preli1} and \cref{preli2}, to follow similar steps as that of \cref{fracsolu}, to get the boundedness. The local term is also $\leq  C\iint_{\operatorname{supp}\phi}|\nabla u_k|dxdt$ and hence is bounded by \cref{preli1} and \cref{preli2}. Thus we have \begin{equation*}
\begin{array}{rcl}
     \iint_{\Omega_T} \frac{\nu_k}{\left(u_k+\frac{1}{k}\right)^{\gamma(x,t)}}\phi\,dxdt&\leq& C+ \underbrace{\int_\Omega u_k(x,T)\phi(x,T)dx}_{=0}-\underbrace{\int_\Omega u_k(x,0)\phi(x,0)dx}_{=0}+c\int_0^T\int_\Omega |u_k| dxdt\smallskip\\&\leq& C+c\|u_k\|_{L^\infty(0,T;L^1(\Omega))}\leq C.
     \end{array}
\end{equation*}
This concludes the proof.
\end{proof}
\begin{lemma}{\label{convergenceinL1}}
 Let $\nu$ be a non-negative bounded Radon measure in $\Omega_T$, which is non-singular with respect to the Lebesgue measure. Under the assumptions of \cref{preli1} and \cref{preli2}, there exists $u \in L^\infty(0, T ;L^1(\Omega))$ ($u$ is weak limit of $\{u_k\}_k$ in the corresponding Sobolev spaces of \cref{preli1} and \cref{preli2}, where the sequence is bounded) such that, up to a subsequence, $\{u_k\}$ converges to $u$ a.e. on $\Omega_T$ and strongly in $L^1_{\mathrm{loc}}(0, T ; L_{\mathrm{loc}}^1(\Omega))$. 
\end{lemma}
\begin{proof} From \cref{preli1} and \cref{preli2} we know that $\{u_k\}$ is bounded in $L^q(t_0, T ; W_{\mathrm{loc }}^{1, q}(\Omega))$ for some $1<q\leq 2$. Moreover from \cref{coRo} one has that the right hand side of \cref{aproxproblem} is bounded in $L^1_{\mathrm{loc}}(0, T ; L_{\mathrm{loc }}^1(\Omega))$. Hence, for fix $t_1,t_2\in(0,T)$ and $\omega\subset\subset\Omega$, one can again follow the steps of the proof of \cref{fracsolu} and use the boundedness of $\{u_k\}_k$ in $L^\infty(0,T;L^1(\Omega))$ to deduce from the equation \cref{aproxproblem} that $\{(u_k)_t\}$ is bounded in $L^q(t_1, t_2 ; W^{-1, q}(\omega))+L^1(\omega\times[t_1,t_2])$. This is sufficient to apply [\citealp{conv}, Corollary 4] in order to deduce that $\{u_k\}_k$ converges to a function $u$ in $L^1_{\mathrm{loc}}(0, T ; L_{\mathrm{loc }}^1(\Omega))$ and then the almost everywhere convergence in $\Omega_T$. Finally by Fatou's lemma one gets $u\in L^\infty(0,T;L^1(\Omega))$.
\end{proof}
\textbf{Preliminaries for \cref{nonexistence}}\begin{lemma}{\label{preli3}}
    Let $\nu$ be a non-negative bounded Radon measure, $\gamma(x,t)\equiv \gamma\geq 1$ be a constant, $u_0=0$ and consider the approximated problem for each $k\in\mathbb{N}$ as
    \begin{equation}{\label{aproxproblemnon}}
    \begin{array}{c}
         (u_k)_t-\Delta u_k
         =\frac{\nu_k}{\left(u_k+\frac{1}{k}\right)^{\gamma}}  \text { in } \Omega_T, \smallskip\\u_k>0 \text{ in }\Omega_T,\quad u_k=0  \text { in }(\mathbb{R}^n \backslash \Omega) \times(0, T), \smallskip\\ u_k(x, 0)=0 \text { in } \Omega ;
    \end{array}
\end{equation}
where $0\leq \nu_k\in L^\infty(\Omega_T)$ is bounded in $L^1(\Omega_T)$ and converges in narrow topology to $\nu$. 
  Assume that the solutions of \cref{aproxproblemnon} 
  is denoted by $u_k$. Then the sequence $\{u^{\frac{\gamma+1}{2}}_k\}_{k \in \mathbb{N}}$ is uniformly bounded in $L^2(0,T;W_0^{1,2}(\Omega))\cap L^\infty(0,T;L^{2}(\Omega))$. Further, the sequence $\{u_k\}_{k \in \mathbb{N}}$ is bounded in $ L^2_{\mathrm{loc}}(0,T;W_{\mathrm{loc}}^{1,2}(\Omega))\cap L^\infty(0,T; L^{\gamma+1}(\Omega))
  $. Moreover, up to a subsequence, $\{u_k\}_{k\in\mathbb{N}}$ converges to $u$ strongly in $L^1_{\mathrm{loc}}(0, T ; L_{\mathrm{loc}}^1(\Omega))$ and a.e. on $\Omega_T$.
\end{lemma}
\begin{proof}
  By classical theory, there exists solution to \cref{aproxproblemnon} for each $k$. In fact simple changes in proofs of \cref{lemma1} and \cref{exisforaprox}, yields the existence of unique $u_k\in L^2(0,T:W^{1,2}_0(\Omega))\cap L^\infty(\Omega_T)\cap C([0,T];L^2(\Omega))$ such that $(u_k)_t\in L^2(0,T;W^{-1,2}(\Omega))$ and $u_k$ solves \cref{aproxproblemnon}. As each $u_k$ is bounded, take $u_k^\gamma\chi(0,\tau)$, $\tau\in(0,T]$ as a test function in \cref{aproxproblemnon} to get \begin{equation}{\label{impo55}}
    \begin{array}{l}
      \quad  \int_0^\tau \int_{\Omega}(u_k)_tu_k^\gamma d x d t +\gamma\int_0^\tau\int_\Omega |\nabla u_k|^2u_k^{\gamma-1}
      d x d t
= \int_0^\tau\int_{\Omega} \frac{\nu_k}{\left(u_k+\frac{1}{k}\right)^{\gamma}} u_k^\gamma
d x d t \leq \|\nu_k\|_{L^1(\Omega_T)}.
    \end{array}
\end{equation}As $u_k(\cdot,0)\equiv 0$, \cref{impo55} can be further estimated as \begin{equation}{\label{impo567}}
    \begin{array}{rcl}
      \quad  \frac{1}{\gamma+1} \int_{\Omega}u_k^{\gamma+1}(x,\tau) d x +\gamma\int_0^\tau\int_\Omega |\nabla u_k|^2u_k^{\gamma-1}
      d x d t&\leq& C.
    \end{array}
\end{equation} Taking supremum with respect to $\tau$ in \cref{impo567} we obtain that
\begin{equation*}{\label{impo56}}
\sup _{0 \leq \tau \leq T} \int_{\Omega}\left|u_k(x, \tau)\right|^{{\gamma}+1} d x +\int_0^T\int_\Omega |\nabla u_k^{\frac{\gamma+1}{2}}|^2d x d t\leq C  ,
\end{equation*}
where $C$ is a positive constant independent of $k$. The above boundedness gives that the sequence $\left\{u_k^{\frac{\gamma+1}{2}}\right\}_k$ is uniformly bounded in $
L^2(0, T ; W_0^{1,2}(\Omega)) \cap L^{\infty}(0, T ; L^{2}(\Omega))$ and $\left\{u_k\right\}_k$ is uniformly bounded in $
L^{\infty}(0, T ; L^{\gamma+1}(\Omega))$.\smallskip\\Now let $\phi \in C_c^\infty(\Omega_T)$ be non-negative and take $(u_k-1)\phi^2$ as a test function in \cref{aproxproblemnon} to get
 \begin{equation}{\label{non1}}
    \begin{array}{l}
      \quad  \int_0^T \langle(u_k)_t(u_k-1)\phi^2\rangle d t +\int_0^T\int_\Omega |\nabla u_k|^2\phi^2
      d x d t+2\int_0^T\int_\Omega (u_k-1)\phi\nabla u_k\cdot\nabla \phi \,dxdt
      
=\iint_{\Omega_T} \frac{\nu_k(u_k-1)\phi^2}{\left(u_k+\frac{1}{k}\right)^{\gamma}} 
dxdt.
    \end{array}
\end{equation}
The right hand side of \cref{non1} is non-positive for $u_k\leq 1$, and for $u_k>1$, we note $\gamma \geq 1$ and estimate the term as
\begin{equation}{\label{non2}}
\begin{array}{rcl}
 \iint_{\Omega_T\cap \{u_k>1\}} \frac{\nu_k(u_k-1)\phi^2}{\left(u_k+\frac{1}{k}\right)^{\gamma}} 
&\leq& \iint_{\Omega_T\cap \{u_k>1\}} \frac{\nu_ku_k\phi^2}{u_k^{\gamma}} 
 = \iint_{\Omega_T\cap \{u_k>1\}}  \frac{\nu_k\phi^2}{u_k^{\gamma-1}}
\leq \iint_{\Omega_T}  \nu_k \phi^2
 \leq C\|\nu_k\|_{L^1(\Omega_T)}.
\end{array}
\end{equation}
Thus from \cref{non1} using \cref{non2} we get by Young's inequality 
\begin{equation}{\label{non4}}
    \begin{array}{l}
      \int_0^T\int_\Omega |\nabla u_k|^2\phi^2
      d x d t
      \leq C+\ep\int_0^T\int_\Omega |\nabla u_k|^2\phi^2
      d x d t+c\int_0^T\int_\Omega |\nabla \phi|^2(u_k-1)^2
      d x d t.
    \end{array}
\end{equation}
where we also used that
\begin{equation*}
\left|\int_0^T \langle(u_k)_t(u_k-1)\phi^2\rangle d t\right| \leq 2\left|\iint_{\Omega_T}(u_k^2-u_k) \phi \phi_t\right| \leq C \|u_k\|_{L^\infty(0,T; L^2(\Omega))}\leq C.   
\end{equation*}
Hence, by choosing appropriate $\ep
$ one has from \cref{non4}
\begin{equation}{\label{non5}}
    \begin{array}{l}
         \int_0^T\int_\Omega |\nabla u_k|^2\phi^2
      d x d t
      \leq C+c\int_0^T\int_\Omega |\nabla \phi|^2(u_k-1)^2
      d x d t
      \leq C+c\|u_k\|_{L^\infty(0,T; L^2(\Omega))}.
    \end{array}
 \end{equation} Noting the boundedness of $\{u_k\}$ in $L^\infty(0,T; L^2(\Omega))$
from \cref{non5}, we get 
\begin{equation*}
    \int_0^T\int_\Omega |\nabla u_k|^2\phi^2dxdt\leq C,
\end{equation*}
where the constant $C$ does not depend on $k$. Thus the sequence $\{u_k\}_{k \in \mathbb{N}}$ is bounded in $ L^2_{\mathrm{loc}}(0,T;W_{\mathrm{loc}}^{1,2}(\Omega))$. Proceeding similarly like \cref{coRo} and \cref{convergenceinL1}, we get the convergence results. Indeed the absence of nonlocal term will not affect in any way the proofs of \cref{coRo}, \cref{convergenceinL1}.
\end{proof}
\section{Proof of non-existence result \cref{nonexistence}}{\label{sec3}}
In the whole proof, $\omega$ will indicate any quantity that vanishes as the parameters in its argument go to their (obvious, if not explicitly stressed) limit point with the same order in which they appear, that is, for example,
\begin{equation*}
\lim _{\delta \rightarrow 0^{+}} \limsup _{m \rightarrow+\infty} \limsup _{k \rightarrow \infty}|\omega(k, m, \delta)|=0
\end{equation*}Moreover, for the sake of simplicity, in what follows, the convergences, even if not explicitly stressed, may be understood to be taken possibly up to a suitable subsequence extraction.

By \cref{preli3}, we have $\exists\,u\in L^2_{\mathrm{loc}}(0,T;W_{\mathrm{loc}}^{1,2}(\Omega))\cap L^\infty(0,T; L^{\gamma+1}(\Omega))$ such that $u^{\frac{\gamma+1}{2}}\in L^2(0,T;W_0^{1,2}(\Omega))\cap L^\infty(0,T;L^{2}(\Omega))$ and up to a subsequence $u_k\rightharpoonup u $ weakly in $L^2_{\mathrm{loc}}(0,T;W_{\mathrm{loc}}^{1,2}(\Omega))$ and $u_k^{\frac{\gamma+1}{2}}\rightharpoonup u^{\frac{\gamma+1}{2}}$ weakly in $L^2(0,T;W_0^{1,2}(\Omega))$. Further, $\{u_k\}_{k\in\mathbb{N}}$ converges to $u$ strongly in $L^1_{\mathrm{loc}}(0, T ; L_{\mathrm{loc}}^1(\Omega))$ and a.e. on $\Omega_T$.
\\Let $\eta>0$, and let $\psi_\eta$ in $C_c^1(\Omega_T)$ be such that
\begin{equation}{\label{nonn1}}
0 \leq \psi_\eta \leq 1, \quad 0 \leq \iint_{\Omega_T}\left(1-\psi_\eta\right) d \nu \leq \eta, \quad \iint_{\Omega_T}|\nabla \psi_\eta|^2 \leq \eta   , 
\end{equation}
and there exists a decomposition of $(\psi_\eta)_t$ in $L^1(\Omega_T)+L^2(0,T;W^{-1,2}(\Omega))$ such that
\begin{equation}{\label{nonn2}}
    \left\|(\psi_\eta)^1_t\right\|_{L^1(\Omega_T)}\leq \frac{\eta}{3}, \quad \left\|(\psi_\eta)^1_t\right\|_{L^2(0,T;W^{-1,2}(\Omega))}\leq \frac{\eta}{3}.
\end{equation}
Moreover, $\psi_\eta$ converges to $0$ weakly-$*$ in $L^\infty(\Omega_T)$, in $L^1(\Omega_T)$, and, up to subsequences, almost everywhere as $\eta$ vanishes. 
Such a function exists since the set $E$ where the measure is concentrated has zero parabolic $2$ capacity (see for example [\citealp{pett}, Lemma 5]).

For technical reasons, we will from now on use double cut-off functions $\psi_\eta\psi_\delta$. Recall \cref{beta} and denote $0\leq h_{m,\gamma}(r)=1-\beta_m^\gamma(r)\leq 1$, $\Psi_{\eta\delta}=\psi_\eta\psi_\delta$ and take $T_1^\gamma(u_k)(1-\Psi_{\eta\delta})h_{m\gamma}(u_k)$ as a test function in \cref{aproxproblemnonn} to get
\begin{equation}{\label{nnon1}}
    \begin{array}{l}
      \quad  \underbrace{\int_0^T \langle(u_k)_t,T_1^\gamma(u_k)(1-\Psi_{\eta\delta})h_{m\gamma}(u_k)\rangle d t }_{I}+\int_0^T\int_\Omega(\nabla u_k\cdot\nabla T_1^\gamma(u_k))(1-\Psi_{\eta\delta})h_{m\gamma}(u_k)\\-\underbrace{\iint_{\Omega_T}(\nabla u_k\cdot\nabla \Psi_{\eta\delta})T_1^\gamma(u_k)h_{m\gamma}(u_k)}_{II}
     \\
= \underbrace{\int_0^T\int_{\Omega} \frac{\nu_k}{\left(u_k+\frac{1}{k}\right)^{\gamma}} T_1^\gamma(u_k)(1-\Psi_{\eta\delta})h_{m\gamma}(u_k)
}_{III}-\underbrace{\int_0^T\int_\Omega(\nabla u_k\cdot\nabla h_{m\gamma}(u_k))T_1^\gamma(u_k)(1-\Psi_{\eta\delta})}_{IV}.
    \end{array}
\end{equation}
Note that $\frac{T_1^\gamma(u_k)h_{m\gamma}(u_k)}{\left(u_k+\frac{1}{k}\right)^{\gamma}}\leq \frac{u_k^\gamma}{\left(u_k+\frac{1}{k}\right)^{\gamma}}\leq 1$. Since $\nu_k$ converges to $\nu$ in the narrow topology of measures, one has from \cref{nonn1} that 
\begin{equation}{\label{nnon2}}
\begin{array}{rcl}
0\leq III&=&\int_0^T\int_{\Omega} \frac{\nu_k}{\left(u_k+\frac{1}{k}\right)^{\gamma}} T_1^\gamma(u_k)(1-\Psi_{\eta\delta})h_{m\gamma}(u_k)
d x dt\leq \iint_{\Omega_T} \nu_k(1-\Psi_{\eta\delta})dxdt\smallskip\\&=&  \iint_{\Omega_T} (1-\psi_{\eta,\delta})d\nu+\omega(k)=\iint_{\Omega_T} (1-\psi_{\eta})d\nu+\iint_{\Omega_T} \psi_\eta(1-\psi_{\delta})d\nu+\omega(k)\smallskip\\&\leq&\iint_{\Omega_T} (1-\psi_{\eta})d\nu+\iint_{\Omega_T} (1-\psi_{\delta})d\nu+\omega(k)\leq \eta+\delta+\omega(k)=\omega(k,\delta,\eta).   \end{array} 
\end{equation}The last term in right of \cref{nnon1} becomes
\begin{equation*}
\begin{array}{rcl}
   -\int_0^T\int_\Omega(\nabla u_k\cdot\nabla h_{m\gamma}(u_k))T_1^\gamma(u_k)(1-\Psi_{\eta\delta})&=&\int_0^T\int_\Omega(\nabla u_k\cdot\nabla \beta_m^\gamma(u_k))T_1^\gamma(u_k)(1-\Psi_{\eta\delta})\smallskip\\
    &=&\frac{\gamma}{m}\int_{\{m\leq u_k\leq 2m\}}\nabla u_k\cdot\nabla u_k(1-\psi_\eta)T_1^\gamma(u_k)(\beta_m(u_k))^{\gamma-1}\smallskip\\&&+\frac{\gamma}{m}\int_{\{m\leq u_k\leq 2m\}}\nabla u_k\cdot\nabla u_k\psi_\eta(1-\psi_\delta)T_1^\gamma(u_k)(\beta_m(u_k))^{\gamma-1}\end{array}
\end{equation*}Note that $\nabla u_k\cdot\nabla u_k(1-\psi)(\beta_m(u_k))^{\gamma-1}\geq 0$, where $\psi=\psi_\eta$ or $\psi_\delta$, and $0\leq T_1(u_k),\psi_\delta,\psi_\eta\leq 1$. Further as $\{u^{\frac{\gamma+1}{2}}_k\}_{k \in \mathbb{N}}$ is bounded in $L^2(0,T;W_0^{1,2}(\Omega))$, we have using the above estimate, 
\begin{equation}{\label{nonn3}}
    \begin{array}{rcl}
         0\leq IV=-\int_0^T\int_\Omega(\nabla u_k\cdot\nabla h_{m\gamma}(u_k))T_1^\gamma(u_k)(1-\Psi_{\eta\delta})&\leq &
         \frac{2\gamma}{m}\int_{\{m\leq u_k\leq 2m\}}|\nabla u_k|^2\smallskip\\&\leq& 
         \frac{2\gamma}{m^\gamma}\int_{\{m\leq u_k\leq 2m\}}u_k^{\gamma-1}|\nabla u_k|^2\smallskip\\&=&
         \frac{c}{m^\gamma}\int_{\{m\leq u_k\leq 2m\}}|\nabla u^{\frac{\gamma+1}{2}}_k|^2\leq 
         \frac{C}{m^\gamma}=
         \omega(m).
    \end{array}
\end{equation}   
Now note that $\nabla u_k\rightharpoonup\nabla u$ in $\left(L^2_{\mathrm{loc}}(\Omega_T)\right)^n$ and $T_1^\gamma(u_k)h_{m\gamma}(u_k)\nabla \Psi_{\eta\delta}$ converges to $T_1^\gamma(u)h_{m\gamma}(u)\nabla \Psi_{\eta\delta}$ pointwise as $k\to \infty$ in $\Omega_T$. Further noting that $T_1^\gamma(u_k)h_{m\gamma}(u_k)\leq 1$ and $\Psi_{\eta\delta}\in C_c^1(\Omega_T)$, we have by dominated convergence theorem $T_1^\gamma(u_k)h_{m\gamma}(u_k)\nabla \Psi_{\eta\delta}$ converges weakly-$*$ to $T_1^\gamma(u)h_{m\gamma}(u)\nabla \Psi_{\eta\delta}$ in $\left(L^\infty(\Omega_T)\right)^n$. Therefore by \cref{limits}, we have 
\begin{equation}{\label{mie22}}
    -II=\iint_{\Omega_T}(\nabla u_k\cdot\nabla \Psi_{\eta\delta})T_1^\gamma(u_k)h_{m\gamma}(u_k)dxdt=\iint_{\Omega_T}(\nabla u\cdot\nabla \Psi_{\eta\delta})T_1^\gamma(u)h_{m\gamma}(u)dxdt+w(k).
\end{equation}
Now $|\nabla u|h_{m\gamma}(u)T_1^\gamma(u)\leq \frac{2}{\gamma+1}\left|\nabla u^{\frac{\gamma+1}{2}}\right|h_{m\gamma}(u)T_1(u)$ and the later function is in $L^2(\Omega_T)$. So by \cref{nonn1}, as $\psi_\delta\to 0$ in $L^2(0,T;W^{1,2}_0(\Omega))$, we have from \cref{mie22}
\begin{equation}{\label{nonn6}}
   - II=\iint_{\Omega_T}\psi_\delta(\nabla u\cdot\nabla \psi_{\eta})T_1^\gamma(u)h_{m\gamma}(u)dxdt+\iint_{\Omega_T}(\nabla u\cdot\nabla \psi_{\delta})\psi_\eta T_1^\gamma(u)h_{m\gamma}(u)dxdt=\omega(k,\delta).
\end{equation}
For the term $(I)$, denote $\Theta_{m,\gamma}(r)=\int_0^rT_1^\gamma(s)h_{m\gamma}(s)ds$ and note that $\Theta_{m,\gamma}(u_k)$ converges to $\Theta_{m,\gamma}(u)$ as $k\to \infty$ a.e in $\Omega_T$. Further, as $\gamma\geq 1$ and $h_{m\gamma}\leq 1$, so we have $|(\Psi_{\eta\delta})_t \Theta_{m,\gamma}(u_k)|\leq u_k|(\Psi_{\eta\delta})_t|$. So by dominated convergence theorem and the fact that $u_k\to u$ in $L^1_{\mathrm{loc}}(\Omega_T)$, one gets 
\begin{equation}{\label{mie}}
   \lim_{k\rightarrow\infty} \int_0^T\int_\Omega (\Psi_{\eta\delta})_t \Theta_{m,\gamma}(u_k)dxdt= \int_0^T\int_\Omega (\Psi_{\eta\delta})_t \Theta_{m,\gamma}(u)dxdt.
\end{equation} 
From \cref{mie}, we use the fact $u_k(\cdot,0)\equiv 0$ and estimate as
\begin{equation}{\label{nonn7}}
\begin{array}{rcl}
    I=\int_0^T \langle(u_k)_t,T_1^\gamma(u_k)(1-\Psi_{\eta\delta})h_{m\gamma}(u_k)\rangle  d t&=&\underbrace{\int_\Omega \Theta_{m,\gamma}(u_k)(x,T)(1-\Psi_{\eta\delta})(x,T)dx}_{\geq0}+\int_0^T\int_\Omega (\Psi_{\eta\delta})_t \Theta_{m,\gamma}(u_k)\smallskip\\&\geq&\int_0^T\int_\Omega (\Psi_{\eta\delta})_t \Theta_{m,\gamma}(u)dxdt+\omega(k).
    \end{array} 
\end{equation}
As $m\geq 1$, $u$ is not necessarily bounded, so maximum possible value of $\Theta_{k,m}$ is given by
\begin{equation}{\label{mie33}}
    \begin{array}{rcl}
         \Theta_{m,\gamma}(u)&\leq&\int_0^u T_k(s)h_m(s) ds=\int_0^1+\int_1^m+\int_m^{2m}+\int_{2m}^u\smallskip\\&=& \int_0^1 s\times1 ds+\int_1^m 1\times 1 ds+\int_m^{2m} 1\times \left(1-(\frac{s}{m}-1)^\gamma\right)ds+\int_{2m}^u 0ds\smallskip\\&\leq &\frac{1}{2}+(m-1)+(2m-m)
         =\frac{4m-1}{2}.
    \end{array}
\end{equation}
From \cref{nonn2}, let $(\psi_\delta)^1_t+(\psi_\delta)^2_t$ be the decomposition of $(\psi_\delta)_t$ in $L^1(\Omega_T)+L^2(0,T;W^{-1,2}(\Omega))$. Note that $\left|\nabla \Theta_{m,\gamma}(u)\right|\leq c|\nabla u|$, and $u\in L_{\mathrm{loc}}^2(0,T;W_{\mathrm{loc}}^{1,2}(\Omega))$.
As $0\leq\psi_\delta,\psi_\eta\leq 1$, and $\psi_\delta,\psi_\eta\in C_c^\infty(\Omega_T)$, we have 
$$\begin{aligned}
    \left|\int_0^T\int_\Omega (\Psi_{\eta\delta})_t \Theta_{m,\gamma}(u)dxdt\right|&=\left|\iint_{\Omega} \Theta_{m,\gamma}(u)\left(\psi_\eta(\psi_\delta)_t+(\psi_\eta)_t\psi_\delta\right)\right|\smallskip\\&=\left|\iint_{\Omega} \Theta_{m,\gamma}(u)\left(\psi_\eta(\psi_\delta)^1_t+\psi_\eta(\psi_\delta)^2_t+(\psi_\eta)_t\psi_\delta\right)\right|\smallskip\\&\leq \iint_{\Omega_T} \Theta_{m,\gamma}(u)\left|\psi_\eta(\psi_\delta)^1_t+(\psi_\eta)_t\psi_\delta\right|+\left|\int_0^T \langle(\psi_\delta)^2_t,\Theta_{m,\gamma}(u)\psi_\eta\rangle\right|\smallskip\\&\leq \iint_{\Omega_T} \Theta_{m,\gamma}(u)\left|\psi_\eta(\psi_\delta)^1_t\right|+\iint_{\Omega_T} \Theta_{m,\gamma}(u)\left|(\psi_\eta)_t\psi_\delta\right|\smallskip\\&\quad+\|(\psi_\delta)^2_t\|_{L^2(0,T;W^{-1,2}(\Omega))}\|\Theta_{m,\gamma}(u)\psi_\eta\|_{L^2(0,T;W^{1,2}_0(\Omega))}
    \smallskip\\&\stackrel{\cref{nonn2}\cref{mie33}}{\leq} \iint_{\Omega_T}\frac{4m-1}{2}|(\psi_\delta)^1_t|+\frac{4m-1}{2}\|(\psi_\eta)_t\|_{L^\infty(\Omega_T)}\iint_{\Omega_T}|\psi_\delta|\smallskip\\&\quad +C(\gamma)\delta\left( \|u\|_{L^2(t_1,t_2;W^{1,2}(\Omega_1))}+\frac{4m-1}{2}\|\nabla \psi_\eta\|_{L^\infty(\Omega_T)}\right)\smallskip\\&\leq {4m}\delta+4mC_\eta\delta+C(\gamma)\|u\|_{L^2(t_1,t_2;W^{1,2}(\Omega_1))}\delta\smallskip\\
    &\leq 3\eta,
\end{aligned}
$$
provided $\delta<\min\{\frac{\eta}{4m},\frac{\eta}{4mC_\eta},\frac{\eta}{C(\gamma)\|u\|_{L^2(t_1,t_2;W^{1,2}(\Omega_1))}}\}$, where $\operatorname{supp}\psi_\eta\subset[t_1,t_2]\times\Omega_1\subset\subset\Omega_T$. Letting now $k\to \infty$ and then $\delta\to 0+$, we get
\begin{equation*}
  \lim_{\delta\to 0+}\lim_{k\to \infty}   \left|\iint_{\Omega_T} \Theta_{m,\gamma}(u_k)(\Psi_{\eta\delta})_t \right|\leq 3\eta,
\end{equation*}
which gives by taking $m\to \infty$ and then $\eta\to 0+$, 
\begin{equation*}
    \lim_{\eta\to 0+}\lim_{m\to \infty}  \lim_{\delta\to 0+}\lim_{k\to \infty}   \left|\iint_{\Omega_T} \Theta_{m,\gamma}(u_k)(\Psi_{\eta\delta})_t \right|=0.
\end{equation*}
i.e. 
\begin{equation}{\label{nnon77}}
       \left|\iint_{\Omega_T} \Theta_{m,\gamma}(u_k)(\Psi_{\eta\delta})_t \right|=\omega(k,\delta,m,\eta).
\end{equation}
Finally merging \cref{nnon2}, \cref{nonn3}, \cref{nonn6}, \cref{nonn7} and \cref{nnon77} 
in \cref{nnon1}, we have by weak lower semi-continuity
\begin{equation*}{\label{nnonfin}}
    \begin{array}{rcl}
      0\leq \int_0^T\int_\Omega|\nabla T_1^{\frac{\gamma+1}{2}}(u)|^2(1-\Psi_{\eta\delta})&\leq& \liminf_{k\to \infty}\int_0^T\int_\Omega|\nabla T_1^{\frac{\gamma+1}{2}}(u_k)|^2(1-\Psi_{\eta\delta})\\&=&\liminf_{k\to \infty}\int_0^T\int_\Omega|\nabla T_1^{\frac{\gamma+1}{2}}(u_k)|^2(1-\Psi_{\eta\delta})h_{m\gamma}(u_k)\leq \omega(\delta,m,\eta).
    \end{array}
\end{equation*}As $\psi_{\delta}\to 0$ weakly-$*$ in $L^\infty(\Omega_T)$, 
we get $\iint_{\Omega_T}|\nabla T_1^{\frac{\gamma+1}{2}}(u)|^2=0$, which gives $T_1(u)=0$ implying $u=0$ as desired.
\section{Proof of existence results}
\subsection*{Proof of \cref{exis1}}{\label{sec4}}
We assume that $u_0\in L^1(\Omega)$, $\gamma\in C(\overline{\Omega}_T)$ and is locally Lipschitz continuous with respect to $x$ variable in $\Omega_T$ and satisfies the condition $(P_1)$ for some $\gamma^*\geq 1$ and $\delta>0$. We consider the cases $\gamma^*=1$ and $\gamma^*>1$ separately.\smallskip\\
As per hypotheses, $\nu \in \overline{\mathcal{M}}_0^p(\Omega)$ for some $1<p<2-\frac{n}{n+1}$, is a non-negative bounded Radon measure on $\Omega_T$ and $\mu$ is a non-negative bounded Radon measure on $\Omega_T$. Further $\nu$ is non-singular with respect to the Lebesgue measure. Hence, by Lebesgue decomposition theorem, [\citealp{royden}, page 384]
\begin{equation*}
    \nu=\nu_a+\nu_s,
\end{equation*}
where $0\neq \nu_a \ll \mathcal{L}$ and $\nu_s\perp \mathcal{L}$. By Radon-Nikodym theorem [\citealp{royden}, Page 382], there exists a non-negative Lebesgue measurable function $h$ such that for every measurable set $E \subset \Omega_T$,
\begin{equation}{\label{muu}}
\nu_a(E)=\iint_E h d x dt.    
\end{equation}
Furthermore, as $\nu$ is bounded then $h \in L^1(\Omega_T)$. Since $\nu \in \overline{\mathcal{M}}_0^p(\Omega)$, $0 \leq \nu_s \leq \nu$, and \cref{muu} holds, we have $\nu_s \in \overline{\mathcal{M}}_0^p(\Omega)$. Let $f\in L^1(\Omega_T)$ and $G\in \left(L^{p^\prime}(\Omega_T)\right)^n$ are such that for every $\phi\in C_c^\infty(\Omega_T)$, it holds
   \begin{equation*}
         \iint_{\Omega_T} \phi \,d \nu_s=\iint_{\Omega_T} f \phi \,d x d t+\int_0^T G\cdot\nabla \phi \, dxd t,
     \end{equation*}
   Formally we write $\nu_s=f-\operatorname{div}(G)$. Now noting \cref{measapx}, we take the following approximations of $\nu_s$ and $\nu$:
\begin{equation}{\label{appxx}}
\begin{array}{c}
  0\leq   (\nu_s)_k=f_k-\operatorname{div}\left(G_k\right), \text{ i.e. } \nu_k=T_k(h)+f_k-\operatorname{div}\left(G_k\right),\\  \left\|(\nu_s)_k\right\|_{L^1(\Omega_T)} \leq C\|\nu_s\|_{\mathcal{M}(\Omega_T)}\leq C\|\nu\|_{\mathcal{M}(\Omega_T)}.\end{array}
\end{equation}
where $f_k \in C_c^{\infty}(\Omega_T)$ is a sequence of functions which converges to $f$ weakly in $L^1(\Omega)$, $G_k \in C_c^{\infty}(\Omega_T)$ is a sequence of functions which converges to $G$ strongly in $\left(L^{p^{\prime}}(\Omega_T)\right)^n$. Thus note that $\nu_k\in L^\infty(\Omega_T)$ is a sequence of non-negative functions which is bounded in $L^1(\Omega_T)$ and also converges to $\nu$ in narrow topology of measures. 
\smallskip\\Further, since $\mu$ is a non-negative bounded Radon measure on $\Omega_T$, there exists a non-negative sequence $\{\mu_k\}_{k \in \mathbb{N}} \subset L^{\infty}(\Omega_T)$ such that $\|\mu_k\|_{L^1(\Omega_T)} \leq C$ for some constant $C>0$ independent of $k$ and $\mu_k \rightharpoonup \mu$ in the narrow topology. 

Now, for each $k\in \mathbb{N}$, consider the approximated problems
  \begin{equation}{\label{aproxproblem12}}
    \begin{array}{c}
         (u_k)_t-\Delta u_k+(-\Delta)^s u_k=\frac{\nu_k}{\left(u_k+\frac{1}{k}\right)^{\gamma(x,t)}} +\mu_k \text { in } \Omega_T, \\u_k>0 \text{ in }\Omega_T,\quad u_k=0  \text { in }(\mathbb{R}^n \backslash \Omega) \times(0, T), \\ u_k(x, 0)=T_k(u_0)(x) \text { in } \Omega ;
    \end{array}
\end{equation} and let $u_k$ denotes the solution to \cref{aproxproblem12}, which exists by \cref{exisforaprox}. Further, $0\leq u_k  \in L^{\infty}(\Omega_T)$ and $(u_k)_t \in L^{2}(0, T ; W^{-1, 2}(\Omega))$. In particular $u_k\in C([0,T],L^2(\Omega))$. Also, as $\nu$ is non-singular with respect to the Lebesgue measure, then 
    for each $t_0>0$, $\omega\subset\subset\Omega$ and $\forall k$, it holds \begin{equation}{\label{greater}}
        u_k\geq C\quad \text{ in }\omega\times[t_0,T),
    \end{equation} where $C\equiv C(\omega,t_0,n,s)$ is a constant independent of $k$. \smallskip\\
\textbf{Case 1: $\gamma^*=1$}\\ By \cref{preli1}, there exists a subsequence of $\{u_k\}_{k \in \mathbb{N}}$, still denoted by $\{u_k\}_{k \in \mathbb{N}}$ and $u \in L^p(0,T;W_0^{1, p}(\Omega))$ such that $u_k \rightharpoonup u$ weakly in $L^p(0,T;W_0^{1, p}(\Omega))$. Further by \cref{convergenceinL1}, we have $u_k\rightarrow u$ in $L^1_{\mathrm{loc}}(0,T;L^1_{\mathrm{loc}}(\Omega))$ and $u_k \rightarrow u$ pointwise a.e. in $\Omega_T$. Also by Fatou's lemma $u\in L^\infty(0,T;L^1(\Omega))$. Suppose $\phi \in C_c^{\infty}(\Omega_T)$, and $\operatorname{supp}(\phi)\subset \omega\times[t_1,t_2]\subset\subset\Omega_T$. From the weak formulation of \cref{aproxproblem12} noting \cref{first} and \cref{appxx}, we get
\begin{equation}{\label{e1}}
\begin{array}{l}
    -  \int_0^T\int_\Omega u_k \phi_tdx d t+\int_0^T\int_{\Omega}\nabla u_k\cdot \nabla \phi\, d x d t+ \int_0^T \int_{\mathbb{R}^{n}}\int_{\mathbb{R}^{n}} \frac{(u_k(x, t)-u_k(y, t))(\phi(x, t)-\phi(y, t))}{|x-y|^{n+2 s}} d x d y d t \smallskip\\ = \int_0^T\int_{\Omega} \frac{T_k(h)}{\left(u_k+\frac{1}{k}\right)^{\gamma(x,t)}} \phi \, dxdt+\int_0^T\int_{\Omega} \frac{f_k}{\left(u_k+\frac{1}{k}\right)^{\gamma(x,t)}} \phi \, dxdt \smallskip\\\quad+\int_0^T\int_{\Omega} G_k\cdot\nabla \left(\frac{\phi}{\left(u_k+\frac{1}{k}\right)^{\gamma(x,t)}} \right) dxdt+\int_0^T\int_{\Omega} \mu_k\phi \,
d x d t .
\end{array}
\end{equation}
Since $u_k \rightarrow u$ weakly in $L^p(0,T;W_0^{1, p}(\Omega))$, for every $\phi \in C_c^{\infty}(\Omega_T)$, it follows that that
\begin{equation}{\label{e2}}
\lim _{k \rightarrow \infty} \int_0^T\int_{\Omega}\nabla u_k\cdot \nabla \phi\, d x d t=\int_0^T\int_{\Omega}\nabla u\cdot \nabla \phi\, d x d t
\end{equation}
and
\begin{equation}{\label{e3}}
\begin{array}{c}
    \lim _{k \rightarrow \infty}\int_0^T \int_{\mathbb{R}^{n}}\int_{\mathbb{R}^{n}} \frac{(u_k(x, t)-u_k(y, t))(\phi(x, t)-\phi(y, t))}{|x-y|^{n+2 s}} d x d y d t \smallskip\\
    =\int_0^T \int_{\mathbb{R}^{n}}\int_{\mathbb{R}^{n}} \frac{(u(x, t)-u(y, t))(\phi(x, t)-\phi(y, t))}{|x-y|^{n+2 s}} d x d y d t .\end{array}
\end{equation}
By \cref{convergenceinL1}, we have \begin{equation}{\label{e4}}
     \lim _{k \rightarrow \infty}\int_0^T\int_\Omega u_k \phi_tdx d t=\int_0^T\int_\Omega u\phi_tdx d t.
\end{equation}
Moreover, since and $\mu_k\rightharpoonup \mu$ in the narrow topology, for every $\phi \in C_c^{\infty}(\Omega_T)$, we have
\begin{equation}{\label{e5}}
\lim _{k \rightarrow \infty} \int_0^T\int_{\Omega}\mu_k \phi  \,d xdt=dt \iint_{\Omega_T} \phi\, d \mu.
\end{equation}
By \cref{greater}, there exists $C>0$ (independent of $k$) such that $u_k\geq C$ in $\omega\times[t_1,t_2]\supset\operatorname{supp}\phi$, for all $k$. Since $h \in L^1(\Omega_T)$, from the Lebesgue's dominated convergence theorem, we get
\begin{equation}{\label{e6}}
      \lim _{k \rightarrow \infty}\int_0^T\int_{\Omega} \frac{T_k(h)}{\left(u_k+\frac{1}{k}\right)^{\gamma(x,t)}} \phi \, dxdt=\int_0^T\int_{\Omega} \frac{h}{u^{\gamma(x,t)}} \phi \, dxdt.
\end{equation}
It follows from the facts $f_k \rightharpoonup f$ weakly in $L^1(\Omega_T)$ and $\frac{\phi}{\left(u_k+\frac{1}{k}\right)^{\gamma(x,t)}} \rightarrow \frac{\phi}{u^{\gamma(x,t)}}$ pointwise a.e. in $\Omega_T$ and weak* in $L^{\infty}(\Omega_T)$ that (also note \cref{limits})
\begin{equation}{\label{e7}}
     \lim _{k \rightarrow \infty} \int_0^T\int_{\Omega} \frac{f_k}{\left(u_k+\frac{1}{k}\right)^{\gamma(x,t)}} \phi \, dxdt =\int_0^T\int_{\Omega} \frac{f}{u^{\gamma(x,t)}} \phi \, dxdt.
\end{equation}
We only have to pass the limit in the second last term of \cref{e1}. We observe that for every $\phi \in C_c^{\infty}(\Omega_T)$,
\begin{equation}{\label{e8}}
    \begin{array}{rcl}
    \int_0^T\int_{\Omega} G_k\cdot\nabla \left(\frac{\phi}{\left(u_k+\frac{1}{k}\right)^{\gamma(x,t)}} \right) dxdt&=&\int_0^T\int_{\Omega}  \frac{G_k\cdot\nabla\phi}{\left(u_k+\frac{1}{k}\right)^{\gamma(x,t)}}  dxdt\smallskip\\&&-\int_0^T\int_{\Omega}  \frac{G_k\cdot\nabla\gamma(x,t) }{\left(u_k+\frac{1}{k}\right)^{\gamma(x,t)}} \log\left(u_k+\frac{1}{k} \right)\phi\, dxdt\smallskip\\&&-\int_0^T\int_{\Omega}\gamma(x,t)  \frac{G_k\cdot\nabla u_k}{\left(u_k+\frac{1}{k}\right)^{\gamma(x,t)+1}}\phi\,  dxdt.
    \end{array}
\end{equation}
The fact $G_k\rightarrow G$ strongly in $\left(L^{p^{\prime}}(\Omega_T)\right)^n$ implies $G_k\rightarrow G$ strongly in $\left(L^{1}(\Omega_T)\right)^n$. So by \cref{greater}, it holds 
\begin{equation*}
    \left| \frac{G_k\cdot\nabla\phi}{\left(u_k+\frac{1}{k}\right)^{\gamma(x,t)}}\right|\leq C |G_k|\text{ in }\operatorname{supp}\phi.
\end{equation*} Therefore by Lebesgue dominated convergence theorem, it holds 
\begin{equation}{\label{e9}}
     \lim _{k \rightarrow \infty}\int_0^T\int_{\Omega}  \frac{G_k\cdot\nabla\phi}{\left(u_k+\frac{1}{k}\right)^{\gamma(x,t)}}  dxdt=\int_0^T\int_{\Omega}  \frac{G\cdot\nabla\phi}{u^{\gamma(x,t)}}  dxdt.
\end{equation}
Since \begin{equation*}\frac{\phi\gamma(x,t)\nabla u_k}{\left(u_k+\frac{1}{k}\right)^{\gamma(x,t)+1}}  \rightharpoonup \frac{\phi\gamma(x,t)\nabla u}{u^{\gamma(x,t)+1}} \end{equation*} weakly in $\left(L^p(\Omega_T)\right)^n$ and $G_k \rightarrow G$ strongly in $\left(L^{p^{\prime}}(\Omega_T)\right)^n$, we obtain
\begin{equation}{\label{e10}}
   \lim _{k \rightarrow \infty}  \int_0^T\int_{\Omega}\gamma(x,t)  \frac{G_k\cdot\nabla u_k}{\left(u_k+\frac{1}{k}\right)^{\gamma(x,t)+1}}\phi\,dxdt=\int_0^T\int_{\Omega}\gamma(x,t)  \frac{G\cdot\nabla u}{u^{\gamma(x,t)+1}}\phi\,dxdt
\end{equation}
In order to pass to the limit in the second last integral in \cref{e8}, we use the condition $\gamma(x)>0$ and $\gamma$ is locally Lipschitz continuous with respect to $x$ variable in $\Omega_T$. This is the only place where we used the Lipschitz property; to be more precise, we use the fact $\|\nabla \gamma\|_{L_{\mathrm{loc }}^{\infty}(\Omega_T)}<\infty$ to apply Lebesgue's dominated convergence theorem. Define $r:=\min _{\Omega_T} \gamma(x,t)>0$ and observe that $\frac{\log z}{z^{r}}$ is bounded on $[C, \infty)$, where $C>0$ is a uniform lower bound of $u_k$ in $\omega\times[t_2,t_2]$.
We have
\begin{equation*}
    \left|\frac{G_k\cdot\nabla\gamma(x,t) }{\left(u_k+\frac{1}{k}\right)^{\gamma(x,t)}} \log\left(u_k+\frac{1}{k} \right)\phi \right|=\left|\left(\frac{\phi}{\left(u_k+\frac{1}{k}\right)^{\gamma(x,t)-r}}\right)\left(\frac{\log \left(u_k+\frac{1}{k}\right)}{\left(u_k+\frac{1}{k}\right)^r}\right) G_k\cdot \nabla \delta\right|\leq C|G_k|.
\end{equation*}
in $\omega\times [t_1,t_2]$. Using this together with the fact $\lim _{k \rightarrow \infty} \iint_{\Omega_T}\left|G_k\right|=\iint_{\Omega_T}|G|$, we conclude
\begin{equation}{\label{e11}}
   \lim _{k \rightarrow \infty}  \int_0^T\int_{\Omega}  \frac{G_k\cdot\nabla\gamma(x,t) }{\left(u_k+\frac{1}{k}\right)^{\gamma(x,t)}} \log\left(u_k+\frac{1}{k} \right)\phi\, dxdt=\int_0^T\int_{\Omega}  \frac{G\cdot\nabla\gamma(x,t) }{u^{\gamma(x,t)}} (\log u)\phi\, dxdt.
\end{equation}
Combining \cref{e9,e10,e11}, the identity \cref{e8} leads to conclude that
\begin{equation}{\label{e12}}
     \lim _{k \rightarrow \infty}\int_0^T\int_{\Omega} G_k\cdot\nabla \left(\frac{\phi}{\left(u_k+\frac{1}{k}\right)^{\gamma(x,t)}} \right) dxdt=\int_0^T\int_{\Omega} G\cdot\nabla \left(\frac{\phi}{u^{\gamma(x,t)}} \right) dxdt.
\end{equation}
Thus noting that $\frac{\phi}{u^{\gamma(x,t)}}\in L^p(0,T,W^{1,p}_0(\Omega)\cap L^\infty(\Omega_T)$, we let $k \rightarrow \infty$ in both sides of the equality \cref{e1}; and use \cref{e2,e3,e4,e5,e6,e7} and \cref{e12}, to obtain
\begin{equation*}{\label{e13}}
\begin{array}{l}
    -  \int_0^T\int_\Omega u \phi_tdx d t+\int_0^T\int_{\Omega}\nabla u\cdot \nabla \phi\, d x d t+ \int_0^T \int_{\mathbb{R}^{n}}\int_{\mathbb{R}^{n}} \frac{(u(x, t)-u(y, t))(\phi(x, t)-\phi(y, t))}{|x-y|^{n+2 s}} d x d y d t \smallskip\\ = \int_0^T\int_{\Omega} \frac{h}{u^{\gamma(x,t)}} \phi \, dxdt+\int_0^T\int_{\Omega} \frac{f}{u^{\gamma(x,t)}} \phi \, dxdt +\int_0^T\int_{\Omega} G\cdot\nabla \left(\frac{\phi}{u^{\gamma(x,t)}} \right) dxdt+\int_0^T\int_{\Omega} \phi \,d\mu\smallskip\\=\iint_{\Omega_T}\frac{\phi}{u^{\gamma(x,t)}}d\nu+\iint_{\Omega_T}{\phi}\,d\mu.
\end{array}
\end{equation*}
Hence, $u \in L^p(0,T:W_0^{1, p}(\Omega))\cap L^\infty(0,T;L^1(\Omega))$ is a solution of the equation \cref{mainproblem}.\smallskip\\
\textbf{Case 2: $\gamma^*>1$} \\ We apply \cref{preli1}, \cref{convergenceinL1} and conclude that there exist a subsequence of $\{u_k\}_{k \in \mathbb{N}}$, still denoted by $\{u_k\}_{k \in \mathbb{N}}$ and $u \in L^p(t_0,T;W_{\mathrm{loc}}^{1, p}(\Omega))$, for each $t_0\in (0,T)$, such that
$$
\left\{\begin{array}{l}
u_k \rightarrow u \text { weakly in } L^p(t_0,T;W_{\mathrm{loc}}^{1, p}(\Omega)),\, \forall t_0\in (0,T),\smallskip\\ u_k \rightarrow u \text { strongly in } L^1_{\mathrm{loc}}(0,T;L_{\mathrm{loc}}^1(\Omega)) \smallskip\\
u_k \rightarrow u \text { pointwise a.e. in } \Omega_T .
\end{array}\right.
$$Moreover, by Fatou's lemma $u\in L^\infty(0,T;L^1(\Omega))$.

For any $\phi \in C_c^{\infty}(\Omega_T)$, and $\operatorname{supp}(\phi)\subset \omega\times[t_1,t_2]\subset\subset\Omega_T$, from the weak formulation of \cref{aproxproblem12} noting \cref{first} and \cref{appxx}, we get
\begin{equation}{\label{ee1}}
\begin{array}{l}
    -  \int_0^T\int_\Omega u_k \phi_tdx d t+\int_0^T\int_{\Omega}\nabla u_k\cdot \nabla \phi\, d x d t+ \int_0^T \int_{\mathbb{R}^{n}}\int_{\mathbb{R}^{n}} \frac{(u_k(x, t)-u_k(y, t))(\phi(x, t)-\phi(y, t))}{|x-y|^{n+2 s}} d x d y d t \smallskip\\ = \int_0^T\int_{\Omega} \frac{T_k(h)}{\left(u_k+\frac{1}{k}\right)^{\gamma(x,t)}} \phi \, dxdt+\int_0^T\int_{\Omega} \frac{f_k}{\left(u_k+\frac{1}{k}\right)^{\gamma(x,t)}} \phi \, dxdt \smallskip\\\quad+\int_0^T\int_{\Omega} G_k\cdot\nabla \left(\frac{\phi}{\left(u_k+\frac{1}{k}\right)^{\gamma(x,t)}} \right) dxdt+\int_0^T\int_{\Omega} \mu_k\phi \,
d x d t .
\end{array}
\end{equation}
By repeating the similar proof as in case $1$ above, one can pass to the limit in all the integrals in \cref{ee1} except the third integral, which is nonlocal. In order to pass to the limit there, by \cref{preli1}, since the sequence $\left\{T_1^{\frac{\gamma^*+1}{2}}(u_k)\right\}_{k\in \mathbb{N}}$ is uniformly bounded in $L^2(0,T;W_0^{1,2}(\Omega))$, by a similar argument as in [\citealp{fracvari1}, Theorem 3.6] (see also [\citealp{pap}, Remark 3.4] for a parabolic version), we have
\begin{equation}{\label{ee222}}
    \begin{array}{c}
      \lim _{k \rightarrow \infty}\int_0^T \int_{\mathbb{R}^{n}}\int_{\mathbb{R}^{n}} \frac{(T_1(u_k)(x, t)-T_1(u_k)(y, t))(\phi(x, t)-\phi(y, t))}{|x-y|^{n+2 s}} d x d y d t\\=\int_0^T \int_{\mathbb{R}^{n}}\int_{\mathbb{R}^{n}} \frac{(T_1(u)(x, t)-T_1(u)(y, t))(\phi(x, t)-\phi(y, t))}{|x-y|^{n+2 s}} d x d y d t 
    \end{array}
\end{equation}
Furthermore, by \cref{preli1}, since $\left\{G_1(u_k)\right\}_{k \in \mathbb{N}}$ is uniformly bounded in $L^p(0,T;W_0^{1, p}(\Omega))$, up to a subsequence we have $G_1(u_k) \rightarrow G_1(u)$ weakly in $L^p(0,T;W_0^{1, p}(\Omega)$. Thus
\begin{equation}{\label{ee2}}
    \begin{array}{c}
      \lim _{k \rightarrow \infty}\int_0^T \int_{\mathbb{R}^{n}}\int_{\mathbb{R}^{n}} \frac{(G_1(u_k)(x, t)-G_1(u_k)(y, t))(\phi(x, t)-\phi(y, t))}{|x-y|^{n+2 s}} d x d y d t \\=\int_0^T \int_{\mathbb{R}^{n}}\int_{\mathbb{R}^{n}} \frac{(G_1(u)(x, t)-G_1(u)(y, t))(\phi(x, t)-\phi(y, t))}{|x-y|^{n+2 s}} d x d y d t 
    \end{array}
\end{equation}
Combining \cref{ee222,ee2} and using the fact $u_k=T_1(u_k)+G_1(u_k)$, we obtain
\begin{equation}{\label{ee21}}
    \begin{array}{c}
      \lim _{k \rightarrow \infty}\int_0^T \int_{\mathbb{R}^{n}}\int_{\mathbb{R}^{n}} \frac{(u_k(x, t)-u_k(y, t))(\phi(x, t)-\phi(y, t))}{|x-y|^{n+2 s}} d x d y d t \\=\int_0^T \int_{\mathbb{R}^{n}}\int_{\mathbb{R}^{n}} \frac{(u(x, t)-u(y, t))(\phi(x, t)-\phi(y, t))}{|x-y|^{n+2 s}} d x d y d t 
    \end{array}
\end{equation}
Letting $k \rightarrow \infty$ on both sides of the identity \cref{ee1} and using \cref{ee21}, we obtain
\begin{equation*}{\label{ee13}}
\begin{array}{c}
    -  \int_0^T\int_\Omega u \phi_tdx d t+\int_0^T\int_{\Omega}\nabla u\cdot \nabla \phi\, d x d t+ \int_0^T \int_{\mathbb{R}^{n}}\int_{\mathbb{R}^{n}} \frac{(u(x, t)-u(y, t))(\phi(x, t)-\phi(y, t))}{|x-y|^{n+2 s}} d x d y d t \smallskip\\ =\iint_{\Omega_T}\frac{\phi}{u^{\gamma(x,t)}}d\nu+\iint_{\Omega_T}{\phi}\,d\mu.
\end{array}
\end{equation*}
Thus $u \in L^p(t_0,T;W_{\mathrm{loc}}^{1, p}(\Omega))\cap  L^\infty(0,T;L^1(\Omega))$, for each $t_0\in (0,T)$, solves the equation \cref{mainproblem}. Furthermore, using \cref{preli1}-(2) and \cref{convergenceinL1} one obtain that $T_l(u) \in L^p(t_0,T;W_{\mathrm{loc}}^{1, p}(\Omega))$ such that $T_l^{\frac{\gamma^*+1}{2}}(u) \in L^2(0,T;W_0^{1,2}(\Omega))$ for every $l>0$, $t_0\in (0,T)$.

Finally, if $\nu \in L^1(\Omega_T) \backslash\{0\}$ is a non-negative, therefore $\nu \in \overline{\mathcal{M}}_0^p(\Omega_T)$ for every $1<p<2-\frac{n}{n+1}$. Moreover, $\nu_s=0$, and hence the gradient of $\gamma$ does not appear, so the Lipschitz condition is completely redundant. Thus taking into account \cref{preli1} and \cref{convergenceinL1}, we get the same existence results.
\subsection*{Proof of \cref{exis2}} 
Let  $\gamma(x,t)\equiv \gamma$ be a constant, $u_0\in L^{\gamma+1}(\Omega)$ and $\mu\in L^{\gamma+1}\left(0,T;L^{\frac{n(\gamma+1)}{n+2\gamma}}\Omega)\right)$ or $\mu\in L^r(\Omega_T)$ where $r=\frac{(n+2)(\gamma+1)}{n+2(\gamma+1)}$ is a non-negative function in $\Omega_T$. Define \begin{equation}{\label{qqq}}
    q=\begin{cases}
    2, \text{ if } \gamma\geq 1\\ \frac{(\gamma+1)(n+2)}{(n+\gamma+1)}\text{ if } \gamma< 1.\end{cases}
\end{equation}As, $\nu \in \overline{\mathcal{M}}_0^q(\Omega)$, is a non-negative bounded Radon measure on $\Omega_T$ and is non-singular with respect to the Lebesgue measure, so similar to the previous proof, we take the same approximation \cref{appxx} of $\nu$. That is take

\begin{equation*} 0\leq \nu_k=T_k(h)+f_k-\operatorname{div}\left(G_k\right), \left\|\nu_k\right\|_{L^1(\Omega_T)}\leq C\|\nu\|_{\mathcal{M}(\Omega_T)},\end{equation*} where $0\leq h\in L^1(\Omega_T)$, $f_k \in C_c^{\infty}(\Omega_T)$ is a sequence of functions which converges to $f$ weakly in $L^1(\Omega)$, $G_k \in C_c^{\infty}(\Omega_T)$ is a sequence of functions which converges to $G$ strongly in $\left(L^{q^{\prime}}(\Omega_T)\right)^n$. Thus note that $\nu_k\in L^\infty(\Omega_T)$ is a sequence of non-negative functions which is bounded in $L^1(\Omega_T)$ and also converges to $\nu$ in narrow topology of measures. Further, take the approximation $\mu_k=T_k(\mu)$ and consider the approximated problem for each $k\in\mathbb{N}$
  \begin{equation}{\label{aproxproblem123}}
    \begin{array}{c}
         (u_k)_t-\Delta u_k+(-\Delta)^s u_k=\frac{T_k(h)+f_k-\operatorname{div}(G_k)}{\left(u_k+\frac{1}{k}\right)^{\gamma}} +\mu_k \text { in } \Omega_T, \smallskip\\u_k>0 \text{ in }\Omega_T,\quad u_k=0  \text { in }(\mathbb{R}^n \backslash \Omega) \times(0, T), \smallskip\\ u_k(x, 0)=T_k(u_0)(x) \text { in } \Omega ;
    \end{array}
\end{equation} and let $u_k$ denotes the solution to \cref{aproxproblem123}, which exists by \cref{exisforaprox}. Further, $0\leq u_k  \in L^{\infty}(\Omega_T)$ and $(u_k)_t \in L^{2}(0, T ; W^{-1, 2}(\Omega))$. In particular $u_k\in C([0,T],L^2(\Omega))$. Also, as $\nu$ is non-singular with respect to the Lebesgue measure, then 
    for each $t_0>0$, $\omega\subset\subset\Omega$ and $\forall k$, it holds \begin{equation}{\label{greater11}}
        u_k\geq C\quad \text{ in }\omega\times[t_0,T),
    \end{equation} where $C\equiv C(\omega,t_0,n,s)$ is a constant independent of $k$.
Thus, for every $\phi \in C_c^{\infty}(\Omega_T)$, noting \cref{first}, we have
\begin{equation}{\label{eee1}}
\begin{array}{l}
   -  \int_0^T\int_\Omega u_k \phi_tdx d t+\int_0^T\int_{\Omega}\nabla u_k\cdot \nabla \phi\, d x d t+ \int_0^T \int_{\mathbb{R}^{n}}\int_{\mathbb{R}^{n}} \frac{(u_k(x, t)-u_k(y, t))(\phi(x, t)-\phi(y, t))}{|x-y|^{n+2 s}} d x d y d t \smallskip\\ = \int_0^T\int_{\Omega} \frac{T_k(h)}{\left(u_k+\frac{1}{k}\right)^{\gamma}} \phi \, dxdt+\int_0^T\int_{\Omega} \frac{f_k}{\left(u_k+\frac{1}{k}\right)^{\gamma}} \phi \, dxdt +\int_0^T\int_{\Omega} G_k\cdot\nabla \left(\frac{\phi}{\left(u_k+\frac{1}{k}\right)^{\gamma}} \right) dxdt+\int_0^T\int_{\Omega} \mu_k\phi \,
d x d t .
\end{array}
\end{equation}
\textbf{Case 1: $\gamma<1$}\\
Then recall the definition of $q$ from \cref{qqq}. By \cref{preli2}, we have $\{u_k\}_{k \in \mathbb{N}}$ is bounded in $L^q(0,T;W_0^{1, q}(\Omega))$. Consequently, there exists $u \in L^q(0,T;W_0^{1, q}(\Omega))$ such that up to a sub-sequence, $u_k \rightharpoonup u$ weakly in $L^q(0,T;W_0^{1, q}(\Omega))$. Moreover by \cref{convergenceinL1}, $u_k \rightarrow u$ stongly in $L^1_{\mathrm{loc}}(0,T;L^1_{\mathrm{loc}}(\Omega))$ and pointwise a.e. in $\Omega_T$.
Except the last two integral in \cref{eee1}, by using similar justification as in the proof of \cref{exis1}, one can pass to the limit. Now, for the second last integral, we observe that
for every $\phi \in C_c^{\infty}(\Omega_T)$,
\begin{equation}{\label{eee8}}
    \begin{array}{rcl}
    \int_0^T\int_{\Omega} G_k\cdot\nabla \left(\frac{\phi}{\left(u_k+\frac{1}{k}\right)^{\gamma}} \right) dxdt&=&\int_0^T\int_{\Omega}  \frac{G_k\cdot\nabla\phi}{\left(u_k+\frac{1}{k}\right)^{\gamma}}  dxdt-\gamma\int_0^T\int_{\Omega}\frac{G_k\cdot\nabla u_k}{\left(u_k+\frac{1}{k}\right)^{\gamma+1}}\phi\,  dxdt.
    \end{array}
\end{equation}
The fact $G_k\rightarrow G$ strongly in $\left(L^{q^{\prime}}(\Omega_T)\right)^n$ implies $G_k\rightarrow G$ strongly in $\left(L^{1}(\Omega_T)\right)^n$. Also by \cref{greater11},
\begin{equation*}
    \left| \frac{G_k\cdot\nabla\phi}{\left(u_k+\frac{1}{k}\right)^{\gamma}}\right|\leq C |G_k|\text{ in }\operatorname{supp}\phi.
\end{equation*} Therefore, by Lebesgue dominated convergence theorem, it holds 
\begin{equation}{\label{eee9}}
     \lim _{k \rightarrow \infty}\int_0^T\int_{\Omega}  \frac{G_k\cdot\nabla\phi}{\left(u_k+\frac{1}{k}\right)^{\gamma}}  dxdt=\int_0^T\int_{\Omega}  \frac{G\cdot\nabla\phi}{u^{\gamma}}  dxdt.
\end{equation}
Since $\frac{\phi\nabla u_k}{\left(u_k+\frac{1}{k}\right)^{\gamma+1}}  \rightharpoonup \frac{\phi\nabla u}{u^{\gamma+1}} $ weakly in $\left(L^p(\Omega_T)\right)^n$ and $G_k \rightarrow G$ strongly in $\left(L^{p^{\prime}}(\Omega_T)\right)^n$, we obtain
\begin{equation}{\label{eee10}}
   \lim _{k \rightarrow \infty} \gamma \int_0^T\int_{\Omega}  \frac{G_k\cdot\nabla u_k}{\left(u_k+\frac{1}{k}\right)^{\gamma+1}}\phi\,dxdt=\gamma\int_0^T\int_{\Omega}  \frac{G\cdot\nabla u}{u^{\gamma+1}}\phi\,dxdt
\end{equation}
Combining \cref{eee8,eee9,eee10}, we deduce
\begin{equation}{\label{eee11}}
    \lim_{k\rightarrow\infty}\int_0^T\int_{\Omega} G_k\cdot\nabla \left(\frac{\phi}{\left(u_k+\frac{1}{k}\right)^{\gamma}} \right) dxdt=\int_0^T\int_{\Omega} G\cdot\nabla \left(\frac{\phi}{u^{\gamma}} \right) dxdt.
\end{equation}
Also as $\mu_k\to \mu$ in $L^1(\Omega_T)$, we have 
\begin{equation}{\label{eee12}}
     \lim_{k\rightarrow\infty}\int_0^T\int_{\Omega} \mu_k\phi \,dxdt=\int_0^T\int_{\Omega} \mu\phi\,dxdt .
\end{equation}
Taking $k \rightarrow \infty$ in both sides of the identity \cref{eee1} and using \cref{eee11,eee12}, we obtain

\begin{equation*}{\label{eeee13}}
\begin{array}{l}
    -  \int_0^T\int_\Omega u \phi_tdx d t+\int_0^T\int_{\Omega}\nabla u\cdot \nabla \phi\, d x d t+ \int_0^T \int_{\mathbb{R}^{n}}\int_{\mathbb{R}^{n}} \frac{(u(x, t)-u(y, t))(\phi(x, t)-\phi(y, t))}{|x-y|^{n+2 s}} d x d y d t \smallskip\\ = \int_0^T\int_{\Omega} \frac{h}{u^{\gamma}} \phi \, dxdt+\int_0^T\int_{\Omega} \frac{f}{u^{\gamma}} \phi \, dxdt +\int_0^T\int_{\Omega} G\cdot\nabla \left(\frac{\phi}{u^{\gamma}} \right) dxdt+\int_0^T\int_{\Omega} \phi \mu\,dxdt=\iint_{\Omega_T}\frac{\phi}{u^{\gamma}}d\nu+\iint_{\Omega_T}{\phi}\,d\mu.
\end{array}
\end{equation*}
Hence, $u \in L^q(0,T;W_0^{1, q}(\Omega))$ is a solution of the equation \cref{mainproblem}. Moreover by \cref{preli2}, \cref{convergenceinL1}, by Fatou's lemma, we have $u\in L^\infty(0,T;L^1(\Omega))$ if $\gamma\in(0,1)$ and $u\in L^\infty(0,T;L^2(\Omega))$ if $\gamma=1$. 
\smallskip\\\textbf{Case 2: $\gamma>1$}\\
In this case $q=2$. Then by \cref{preli2}, the sequence $\{u_k\}_{k \in \mathbb{N}}$ is uniformly bounded in $L^2(t_0,T;W_{_{\mathrm{loc}}}^{1,2}(\Omega))\cap L^\infty(0,T;L^{\gamma+1}(\Omega))$ for each $0<t_0<T$, such that $\left\{u_k^{\frac{\gamma+1}{2}}\right\}_{ k\in \mathbb{N}}$ is uniformly bounded in $L^2(0,T;W_0^{1,2}(\Omega))$. Also using \cref{convergenceinL1}, there exists $u \in L^2(t_0,T;W_{_{\mathrm{loc}}}^{1,2}(\Omega))\cap L^\infty(0,T;L^{\gamma+1}(\Omega))$, for all $0<t_0<T$, such that $u^{\frac{\gamma+1}{2}} \in L^2(0,T;W_0^{1,2}(\Omega))$, and up to a subsequence $u_k^{\frac{\gamma+1}{2}}\rightharpoonup u^{\frac{\gamma+1}{2}}$ weakly in $L^2(0,T:W^{1,2}_0(\Omega))$, $u_k \rightharpoonup u$ weakly in $L^2
(t_0,T;W_{{\mathrm{loc}}}^{1,2}(\Omega))$, $u_k \rightarrow u$ strongly in $L^1_{\mathrm{loc}}(0,T;L^1_{\mathrm{loc}}(\Omega))$ and $u_k\to u$ pointwise a.e. in $\Omega_T$.
\smallskip\\Since $u_k,u\equiv 0$ a.e. in $(\mathbb{R}^n\backslash\Omega)\times(0,T)$, by Fatou's lemma and \cref{embedding2}, we have 
\begin{equation*}
\|u^{\frac{\gamma+1}{2}}\|_{L^2(0,T;W^{s,2}_0(\mathbb{R}^n))}\leq {\lim\inf}_k\|u_k^{\frac{\gamma+1}{2}}\|_{L^2(0,T;W^{s,2}_0(\mathbb{R}^n))} \leq {\lim\inf}_k C\|u_k^{\frac{\gamma+1}{2}}\|_{L^2(0,T;W^{1,2}_0(\Omega))}\leq C.
    \end{equation*}
Thus by a similar argument as in [\citealp{fracvari1}, Theorem 3.6] (see also [\citealp{pap}, Remark 3.4] for a parabolic version), we have 
\begin{equation}{\label{eeeee21}}
    \begin{array}{c}
      \lim _{k \rightarrow \infty}\int_0^T \int_{\mathbb{R}^{n}}\int_{\mathbb{R}^{n}} \frac{(u_k(x, t)-u_k(y, t))(\phi(x, t)-\phi(y, t))}{|x-y|^{n+2 s}} d x d y d t \smallskip\\=\int_0^T \int_{\mathbb{R}^{n}}\int_{\mathbb{R}^{n}} \frac{(u(x, t)-u(y, t))(\phi(x, t)-\phi(y, t))}{|x-y|^{n+2 s}} d x d y d t 
    \end{array}
\end{equation}
Rest of the proof follows similarly like the above case and the proof of \cref{exis1} and letting $k \rightarrow \infty$ on both sides of the identity \cref{eee1}, by \cref{eeeee21}, we obtain
\begin{equation*}{\label{eeeee13}}
\begin{array}{c}
    -  \int_0^T\int_\Omega u \phi_tdx d t+\int_0^T\int_{\Omega}\nabla u\cdot \nabla \phi\, d x d t+ \int_0^T \int_{\mathbb{R}^{n}}\int_{\mathbb{R}^{n}} \frac{(u(x, t)-u(y, t))(\phi(x, t)-\phi(y, t))}{|x-y|^{n+2 s}} d x d y d t \smallskip\\ =\iint_{\Omega_T}\frac{\phi}{u^{\gamma}}d\nu+\iint_{\Omega_T}{\phi\mu}\,dxdt.
\end{array}
\end{equation*}
\begin{remark}{\label{boundarybeha2}}
       Notice that the results in \cref{exis2} and \cref{exis1}- part (i) implies, $u^{\eta}\in L^m(0,T;W^{1,m}_0(\Omega))$ for some $\eta\geq 1$ and $m>1$. Hence one has that, for almost every $t \in(0, T)$, $u(\cdot, t)^{\eta} \in W_0^{1, m}(\Omega)$. If $\eta=1$ then $u(\cdot, t) \in W_0^{1, m}(\Omega)$ is known to imply
       \begin{equation}{\label{middd}}
           \lim _{\epsilon \rightarrow 0} \frac{1}{\epsilon} \int_{\Omega_\epsilon} u(x,t)dx=0 \text{ for a.e. }t \in(0, T),
       \end{equation} where $\Omega_\epsilon=\{x\in\Omega:\operatorname{dist}\{x,\partial\Omega\}<\epsilon\}$. Otherwise, if $\eta>1$, one can write
\begin{equation*}
    \frac{1}{\epsilon} \int_{\Omega_\epsilon}u (x,t)dx\leq \frac{1}{\epsilon}\left(\int_{\Omega_\epsilon} u^{\eta}(x,t)dx\right)^{\frac{1}{\eta}}\left|\Omega_\epsilon\right|^{\frac{1}{\eta^\prime}}=\left(\frac{\left|\Omega_\epsilon\right|}{\epsilon}\right)^{\frac{1}{\eta^\prime}}\left(\frac{1}{\epsilon} \int_{\Omega_\epsilon} u^{\eta}(x,t)dx\right)^{\frac{1}{\eta}},
\end{equation*}
and taking $\epsilon \rightarrow 0$ in the previous one has \cref{middd}. This also gives one notion of boundary behavior for the solutions. \\Finally for \cref{exis1}- part (ii), we note from \cref{preli1}, that $G_l(u)\in L^q(0,T;W^{1,q}_0(\Omega))$, for all $l>0$ and some $q>1$. Merging this with \cref{boundarybeha}, we get \cref{middd}. 
    \end{remark}
\section{A priori estimates for regularity results}{\label{sec5}}
    Unless otherwise mentioned, we take $u_0\in L^\infty(\Omega)$ and $\nu \in L^{r_1}(0,T;L^{q_1}(\Omega))$, $ \mu\in L^{r_2}(0,T;L^{q_2}(\Omega))$ for some $r_1,r_2, q_1,q_2 \geq 1$ be two non-negative functions.
\\By \cref{exisforaprox}, for each fixed $k\in \mathbb{N}$, we have existence of a unique non-negative solution $u_k\in  L^2(0,T;W^{1,2}_0(\Omega))\cap L^\infty(\Omega_T)$ such that $(u_k)_t\in L^2(0,T;W^{-1,2}(\Omega))$ to the problem
     \begin{equation}{\label{aproxproblembdd}}
    \begin{array}{c}
         (u_k)_t-\Delta u_k+(-\Delta)^s u_k=\frac{T_k(\nu)}{\left(u_k+\frac{1}{k}\right)^{\gamma(x,t)}} +T_k(\mu) \text { in } \Omega_T, \\u_k>0 \text{ in }\Omega_T,\quad u_k=0  \text { in }(\mathbb{R}^n \backslash \Omega) \times(0, T), \\ u_k(\cdot, 0)=T_k(u_0) \text { in } \Omega .
    \end{array}
\end{equation}Following the same proof of \cref{exisforaprox}, we have for each fix $k\in \mathbb{N}$, the existence of a unique non-negative weak solution 
$v_k\in  L^2(0,T;W^{1,2}_0(\Omega))\cap L^\infty(\Omega_T)$ such that $(v_k)_t\in L^2(0,T;W^{-1,2}(\Omega))$ to the problem
     \begin{equation}{\label{aproxproblembdd2}}
    \begin{array}{c}
         (v_k)_t-\Delta v_k+(-\Delta)^s v_k=\frac{T_k(\nu)}{\left(v_k+\frac{1}{k}\right)^{\gamma(x,t)}}  \text { in } \Omega_T, \smallskip\\v_k>0 \text{ in }\Omega_T,\quad v_k=0  \text { in }(\mathbb{R}^n \backslash \Omega) \times(0, T), \smallskip\\ v_k(\cdot, 0)=T_k(u_0)\text { in } \Omega ;
    \end{array}\end{equation}and by \cref{lemma1}, the existence of a unique non-negative weak solution 
$w_k\in  L^2(0,T;W^{1,2}_0(\Omega))\cap L^\infty(\Omega_T)$ such that $(w_k)_t\in L^2(0,T;W^{-1,2}(\Omega))$ to the problem

     \begin{equation}{\label{aproxproblembdd3}}
    \begin{array}{c}
         (w_k)_t-\Delta w_k+(-\Delta)^s w_k=T_k(\mu)  \text { in } \Omega_T, \smallskip\\w_k>0 \text{ in }\Omega_T,\quad w_k=0  \text { in }(\mathbb{R}^n \backslash \Omega) \times(0, T), \smallskip\\ w_k(\cdot, 0)=0\text { in } \Omega .\end{array}\end{equation}
    \begin{lemma}{\label{bounded1}}
      Let $k\in \mathbb{N}$ and let $u_k,v_k,w_k\in  L^2(0,T;W^{1,2}_0(\Omega))\cap L^\infty(\Omega_T)$ such that $(u_k)_t, (v_k)_t, (w_k)_t\in  L^2(0,T;W^{-1,2}(\Omega)) $ are the solutions to the problems \cref{aproxproblembdd}, \cref{aproxproblembdd2} and \cref{aproxproblembdd3} respectively.  If $\{v_k\}_k$ and $\{w_k\}_k$ are uniformly bounded in $L^{a_1}(0,T; L^{b_1}(\Omega))$ and $L^{a_2}(0,T; L^{b_2}(\Omega))$, respectively for some $1\leq a_1,b_1,a_2,b_2\leq \infty$, then $\{u_k\}_k$ is uniformly bounded in $L^{c}(0,T; L^{b}(\Omega))$, where $c=\min\{a_1,a_2\}$ and $d=\min\{b_1,b_2\}$.
    \end{lemma}
\begin{proof}  The function $z_k=(u_k-v_k-w_k)\in L^2(0,T;W^{1,2}_0(\Omega))\cap L^\infty(\Omega_T)$ with $(z_k)_t\in  L^2(0,T;W^{-1,2}(\Omega))$ satisfies the equation  
         \begin{equation}{\label{aproxproblembdd4}}
    \begin{array}{c}
         (z_k)_t-\Delta z_k+(-\Delta)^s z_k=T_k(\nu)\left(\frac{1}{\left(u_k+\frac{1}{k}\right)^{\gamma(x,t)}}-\frac{1}{\left(v_k+\frac{1}{k}\right)^{\gamma(x,t)}}\right)  \text { in } \Omega_T, \smallskip\\ z_k=0  \text { in }(\mathbb{R}^n \backslash \Omega) \times(0, T), \smallskip\\ z_k(\cdot, 0)=0\text { in } \Omega ;
    \end{array}\end{equation}
       Take $(z_k)_+$ as a test function in \cref{aproxproblembdd4} to get \begin{equation}{\label{eqn1111}}
    \begin{array}{l}
        \quad  \int_0^T\langle (z_k)_t, (z_k)_+\rangle d t+\int_{\Omega_T}|\nabla (z_k)_+|^2 d x d t\\\quad+ \underbrace{\int_0^T \iint_{\mathbb{R}^{2n}} \frac{(z_k(x, t)-z_k(y, t))((z_k)_+(x,t)-(z_k)_+(y,t))}{|x-y|^{n+2 s}} d x d y d t}_{\geq 0}\\=\int_{\Omega_T}T_k(\nu)\left(\frac{1}{\left(u_k+\frac{1}{k}\right)^{\gamma(x,t)}}-\frac{1}{\left(v_k+\frac{1}{k}\right)^{\gamma(x,t)}}\right)     (z_k)_+dxdt\leq 0.
    \end{array}
\end{equation} As $z_k(\cdot,0)=0$ and $z_k\in C([0,T],L^2(\Omega))$, so 
\begin{equation*}
    \begin{array}{rcl}
         \int_0^T\left\langle (z_k)_t, (z_k)_+\right\rangle d t
         &=&\int_0^{T}\int_\Omega  \frac{\partial z_k}{\partial t} (z_k)_+d x d t
        = \int_\Omega\int_{z_k(x,0)}^{z_k(x,T)}\sigma_+d\sigma d x 
         =\frac{1}{2}\int_\Omega ((z_k)_+(x,T))^2d x\geq 0.
    \end{array}
\end{equation*} Therefore we obtain from \cref{eqn1111} that 
$ \int_{\Omega_T}|\nabla (z_k)_+|^2 d x d t=0$,
 which by Poincar\'e inequality implies $\int_{\Omega_T}|(z_k)_+|^2 d x d t=0$, which gives $u_k\leq v_k+w_k$ in $\Omega_T$ and hence $\{u_k\}_k$ is uniformly bounded in $L^{c}(0,T; L^{b}(\Omega))$.
\end{proof}
The following lemma is for the boundedness of $\{v_k\}_k$ for $\gamma$ being a constant.
\begin{lemma}{\label{bounded2}}
(1) Assume that $\nu \in L^r(0, T ; L^q(\Omega))$ with $r, q$ satisfying
\begin{equation*}
\frac{1}{r}+\frac{n}{2 q }<1  , 
\end{equation*}
and let $u_{0} \in L^{\infty}(\Omega)$ and $\gamma$ be a constant. Then there exists a positive constant $c>0$ such that the unique solution of \cref{aproxproblembdd2} satisfies
\begin{equation*}
    \left\|v_k(x, t)\right\|_{L^{\infty}\left(\Omega_T\right)} \leq c
    .
\end{equation*}
(2) Assume $u_{0}\in L^\infty(\Omega)$, $\gamma$ be a constant and $\nu \in L^r\left(0, T ; L^q(\Omega)\right)$, with $r>1, q>1$ satisfy
\begin{equation*}
    1<\frac{1}{r}+\frac{n}{2 q }.
\end{equation*}
Further, assume that for $\gamma<1$,\\
$i)$ if $\frac{1}{r}<\frac{n}{n-2} \frac{1}{q}-\frac{2}{n-2}$,  then $q> \left(\frac{2^*}{1-\gamma}\right)^\prime$, and\\
$ii)$ if $\frac{1}{r}\geq\frac{n}{n-2} \frac{1}{q}-\frac{2}{n-2}$, then $r>\frac{2}{1+\gamma}$.\\
Then there exists a positive constant $c$ such that the sequence of solutions of \cref{aproxproblembdd2} satisfies
\begin{equation*}
\left\|v_k\right\|_{L^{\infty}\left(0, T ; L^{2 \sigma}(\Omega)\right)}+\left\|v_k\right\|_{L^{2 \sigma}\left(0, T ; L^{2^* \sigma}\right)} \leq c,
\end{equation*}
where
\begin{equation*}
\sigma=\left\{\begin{array}{lll}
\frac{q(n-2 )(\gamma+1)}{2(n-2 q )} & \text { if } \frac{1}{r}<\frac{n}{n-2} \frac{1}{q}-\frac{2}{n-2}, \smallskip\\
\frac{q r n(\gamma+1)}{2(n r+2q-2 q r)} & \text { if } \frac{1}{r} \geq \frac{n}{n-2} \frac{1}{q}-\frac{2}{n-2}.
\end{array}\right.    
\end{equation*}
\end{lemma}
\begin{proof}
(1) Let $v$ be any measurable function and define \begin{equation*}A_m(v)=\left\{(x, t) \in \Omega_T: v(x, t)>m\right\}, \text{ and for a.e. } t \in(0, T), \,A_m^t(v)=\{x \in \Omega: v(x, t)>m\}.\end{equation*}
We use the idea of the classical proof by D.G. Aronson and J. Serrin, which is to prove a uniform $L^{\infty}$ bound for $v_k$ in $\Omega \times(0, \tau)$, for a positive (small) $\tau$ (to be specified later), and then to iterate such an estimate. We consider \cref{aproxproblembdd2} and take $(v_k-m)_+\chi_{(0,\tilde{t})}\in L^2(0,T;W^{1,2}_0(\Omega))$, for $m>0,\tilde{t}\in(0,\tau),\tau<T$,
as a test function to obtain
\begin{equation}{\label{testfunc}}
\begin{array}{c}
\int_{\Omega} \varphi_m\left(v_k(x, \tilde{t})\right) d x+\int_0^{\tilde{t}}\int_\Omega |\nabla (v_k-m)_+ |^2d x d t\\+\underbrace{\int_0^{\tilde{t} }\int_{\mathbb{R}^n}\int_{\mathbb{R}^n} \frac{\left(v_k(x, t)-v_k(y, t)\right)\left((v_k-m)_+(x, t)-(v_k-m)_+(y, t)\right)}{|x-y|^{n+2 s}} d x d y d t }_{\geq 0}\\
\leq\int_0^\tau \int_{\Omega} \frac{\nu_k (v_k-m)_+}{(v_k+\frac{1}{k})^\gamma} d x d t+\int_{\Omega} \varphi_m\left(v_k(x, 0)\right) d x,
\end{array}
\end{equation}
where \begin{equation*}\varphi_m(\rho)=\int_0^\rho (\sigma-m)_+
d \sigma= \frac{(\rho-m)_+^2}{2}.\end{equation*} Choose $m >\operatorname{max}\{\left\|u_ 0\right\|_{L^{\infty}(\Omega)},1\}$, in order to neglect the last term. Note that the nonlocal integral in non-negative, 
and take supremum over $\tilde{t}\in(0,\tau]$ in \cref{testfunc}, to get
\begin{equation}{\label{iter}}
\begin{array}{c}
\|(v_k-m)_+\|^2_{L^\infty(0,\tau;L^2(\Omega))}+\|(v_k-m)_+\|^2_{L^2(0,\tau;W^{1,2}_0(\Omega))} 

\leq \int_0^\tau \int_{A^t_{m,k}} \nu(v_k-m)_+ d x d t.
\end{array}
\end{equation}
Here, to deal with the singularity, we have used $m\geq 1$ and hence $v_k> 1$ in $A^t_{m,k}$. The subscript $k$ in $A^t_{m,k}$ denotes that we are considering $v_k$.
Now, the term of the right-hand side above can be estimated as,
\begin{equation}{\label{rhs}}
    \int_0^\tau \int_{A_{m,k}^t} \nu(v_k-m)_+ d x d t \leq \int_0^\tau \int_{A_{m,k}^t} \nu(v_k-m)_+ ^2 d x d t+\int_0^\tau \int_{A_{m,k}^t} \nu d x d t .
\end{equation}
We now study each member present in the right-hand side of \cref{rhs}. We first define the followings,
\begin{equation*}
\bar{r}=2 r^{\prime}, \bar{q}=2 q^{\prime}, \eta=\frac{2 \eta_1}{n},\hat{r}=\bar{r}(1+\eta), \hat{q}=\bar{q}(1+\eta),
\end{equation*}
where $\eta_1=1-\frac{1}{r}-\frac{n}{2 q}$. Clearly $\eta_1\in(0,1)$ as $0<\frac{1}{r}+\frac{n}{2q}<1$. Further, simple calculation yields that \begin{equation*}\frac{1}{\hat{r}}+\frac{n}{2 \hat{q} }=\frac{n}{4 }.\end{equation*} Now, applying H\"older inequality repeatedly, we estimate the first term as
\begin{equation*}
\begin{array}{rcl}
\int_0^\tau \int_{A_{m,k}^t} \nu (v_k-m)_+^2 d x d t &\leq& \int_0^\tau \left(\int_{A_{m,k}^t}| \nu (x,t)|^q d x\right )^{\frac{1}{q}}\left(\int_{A_{m,k}^t}(v_k-m)_+^{2q^\prime} d x\right)^{\frac{1}{q^\prime}} d t \\
&\leq& \left(\int_0^\tau \left(\int_{A_{m,k}^t}| \nu (x,t)|^q d x\right )^{\frac{r}{q}}d t\right)^{\frac{1}{r}}\left(\int_0^\tau\left(\int_{A_{m,k}^t}(v_k-m)_+^{2q^\prime} d x\right)^{\frac{r^\prime}{q^\prime}} d t \right)^{\frac{1}{r^\prime}}.
\end{array}
\end{equation*}
As again by H\"older inequality with exponent $(1+\eta)$, we have
\begin{equation*}
    \begin{array}{rcl}
        \left(\int_0^\tau\left(\int_{A_{m,k}^t}(v_k-m)_+^{2q^\prime} d x\right)^{\frac{r^\prime}{q^\prime}} d t \right)^{\frac{1}{r^\prime}}&\leq &\left(\int_0^\tau\left(\int_{A_{m,k}^t}(v_k-m)_+^{2q^\prime(1+\eta)}d x\right)^{\frac{r^\prime}{q^\prime(1+\eta)}}|A^t_{m,k}|^{\frac{r^\prime\eta}{q^\prime(1+\eta)}} d t \right)^{\frac{1}{r^\prime}}\\&=&\left(\int_0^\tau\left(\int_{A_{m,k}^t}(v_k-m)_+^{\hat{q}}d x\right)^{\frac{2r^\prime}{\hat{q}}}|A^t_{m,k}|^{\frac{2r^\prime\eta}{\hat{q}}} d t \right)^{\frac{1}{r^\prime}}
        \\&\leq & \left(\int_0^\tau\left(\int_{A_{m,k}^t}(v_k-m)_+^{\hat{q}}d x\right)^{\frac{\hat{r}}{\hat{q}}}d t \right)^{\frac{1}{r^\prime(1+\eta)}}\left(\int_0^\tau |A^t_{m,k}|^{\frac{\hat{r}}{\hat{q}}}\right)^{\frac{\eta}{r^\prime(1+\eta)}} ,
    \end{array}
\end{equation*}
therefore

\begin{equation*}
    \int_0^\tau \int_{A_{m,k}^t} \nu(v_k-m)_+^2 d x d t \leq \|\nu\|_{L^r(0,T;L^q(\Omega))}\|(v_k-m)_+\|^2_{L^{\hat{r}}(0,\tau;L^{\hat{q}}(\Omega))}\mu_k(m)^{\frac{2 \eta}{\hat{r}}},
\end{equation*}
where $\mu_k(m)=\int_0^\tau\left|A_{m,k}^t\right|^{\frac{\hat{r}}{\hat{q}}} d t$. We now use \cref{gagliardo} for $(v_k-m)_+$ to get
\begin{equation}{\label{firstt}}
    \begin{array}{rcl}
\int_0^\tau \int_{A_{m,k}^t} \nu (v_k-m)_+^2 d x d t &\leq&  c\|\nu\|_{L^r(0,T;L^q(\Omega))} \mu_k(m)^{\frac{2 \eta}{\hat{r}}}\left(\int_0^\tau\left\|(v_k-m)_+\right\|_{L^2(\Omega)}^{(1-\theta) \hat{r}}\left\|\nabla (v_k-m)_+\right\|_{L^2\left(\Omega\right)}^{\hat{r} \theta} d t\right)^{\frac{2}{\hat{r}}} \\&
\leq& c\|\nu\|
\mu_k(m)^{\frac{2 \eta}{\hat{r}}}\bigg[\left\|(v_k-m)_+\right\|_{L^{\infty}\left(0, \tau ; L^2(\Omega)\right)}^2
+\left\| (v_k-m)_+\right\|_{L^2(0, \tau ; W^{1,2}_0(\Omega))}^2\bigg].
\end{array}
\end{equation}
where $\hat{r}\theta=2$, and $c$ is a constant independent of the choice of $k$ and $m$. Note that, we applied Young's inequality with exponents $\frac{\hat{r}}{2}$ and its conjugate.\\
On the other hand, the second term on the right-hand side in \cref{rhs} can be estimated by Hölder inequality as
\begin{equation}{\label{second}}
\begin{array}{rcl}
     \int_0^\tau \int_{A_{m,k}^t} \nu d x d t &\leq&\int_0^\tau \left(\int_{A_{m,k}^t} |\nu(x,t)|^qd x\right)^{\frac{1}{q}} |A^t_{m,k}|^{\frac{1}{q^\prime}}d t\smallskip \\&\leq&\left(\int_0^\tau \left(\int_{A_{m,k}^t} |\nu(x,t)|^qd x\right)^{\frac{r}{q}}dt\right)^{\frac{1}{r}} \left(\int_0^\tau|A^t_{m,k}|^{\frac{r^\prime}{q^\prime}}d t\right)^{\frac{1}{r^\prime}}\smallskip\\ &\leq&\|\nu\|_{L^r(0,T;L^q(\Omega))}\mu_k(m)^{\frac{2(1+\eta)}{\hat{r}}} .
\end{array}  
\end{equation}
Denoting
\begin{equation*}
||| (v_k-m)_+|||^2=\| (v_k-m)_+\|_{L^{\infty}(0, \tau ; L^2(\Omega))}^2+\|(v_k-m)_+ \|_{L^2(0, \tau ; W^{1,2}_0(\Omega))}^2,    
\end{equation*}
and using \cref{firstt}, \cref{second} in \cref{rhs}, we get from \cref{iter},
\begin{equation*}
||| (v_k-m)_+|||^2\leq c\|\nu\|_{L^r(0,T;L^q(\Omega))}\left[\mu_k(m)^{\frac{2 \eta}{\hat{r}}}||| (v_k-m)_+|||^2+\mu_k(m)^{\frac{2(1+\eta)}{\hat{r}}}\right],
\end{equation*}
where $c$ is a constant which does not depend on $k$ and $m$. Note that $\mu_k(m) \leq \tau|\Omega|^{\frac{\hat{r}}{\hat{q}}}$, for all $k\in\mathbb{N}$ and $m$, so that we can fix $\tau$, independent of $u_k$ and $m$, suitable small in such a way that
$c\mu_k(m)^{\frac{2 \eta}{\hat{r}}}\|\nu\|_{L^r(0,T;L^q(\Omega))}=\frac{1}{2}
$ and use again \cref{gagliardo} (see also \cref{firstt}), to deduce that
\begin{equation}{\label{finalboun}}
\left\|(v_k-m)_+\right\|_{L^{\hat{r}}\left(0, \tau ; L^{\hat{q}}(\Omega)\right)}^2 \leq c||| (v_k-m)_+|||^2\leq c\|\nu\|_{L^r(0,T;L^q(\Omega))} \mu_k(m)^{\frac{2(1+\eta)}{\hat{r}}} .  
\end{equation}
Consider $l>m>\operatorname{max}\{\left\|u_ 0\right\|_{L^{\infty}(\Omega)},1\}:=m_0$. Then $A^t_{l,k}\subset A^t_{m,k}$, and using \cref{finalboun}, we get
\begin{equation*}
\begin{array}{rcl}
(l-m)\mu_k(l)^{\frac{1}{\hat{r}}}
=\left(\int_0^\tau\left((l-m)^{\hat{q}}\left|A_{l,k}^t\right|\right)^{\frac{\hat{r}}{\hat{q}}} d t\right)^{\frac{1}{\hat{r}}}&\leq& \left(\int_0^\tau\left(\int_{A_{l,k}^t}(v_k-m)_+^{\hat{q}}d x\right)^{\frac{\hat{r}}{\hat{q}}}d t \right)^{\frac{1}{\hat{r}}}\\&\leq &\left(\int_0^\tau\left(\int_{A_{m,k}^t}(v_k-m)_+^{\hat{q}}d x\right)^{\frac{\hat{r}}{\hat{q}}}d t \right)^{\frac{1}{\hat{r}}}
\leq c\mu_k(m)^{\frac{1+\eta}{\hat{r}}} .    
\end{array}
\end{equation*}
Therefore for all $l>m>m_0$, we have
\begin{equation*}
   \mu_k(l) \leq\frac{c}{(l-m)^{\hat{r}}} \mu_k(m)^{1+\eta} ,
\end{equation*}
where c is a constant independent of $k$.
Now applying [\citealp{Iterative}, Lemma B.$1$], we conclude that
$\mu_k(m_0+d_k)=0,$ where $d_k=c[\mu_k(m_0)]^\eta2^{\hat{r}(1+\eta)/\eta}$. As for each $k$, $d_k\leq c \left(T|\Omega|^{\frac{\hat{r}}{\hat{q}}}\right)^\eta$, we get
\begin{equation*}
\left\|v_k\right\|_{L^{\infty}(\Omega \times[0, \tau])} \leq d .
\end{equation*}
We can iterate this procedure in the sets $\Omega \times[\tau, 2 \tau], \cdots, \Omega \times[j \tau, T]$, where $T-j \tau \leqslant \tau$ to conclude that
\begin{equation*}
    \left\|v_k\right\|_{L^{\infty}\left(\Omega_T\right)} \leq C \quad \text { uniformly in } k \in \mathbb{N} .
\end{equation*}
(2)
For $\gamma\geq 1$, we choose $\delta>\frac{1+\gamma}{2}$ and for $\gamma<1$, we choose $\delta>1$, and take $v_k^{2\delta-1}\chi_{(0,\tau)},0<\tau\leq T$ as test function in \cref{aproxproblembdd2}. We note that, such a test function is admissible as $v_k\in L^\infty(\Omega_T)$.
We get $\forall \tau\leq T,$
\begin{equation*}
 \begin{array}{c}
 \quad\frac{1}{2\delta} \int_{\Omega} v_k^{2\delta}(x, \tau) d x+\int_0^\tau\int_\Omega \nabla v_k\cdot\nabla v_k^{2\delta-1}d x dt\smallskip\\\quad+ \underbrace{\int_0^\tau \int_{\mathbb{R}^n}\int_{\mathbb{R}^n} \frac{\left(v_k(x, t)-v_k(y,t)\right)\left(v_k^{2\delta-1}(x, t)-v_k^{2\delta-1}(y, t)\right)}{|x-y|^{n+2 s}} d x d y d t}_{\geq 0}
\smallskip\\\leq \int_0^\tau\int_{\Omega} \nu v_k^{2\delta-1-\gamma}d x dt+\frac{1}{2\delta} \int_{\Omega} v_k^{2\delta}(x, 0) d x.
\end{array}
\end{equation*}
Taking supremum over $\tau\in(0,T]$, by Hölder and Sobolev inequalities, we have
\begin{equation}{\label{test1}}
\begin{array}{c}
\sup _{\tau \in[0, T]} \int_{\Omega} v_k^{2 \delta} (x,\tau)d x+\frac{2(2\delta-1)}{ \delta} \lambda \int_0^T\left\|v_k^\delta\right\|_{L^{2^*}(\Omega)}^2 d t  \smallskip\\
\leq 2\delta\left\|\nu\right\|_{L^r\left(0, T ; L^q(\Omega)\right)}\left(\int_0^T\left[\int_{\Omega} v_k^{(2 \delta-1-\gamma) q^{\prime}} d x\right]^{\frac{r^{\prime}}{q^{\prime}}} d t\right)^{\frac{1}{r^{\prime}}} +\frac{1}{2\delta} \int_{\Omega} u_0^{2\delta} d x.
\end{array}
\end{equation}
Note that $\frac{2\delta-1}{\delta}>1$, so we can ignore the constants in left. Denoting by
$
A=\left(\int_0^T\left\|v_k\right\|_{L^{(2 \delta-1-\gamma) q^{\prime}}(\Omega)}^{(2 \delta-1-\gamma) r^{\prime}} d t\right)^{\frac{1}{r^{\prime}}},    
$ and $B=|\Omega|\|u_0\|^{2\delta}_{L^\infty(\Omega)}$,
we get from \cref{test1},
\begin{equation}{\label{res1}}
\sup _{t\in[0, T]} \int_{\Omega} v_k^{2 \delta} d x \leq 2 \delta\left\|\nu\right\|_{L^r\left(0, T ; L^q(\Omega)\right)} A+\frac{1}{2\delta} B,    
\end{equation}
and
\begin{equation}{\label{res2}}
    \int_0^T\left[\int_{\Omega} v_k^{2^*\delta} d x\right]^{\frac{2}{2 ^*}} d t  \leq\frac{c}{\lambda} 2\delta\left\|\nu\right\|_{L^r\left(0, T ; L^q(\Omega)\right)} A +\frac{1}{2\delta} B.
\end{equation}
\smallskip\\
\textbf{Case 1: $1<q<\frac{n r}{n+2 (r-1)}$ \text{ i.e. } $\frac{1}{r}<\frac{n}{n-2} \frac{1}{q}-\frac{2}{n-2}$} \smallskip\\Since, $1<q<\frac{n r}{n+2 (r-1)}$ implies $q<r$ and $\frac{r^{\prime}}{q^{\prime}}<\frac{2}{2^*}$, we apply Hölder inequality with exponent $\frac{2q^{\prime}}{2^*r^{\prime}}$ to get
\begin{equation}{\label{res2.5}}
A \leq T^{\frac{1}{r^{\prime}}-\frac{2^*}{2 q^{\prime}}}\left[\int_0^T\left(\int_{\Omega} v_k^{(2 \delta-1-\gamma) q^{\prime}} d x\right)^{\frac{2}{2 ^*}} d t\right]^{\frac{2^*}{2 q^{\prime}}} .    
\end{equation}
Thus from \cref{res2}, it follows
\begin{equation*}{\label{res3}}
\int_0^T\left[\int_{\Omega} v_k^{2^* \delta} d x\right]^{\frac{2}{2^*}} d t \leq 2\delta cA +\frac{1}{2\delta} B\leq 2\delta c\left[\int_0^T\left(\int_{\Omega} v_k^{(2 \delta-1-\gamma) q^{\prime}} d x\right)^{\frac{2}{2 ^*}} d t\right]^{\frac{2^*}{2 q^{\prime}}} +\frac{1}{2\delta} B.  
\end{equation*}
We now choose $2^* \delta=(2 \delta-1-\gamma) q^{\prime}$, that is, $\delta=\frac{1}{2} \cdot \frac{q(n-2 )(\gamma+1)}{(n-2 q )}$. Note that if $\gamma\geq 1$, then $\delta >\frac{\gamma+1}{2}$ if and only if $q>1$, and for $\gamma<1$, $\delta>1$ if and only if $q>\left(\frac{2^*}{1-\gamma}\right)^{\prime}$.  Moreover, since $q<\frac{n}{2 }$ it follows that $\frac{2^*}{2 q^{\prime}}<1$. Thus we get by Young's inequality with $\ep$,
\begin{equation}{\label{res4}}
   \int_0^T\left(\int_{\Omega} v_k^{2^*\delta} d x\right)^{\frac{2}{2^*}} d t=\int_0^T\left(\int_{\Omega} v_k^{\left(2 \delta-1-\gamma\right) q^{\prime}} d x\right)^{\frac{2}{2^*}} d t \leq  c,
\end{equation}
and therefore from \cref{res1}, \cref{res2.5} and \cref{res4}, it follows
\begin{equation*}
\left\|v_k\right\|_{L^\infty\left(0, T ; L^{2\sigma}(\Omega)\right)} +\left\|v_k\right\|_{L^{2\sigma}\left(0, T ; L^{2^*\sigma}(\Omega)\right)} \leq  c, \text{ with } \sigma=\frac{q(n-2 )(\gamma+1)}{2(n-2 q )}>\frac{q}{2}.
\end{equation*}
\smallskip\\\textbf{Case 2: $\frac{nr}{n+2(r-1)} \leq q<\frac{n}{2 } r^{\prime}$ i.e. $\frac{1}{r}\geq\frac{n}{n-2} \frac{1}{q}-\frac{2}{n-2}$ and $\frac{1}{r}+\frac{n}{2 q }>1$}\smallskip
\\In this case we choose $\delta=\frac{1}{2} \frac{q r n(\gamma+1)}{nr-2 q (r-1)}$. Note that for $\gamma\geq 1$, $\delta >\frac{1+\gamma}{2},$ if and only if $\frac{1}{r}+\frac{n}{2 q }<1+\frac{n}{2}$ which is always satisfied and for $\gamma <1$, $\delta>1$ if and only if $q>\frac{2nr}{nr(\gamma+1)+4(r-1)}$ which is satisfied if $r>\frac{2}{\gamma+1}$. Now by interpolation, we get that $\frac{1}{(2 \delta-1-\gamma) q^{\prime}}=\frac{1-\theta}{2 \delta}+\frac{\theta}{2^* \delta}$, where $\theta=\frac{nq(r-1)}{nr(q-1)+2q(r-1)}\in(0,1]
$. Therefore
\begin{equation*}
A^{r^\prime}=\int_0^T\left\|v_k\right\|_{L^{(2 \delta-1-\gamma) q^{\prime}}(\Omega)}^{r^{\prime}(2 \delta-1-\gamma)} d t \leq c\left\|v_k\right\|_{L^{\infty}\left(0, T ; L^{2 \delta}(\Omega)\right)}^{2 \delta \mu_1} \int_0^T\left(\int_{\Omega} v_k^{2^ * \delta} d x\right)^{\frac{2}{2^*} \mu_2} d t,
\end{equation*}
where $\mu_1=\frac{(1-\theta) r^{\prime}(2 \delta-1-\gamma)}{2 \delta}$ and $\mu_2=\frac{\theta r^{\prime}(2 \delta-1-\gamma)}{2\delta}$. Since $\mu_2\leq 1$, we have that
\begin{equation*}
      A^{r^\prime}= \int_0^T\left\|v_k\right\|_{L^{(2 \delta-1-\gamma) q^{\prime}}(\Omega)}^{r^{\prime}(2 \delta-1-\gamma)} d t \leq c\left\|v_k\right\|_{L^{\infty}\left(0, T ; L^{2 \delta}(\Omega)\right)}^{2 \delta \mu_1} \left(\int_0^T\left[\int_{\Omega} v_k^{2^* \delta} d x\right]^{\frac{2}{2^*}} d t\right)^{\mu_2}\stackrel{\cref{res1},\cref{res2}}{\leq }C(A+B)^{\mu_1+\mu_2},
\end{equation*}
and since $\mu_1+\mu_2 < r^{\prime}$, using Young's inequality, we conclude the following, which will give our final result
\begin{equation*}
    \int_0^T\left(\int_{\Omega} v_k^{(2 \delta-1-\gamma) q^{\prime}} d x\right)^{\frac{r^{\prime}}{q^{\prime}}} d t \leq c.
\end{equation*}
\end{proof}
We now consider $\gamma:\overline{\Omega}_T\to (0,\infty)$ to be a continuous function, and prove the following lemma.
\begin{lemma}{\label{bounded3}} Let $\gamma:\overline{\Omega}_T\to (0,\infty)$ is a continuous function and satisfies the condition $(P_1)$ for some $\gamma^*\geq 1$ and $\delta>0$.\smallskip
\\(1) Assume that $\nu \in L^r(0, T ; L^q(\Omega))$ with $r, q$ satisfying
\begin{equation*}
\frac{1}{r}+\frac{n}{2 q }<1  , 
\end{equation*}
and let $u_{0} \in L^{\infty}(\Omega)$. Then there exists a positive constant $c>0$ such that the unique solution of \cref{aproxproblembdd2} satisfies
\begin{equation*}
    \left\|v_k(x, t)\right\|_{L^{\infty}\left(\Omega_T\right)} \leq c
    .
\end{equation*}
(2) Assume $u_{0}\in L^\infty(\Omega)$, and $\nu \in L^r\left(0, T ; L^q(\Omega)\right)$, with $r>1, q>1$ satisfy
\begin{equation*}
    1<\frac{1}{r}+\frac{n}{2 q }.
\end{equation*}
Further, assume that,\\
$i)$ if $\frac{1}{r}<\frac{n}{n-2} \frac{1}{q}-\frac{2}{n-2}$,  then $q> \frac{n(\gamma^*+1)}{n+2\gamma^*}$, and\\
$ii)$ if $\frac{1}{r}\geq\frac{n}{n-2} \frac{1}{q}-\frac{2}{n-2}$, then $r>{1+\gamma^*}$.\\
Then there exists a positive constant $c$ such that the sequence of solutions of \cref{aproxproblembdd2} satisfies
\begin{equation*}
\left\|v_k\right\|_{L^{\infty}\left(0, T ; L^{2 \sigma}(\Omega)\right)}+\left\|v_k\right\|_{L^{2 \sigma}\left(0, T ; L^{2^* \sigma}\right)} \leq c,
\end{equation*}
where
\begin{equation*}
\sigma=\left\{\begin{array}{lll}
\frac{q(n-2 )}{2(n-2 q )} & \text { if } \frac{1}{r}<\frac{n}{n-2} \frac{1}{q}-\frac{2}{n-2}, \smallskip\\
\frac{q r n
}{2(n r+2q-2 q r)} & \text { if } \frac{1}{r} \geq \frac{n}{n-2} \frac{1}{q}-\frac{2}{n-2}.
\end{array}\right.    
\end{equation*}
\end{lemma}
\begin{proof}
    (1) This part follows exactly similar to \cref{bounded2}, part (1).
    \smallskip\\(2)  We choose $\delta>\frac{1+\gamma^*}{2}$ and take $v_k^{2\delta-1}\chi_{(0,\tau)},0<\tau\leq T$ as test function in \cref{aproxproblembdd2}. We note that, such a test function is admissible as $v_k\in L^\infty(\Omega_T)$.
We get $\forall \tau\leq T,$
\begin{equation}{\label{zzzz}}
 \begin{array}{c}
 \quad\frac{1}{2\delta} \int_{\Omega} v_k^{2\delta}(x, \tau) d x+\int_0^\tau\int_\Omega \nabla v_k\cdot\nabla v_k^{2\delta-1}d x dt\smallskip\\\quad+ \underbrace{\int_0^\tau \int_{\mathbb{R}^n}\int_{\mathbb{R}^n} \frac{\left(v_k(x, t)-v_k(y,t)\right)\left(v_k^{2\delta-1}(x, t)-v_k^{2\delta-1}(y, t)\right)}{|x-y|^{n+2 s}} d x d y d t}_{\geq 0}
\\\leq \int_0^\tau\int_{\Omega} T_k(\nu) \frac{v_k^{2\delta-1}}{\left(v_k+\frac{1}{k}\right)^{\gamma(x,t)}}d x dt+\frac{1}{2\delta} \int_{\Omega} v_k^{2\delta}(x, 0) d x.
\end{array}
\end{equation}Note that as $\nu$ is a function, hence (by \cref{exisforaprox}) 
    for each $t_0>0$ and $\omega\subset\subset\Omega$, $\exists C(\omega,t_0,n,s)$, a constant which does not depend on $k$,  such that $\forall k$, $v_k\geq C(\omega,t_0,n,s)$ in $\omega\times[t_0,T)$. Now using the condition $(P_1)$, we get
\begin{equation*}{\label{imp23332}}
\begin{array}{rcl}
\iint_{\Omega_T} \frac{T_k(\nu)v_k^{2\delta-1}}{\left(v_k+\frac{1}{k}\right)^{\gamma(x,t)}} 
&=&\iint_{(\Omega_T)_\delta} \frac{T_k(\nu)v_k^{2\delta-1}}{\left(v_k+\frac{1}{k}\right)^{\gamma(x,t)}} +\iint_{(\omega_{{T}})_\delta} \frac{T_k(\nu)v_k^{2\delta-1}}{\left(v_k+\frac{1}{k}\right)^{\gamma(x,t)}}  \\&\leq &\iint_{(\Omega_T)_\delta}  \nu v_k^{{2\delta-1}-\gamma(x,t)} +\iint_{(\omega_{\tilde{T}})_\delta}  \frac{\nu v_k^{2\delta-1}}{C_{(\omega_{\tilde{T}})_\delta}^{\gamma(x,t)}} \\
 &\leq &\iint_{(\Omega_T)_\delta\cap\left\{v_k \leq 1\right\}}  \nu v_k^{{2\delta-1}-\gamma(x,t)} +\iint_{(\Omega_T)_\delta\cap\left\{v_k \geq 1\right\}}  \nu v_k^{2\delta-1}  
+\iint_{(\omega_{\tilde{T}})_\delta}  \frac{\nu v_k^{2\delta-1}}{C_{(\omega_{\tilde{T}})_\delta}^{\gamma(x,t)}} \smallskip\\&\leq&\|\nu\|_{L^1(\Omega_T)}+\left(1+\left\|C_{(\omega_{\tilde{T}})_\delta}^{-\gamma(\cdot)}\right\|_{L^{\infty}(\Omega_T)}\right) \iint_{\Omega_T} \nu v_k^{2\delta-1},
\end{array}
\end{equation*}
where $(\omega_{{T}})_\delta=\Omega_T\backslash(\Omega_T)_\delta$. Now, taking supremum over $\tau\in(0,T]$ in \cref{zzzz}, by Hölder and Sobolev inequalities, 
\begin{equation}{\label{test111}}
\begin{array}{rcl}
\sup _{\tau \in[0, T]} \int_{\Omega} v_k^{2 \delta} (x,\tau)d x+\frac{2(2\delta-1)}{ \delta} \lambda \int_0^T\left\|v_k^\delta\right\|_{L^{2^*}(\Omega)}^2 d t  &\leq& C+c \iint_{\Omega_T} \nu v_k^{2\delta-1} d x d t+c\|u_0\|_{L^\infty(\Omega)}^{2\delta}\\
&\leq& C+2\delta\left\|\nu\right\|_{L^r\left(0, T ; L^q(\Omega)\right)}\left(\int_0^T\left[\int_{\Omega} v_k^{(2 \delta-1) q^{\prime}} \right]^{\frac{r^{\prime}}{q^{\prime}}} \right)^{\frac{1}{r^{\prime}}}.
\end{array}
\end{equation}
As $\frac{2\delta-1}{\delta}>1$, one can ignore the constants in left. Denoting 
$
A=\left(\int_0^T\left\|v_k\right\|_{L^{(2 \delta-1) q^{\prime}}(\Omega)}^{(2 \delta-1) r^{\prime}} d t\right)^{\frac{1}{r^{\prime}}},    
$
we get by \cref{test111},
\begin{equation}{\label{res11}}
\sup _{t\in[0, T]} \int_{\Omega} v_k^{2 \delta} d x \leq 2 \delta\left\|\nu\right\|_{L^r\left(0, T ; L^q(\Omega)\right)} A+C,    
\end{equation}
and
\begin{equation}{\label{res22}}
    \int_0^T\left[\int_{\Omega} v_k^{2^*\delta} d x\right]^{\frac{2}{2 ^*}} d t  \leq\frac{c}{\lambda} 2\delta\left\|\nu\right\|_{L^r\left(0, T ; L^q(\Omega)\right)} A +C.
\end{equation}
\smallskip\\
\textbf{Case 1: $1<q<\frac{n r}{n+2 (r-1)}$ \text{ i.e. } $\frac{1}{r}<\frac{n}{n-2} \frac{1}{q}-\frac{2}{n-2}$} \smallskip\\Since, $1<q<\frac{n r}{n+2 (r-1)}$ implies $q<r$ and $\frac{r^{\prime}}{q^{\prime}}<\frac{2}{2^*}$, we apply Hölder inequality with exponent $\frac{2q^{\prime}}{2^*r^{\prime}}$ to get
\begin{equation}{\label{res222.5}}
A \leq T^{\frac{1}{r^{\prime}}-\frac{2^*}{2 q^{\prime}}}\left[\int_0^T\left(\int_{\Omega} v_k^{(2 \delta-1) q^{\prime}} d x\right)^{\frac{2}{2 ^*}} d t\right]^{\frac{2^*}{2 q^{\prime}}} .    
\end{equation}
Thus from \cref{res22}, it follows
\begin{equation*}{\label{res33}}
\int_0^T\left[\int_{\Omega} v_k^{2^* \delta} d x\right]^{\frac{2}{2^*}} d t \leq 2\delta cA +C\leq 2\delta c\left[\int_0^T\left(\int_{\Omega} v_k^{(2 \delta-1) q^{\prime}} d x\right)^{\frac{2}{2 ^*}} d t\right]^{\frac{2^*}{2 q^{\prime}}} +C.  
\end{equation*}
We now choose $2^* \delta=(2 \delta-1) q^{\prime}$, that is, $\delta=\frac{1}{2} \cdot \frac{q(n-2 )}{(n-2 q )}$. Note that $\delta >\frac{\gamma^*+1}{2}$ if and only if $q>\frac{n(\gamma^*+1)}{n+2\gamma^*}$.  Moreover, since $q<\frac{n}{2 }$ it follows that $\frac{2^*}{2 q^{\prime}}<1$. Thus we get by Young's inequality with $\ep$,
\begin{equation}{\label{res44}}
   \int_0^T\left(\int_{\Omega} v_k^{2^*\delta} d x\right)^{\frac{2}{2^*}} d t=\int_0^T\left(\int_{\Omega} v_k^{\left(2 \delta-1\right) q^{\prime}} d x\right)^{\frac{2}{2^*}} d t \leq  c,
\end{equation}
and therefore from \cref{res11}, \cref{res222.5} and \cref{res44}, it follows
\begin{equation*}
\left\|v_k\right\|_{L^\infty\left(0, T ; L^{2\sigma}(\Omega)\right)} +\left\|v_k\right\|_{L^{2\sigma}\left(0, T ; L^{2^*\sigma}(\Omega)\right)} \leq  c, \text{ with } \sigma=\frac{q(n-2 )}{2(n-2 q )}>\frac{q}{2}.
\end{equation*}
\textbf{Case 2: $\frac{nr}{n+2(r-1)} \leq q<\frac{n}{2 } r^{\prime}$ i.e. $\frac{1}{r}\geq\frac{n}{n-2} \frac{1}{q}-\frac{2}{n-2}$ and $\frac{1}{r}+\frac{n}{2 q }>1$}\smallskip
\\In this case we choose $\delta=\frac{1}{2} \frac{q r n}{nr-2 q (r-1)}$. Note that $\delta >\frac{1+\gamma^*}{2},$ if and only if $q>\frac{nr(\gamma^*+1)}{nr+2(1+\gamma^*)(r-1)}$ which is satisfied if $r>{\gamma^*+1}$. Now by interpolation, we get that $\frac{1}{(2 \delta-1) q^{\prime}}=\frac{1-\theta}{2 \delta}+\frac{\theta}{2^* \delta}$, where $\theta=n\delta\left[\frac{1}{2\delta}-\frac{1}{(2\delta-1)q^\prime}\right]\in(0,1]
$. Therefore
\begin{equation*}
    A^{r^\prime}=\int_0^T\left\|v_k\right\|_{L^{(2 \delta-1) q^{\prime}}(\Omega)}^{r^{\prime}(2 \delta-1)} d t \leq c\left\|v_k\right\|_{L^{\infty}\left(0, T ; L^{2 \delta}(\Omega)\right)}^{2 \delta \mu_1} \int_0^T\left(\int_{\Omega} v_k^{2^ * \delta} d x\right)^{\frac{2}{2^*} \mu_2} d t,
\end{equation*}
where $\mu_1=\frac{(1-\theta) r^{\prime}(2 \delta-1)}{2 \delta}$ and $\mu_2=\frac{\theta r^{\prime}(2 \delta-1)}{2\delta}$. Since $\mu_2\leq 1$, we have that
\begin{equation*}
      A^{r^\prime}= \int_0^T\left\|v_k\right\|_{L^{(2 \delta-1) q^{\prime}}(\Omega)}^{r^{\prime}(2 \delta-1)} d t \leq c\left\|v_k\right\|_{L^{\infty}\left(0, T ; L^{2 \delta}(\Omega)\right)}^{2 \delta \mu_1} \left(\int_0^T\left[\int_{\Omega} v_k^{2^* \delta} d x\right]^{\frac{2}{2^*}} d t\right)^{\mu_2}\stackrel{\cref{res11},\cref{res22}}{\leq }C(A+B)^{\mu_1+\mu_2},
\end{equation*}
and since $\mu_1+\mu_2 < r^{\prime}$, using Young's inequality, we conclude the following, which will give our final result
\begin{equation*}
    \int_0^T\left(\int_{\Omega} v_k^{(2 \delta-1) q^{\prime}} d x\right)^{\frac{r^{\prime}}{q^{\prime}}} d t \leq c.
\end{equation*}
\end{proof}
We now give boundedness result for $\{w_k\}_k$ and prove the following lemma.
\begin{lemma}{\label{bounded4}}
(1) Assume that $\mu \in L^r(0, T ; L^q(\Omega))$ with $r, q$ satisfying
\begin{equation*}
\frac{1}{r}+\frac{n}{2 q }<1  .
\end{equation*}
 Then there exists a positive constant $c>0$ such that the unique solution of \cref{aproxproblembdd3} satisfies
\begin{equation*}
    \left\|w_k(x, t)\right\|_{L^{\infty}\left(\Omega_T\right)} \leq c
    .
\end{equation*}
(2) Assume, and $\mu \in L^r\left(0, T ; L^q(\Omega)\right)$, with $r>1, q>1$ satisfy
\begin{equation*}
    1<\frac{1}{r}+\frac{n}{2 q }.
\end{equation*}
Further, assume that,\\
$i)$ if $\frac{1}{r}<\frac{n}{n-2} \frac{1}{q}-\frac{2}{n-2}$,  then $\frac{1}{q}\leq \frac{1}{2}+\frac{1}{n}$, and\\
$ii)$ if $\frac{1}{r}\geq\frac{n}{n-2} \frac{1}{q}-\frac{2}{n-2}$, then $1<\frac{1}{r}+\frac{n}{2 q }\leq 1+\frac{n}{4}$.\\
Then there exists a positive constant $c$ such that the sequence of solutions of \cref{aproxproblembdd3} satisfies
\begin{equation*}
\left\|w_k\right\|_{L^{\infty}\left(0, T ; L^{2 \sigma}(\Omega)\right)}+\left\|w_k\right\|_{L^{2 \sigma}\left(0, T ; L^{2^* \sigma}\right)} \leq c,
\end{equation*}
where
\begin{equation*}
\sigma=\left\{\begin{array}{lll}
\frac{q(n-2 )}{2(n-2 q )} & \text { if } \frac{1}{r}<\frac{n}{n-2} \frac{1}{q}-\frac{2}{n-2}, \smallskip\\
\frac{q r n
}{2(n r+2q-2 q r)} & \text { if } \frac{1}{r} \geq \frac{n}{n-2} \frac{1}{q}-\frac{2}{n-2}.
\end{array}\right.    
\end{equation*}
\end{lemma}
\begin{proof}
    (1) This part follows exactly similar to \cref{bounded2}, part (1).
    \smallskip\\(2)  We choose $\delta\geq 1$ and take $w_k^{2\delta-1}\chi_{(0,\tau)},0<\tau\leq T$ as test function in \cref{aproxproblembdd3}. We note that, such a test function is admissible as $w_k\in L^\infty(\Omega_T)$.
We get $\forall \tau\leq T,$
\begin{equation}{\label{zzyyzz}}
 \begin{array}{c}
 \quad\frac{1}{2\delta} \int_{\Omega} w_k^{2\delta}(x, \tau) d x+\int_0^\tau\int_\Omega \nabla w_k\cdot\nabla w_k^{2\delta-1}d x dt\smallskip\\\quad+ \underbrace{\int_0^\tau \int_{\mathbb{R}^n}\int_{\mathbb{R}^n} \frac{\left(w_k(x, t)-w_k(y,t)\right)\left(w_k^{2\delta-1}(x, t)-w_k^{2\delta-1}(y, t)\right)}{|x-y|^{n+2 s}} d x d y d t}_{\geq 0}
\smallskip\\\leq \int_0^\tau\int_{\Omega} T_k(\mu)w_k^{2\delta-1} \leq  \int_0^\tau\int_{\Omega} \mu w_k^{2\delta-1}dxdt.
\end{array}
\end{equation}Note that by \cref{lemma1}-(3), 
    for each $t_0>0$ and $\omega\subset\subset\Omega$, $\exists C(\omega,t_0,n,s)$, a constant which does not depend on $k$,  such that $\forall k$, $w_k\geq C(\omega,t_0,n,s)$ in $\omega\times[t_0,T)$. Taking supremum over $\tau\in(0,T]$ in \cref{zzyyzz}, by Hölder and Sobolev inequalities, we have
\begin{equation}{\label{testtt111}}
\begin{array}{rcl}
\sup _{\tau \in[0, T]} \int_{\Omega} w_k^{2 \delta} (x,\tau)d x+\frac{2(2\delta-1)}{ \delta} \lambda \int_0^T\left\|w_k^\delta\right\|_{L^{2^*}(\Omega)}^2 d t 
&\leq& 2\delta\left\|\mu\right\|_{L^r\left(0, T ; L^q(\Omega)\right)}\left(\int_0^T\left[\int_{\Omega} w_k^{(2 \delta-1) q^{\prime}} d x\right]^{\frac{r^{\prime}}{q^{\prime}}} d t\right)^{\frac{1}{r^{\prime}}}.
\end{array}
\end{equation}
Note that $\frac{2\delta-1}{\delta}>1$, so we can ignore the constants in left. Denoting by
$
A=\left(\int_0^T\left\|w_k\right\|_{L^{(2 \delta-1) q^{\prime}}(\Omega)}^{(2 \delta-1) r^{\prime}} d t\right)^{\frac{1}{r^{\prime}}},    
$
we get from \cref{testtt111},
\begin{equation}{\label{rres11}}
\sup _{t\in[0, T]} \int_{\Omega} w_k^{2 \delta} d x \leq 2 \delta\left\|\mu\right\|_{L^r\left(0, T ; L^q(\Omega)\right)} A,    
\end{equation}
and
\begin{equation}{\label{rres22}}
    \int_0^T\left[\int_{\Omega} w_k^{2^*\delta} d x\right]^{\frac{2}{2 ^*}} d t  \leq\frac{c}{\lambda} 2\delta\left\|\mu\right\|_{L^r\left(0, T ; L^q(\Omega)\right)} A.
\end{equation}
\smallskip\\
\textbf{Case 1: $1<q<\frac{n r}{n+2 (r-1)}$ \text{ i.e. } $\frac{1}{r}<\frac{n}{n-2} \frac{1}{q}-\frac{2}{n-2}$} \smallskip\\Since, $1<q<\frac{n r}{n+2 (r-1)}$ implies $q<r$ and $\frac{r^{\prime}}{q^{\prime}}<\frac{2}{2^*}$, we apply Hölder inequality with exponent $\frac{2q^{\prime}}{2^*r^{\prime}}$ to get
\begin{equation}{\label{rres2.5}}
A \leq T^{\frac{1}{r^{\prime}}-\frac{2^*}{2 q^{\prime}}}\left[\int_0^T\left(\int_{\Omega} w_k^{(2 \delta-1) q^{\prime}} d x\right)^{\frac{2}{2 ^*}} d t\right]^{\frac{2^*}{2 q^{\prime}}} .    \end{equation}
Thus from \cref{rres22}, it follows
\begin{equation*}{\label{rres3}}
\int_0^T\left[\int_{\Omega} w_k^{2^* \delta} d x\right]^{\frac{2}{2^*}} d t \leq 2\delta cA \leq 2\delta c\left[\int_0^T\left(\int_{\Omega} w_k^{(2 \delta-1) q^{\prime}} d x\right)^{\frac{2}{2 ^*}} d t\right]^{\frac{2^*}{2 q^{\prime}}} .  \end{equation*}
We now choose $2^* \delta=(2 \delta-1) q^{\prime}$, that is, $\delta=\frac{1}{2} \cdot \frac{q(n-2 )}{(n-2 q )}$. Note that $\delta \geq 1$ if and only if $\frac{1}{q}\leq \frac{1}{2}+\frac{1}{n}$.  Moreover, since $q<\frac{n}{2 }$ it follows that $\frac{2^*}{2 q^{\prime}}<1$. Thus we get by Young's inequality with $\ep$,
\begin{equation}{\label{rres4}}
   \int_0^T\left(\int_{\Omega} w_k^{2^*\delta} d x\right)^{\frac{2}{2^*}} d t=\int_0^T\left(\int_{\Omega} w_k^{\left(2 \delta-1\right) q^{\prime}} d x\right)^{\frac{2}{2^*}} d t \leq  c,
\end{equation}
and therefore from \cref{rres11}, \cref{rres2.5} and \cref{rres4}, it follows
\begin{equation*}
    \left\|w_k\right\|_{L^\infty\left(0, T ; L^{2\sigma}(\Omega)\right)} +\left\|w_k\right\|_{L^{2\sigma}\left(0, T ; L^{2^*\sigma}(\Omega)\right)} \leq  c, \text{ with } \sigma=\frac{q(n-2 )}{2(n-2 q )}>\frac{q}{2}.
\end{equation*}
\smallskip\\\textbf{Case 2: $\frac{nr}{n+2(r-1)} \leq q<\frac{n}{2 } r^{\prime}$ i.e. $\frac{1}{r}\geq\frac{n}{n-2} \frac{1}{q}-\frac{2}{n-2}$ and $\frac{1}{r}+\frac{n}{2 q }>1$}\smallskip
\\In this case we choose $\delta=\frac{1}{2} \frac{q r n}{nr-2 q (r-1)}$. Note that $\delta \geq 1,$ if and only if $\frac{1}{r}+\frac{n}{2 q }\leq 1+\frac{n}{4}$. Now by interpolation, we get that $\frac{1}{(2 \delta-1) q^{\prime}}=\frac{1-\theta}{2 \delta}+\frac{\theta}{2^* \delta}$, where $\theta=n\delta\left[\frac{1}{2\delta}-\frac{1}{(2\delta-1)q^\prime}\right]\in(0,1]
$. Therefore

\begin{equation*}
    A^{r^\prime}=\int_0^T\left\|w_k\right\|_{L^{(2 \delta-1) q^{\prime}}(\Omega)}^{r^{\prime}(2 \delta-1)} d t \leq c\left\|w_k\right\|_{L^{\infty}\left(0, T ; L^{2 \delta}(\Omega)\right)}^{2 \delta \mu_1} \int_0^T\left(\int_{\Omega} w_k^{2^ * \delta} d x\right)^{\frac{2}{2^*} \mu_2} d t,
\end{equation*}
where $\mu_1=\frac{(1-\theta) r^{\prime}(2 \delta-1)}{2 \delta}$ and $\mu_2=\frac{\theta r^{\prime}(2 \delta-1)}{2\delta}$. Since $\mu_2\leq 1$, we have that
\begin{equation*}
      A^{r^\prime}= \int_0^T\left\|w_k\right\|_{L^{(2 \delta-1) q^{\prime}}(\Omega)}^{r^{\prime}(2 \delta-1)} d t \leq c\left\|w_k\right\|_{L^{\infty}\left(0, T ; L^{2 \delta}(\Omega)\right)}^{2 \delta \mu_1} \left(\int_0^T\left[\int_{\Omega} w_k^{2^* \delta} d x\right]^{\frac{2}{2^*}} d t\right)^{\mu_2}\stackrel{\cref{rres11},\cref{rres22}}{\leq }C(A+B)^{\mu_1+\mu_2},
\end{equation*}
and since $\mu_1+\mu_2 < r^{\prime}$, using Young's inequality, we conclude the following, which will give our final result
\begin{equation*}
    \int_0^T\left(\int_{\Omega} w_k^{(2 \delta-1) q^{\prime}} d x\right)^{\frac{r^{\prime}}{q^{\prime}}} d t \leq c.
\end{equation*}
\end{proof}
\section{Proof of regularity results}{\label{sec6}}
\subsection*{Proof of \cref{regu1}, \cref{regu2} and \cref{regu3}} As the solution obtained in \cref{exis1} is the pointwise limit of the sequence $\{u_k\}_k$ in $\Omega_T$, therefore noting \cref{bounded1}, \cref{bounded2}, \cref{bounded3} ,\cref{bounded4} and using Fatou's lemma, we obtain each conclusion of \cref{regu1}, \cref{regu2} and \cref{regu3}.
\section{Proof of \cref{uniqueness}}{\label{sec7}}
We first show if $u$ is a finite energy solution to \cref{mainproblem}, with $0\leq \nu\in L^1(\Omega_T)$, $\mu\in\mathcal{M}(\Omega_T)$ of the  form $f-\operatorname{div}G$ for some $f\in L^1(\Omega_T)$ and $G\in \left(L^2(\Omega_T)\right)^n$, then $u$ satisfies
\begin{equation}{\label{finene}}
\begin{array}{c}
  \int_0^T \left\langle\left(u_t\right)_1, \phi\right\rangle +\int_0^T\int_\Omega\left(u_t\right)_2 \phi+\iint_{\Omega_T} \nabla u \cdot\nabla \phi+\int_0^T\int_{\mathbb{R}^{n}}\int_{\mathbb{R}^{n}} 
\frac{(u(x,t)-u(y,t))(\phi(x,t)-\phi(y,t))}{|x-y|^{n+2s}}\smallskip\\=\iint_{\Omega_T} \frac{\phi\nu}{u^{\gamma(x,t)}} dxdt+\iint_{\Omega_T}f\phi\,dxdt+\iint_{\Omega_T}G\cdot\nabla\phi\,dxdt,\end{array}
\end{equation}
 for every $\phi \in L^2(0, T ; W_0^{1, 2}(\Omega)) \cap L^{\infty}(\Omega_T)$, where $\left(u_t\right)=\left(u_t\right)_1+\left(u_t\right)_2$ with $\left(u_t\right)_1 \in L^{2}\left(0, T ; W^{-1, 2}(\Omega)\right),\left(u_t\right)_2 \in L^1(\Omega_T)$. Let $\phi \in L^2(0, T ; W_0^{1, 2}(\Omega)) \cap L^{\infty}(\Omega_T)$ be a non-negative function and consider a sequence $\phi_k$ of smooth non-negative functions converging to $\phi$ in $L^2(0, T ; W_0^{1, 2}(\Omega))$, (such a sequence exists by density, see for details [\citealp{softmea}, Appendix A]) without loss of generality one may assume $\|\phi_k\|_{L^{\infty}(\Omega_T)} \leq\|\phi\|_{L^{\infty}(\Omega_T)}$. Now we take $\phi_k$ in \cref{mainproblem} and integrate by parts (see for instance \cref{first}) in order to obtain
\begin{equation*}{\label{finene2}}
\begin{array}{c}
  \int_0^T \left\langle\left(u_t\right)_1, \phi_k\right\rangle +\int_0^T\int_\Omega\left(u_t\right)_2 \phi_k+\iint_{\Omega_T} \nabla u \cdot\nabla \phi_k+\int_0^T\int_{\mathbb{R}^{n}}\int_{\mathbb{R}^{n}} 
\frac{(u(x,t)-u(y,t))(\phi_k(x,t)-\phi_k(y,t))}{|x-y|^{n+2s}}\smallskip\\=\iint_{\Omega_T} \frac{\phi_k\nu}{u^{\gamma(x,t)}} dxdt+\iint_{\Omega_T}f\phi_kdxdt+\iint_{\Omega_T}G\cdot\nabla\phi_kdxdt,\end{array}
\end{equation*}
Due to the definition of $\phi_k$ and on the regularity of $u$ one can pass to the limit in the left hand side of the previous equality along with the last two terms in right; in particular $\frac{\phi_k\nu}{u^{\gamma(x,t)}} $ is bounded in $L^1(\Omega_T)$ and by Fatou's lemma one obtains that also $\frac{\phi\nu}{u^{\gamma(x,t)}} \in L^1(\Omega_T)$. Using dominated convergence theorem one can pass to the limit in this term too and we conclude \cref{finene} (note that for a function in $L^2(0,T;W^{1,2}_0(\Omega))$ with arbitrary sign, one can consider its positive and negative parts separately).

Let $v, w$ be two finite energy solutions of \cref{mainproblem}, and take $T_k(v-w) \phi(t)$ in the difference of formulations \cref{finene} solved by $v, w$, where we define $\phi(t)=\frac{-t}{T}+1$ for $t \in(0, T]$. Dropping positive terms, we obtain that
\begin{equation*}
\int_0^T\left\langle\left(v_t\right)_1-\left(w_t\right)_1, T_k(v-w) \phi(t)\right\rangle+\iint_{\Omega_T}\left(\left(v_t\right)_2-\left(w_t\right)_2\right) T_k(v-w) \phi(t) \leq 0 .
\end{equation*}
It follows from of [\citealp{jerome}, Lemma 7.1] that ($\tilde{T}_{k, 1}(s)$ is defined in \cref{T_k})
\begin{equation*}
\frac{1}{T} \iint_{\Omega_T} \tilde{T}_{k, 1}(v-w)=0
\end{equation*}that, due to the arbitrariness of $k$ and recalling that $v-w \in C\left([0, T] ; L^1(\Omega)\right)$, implies that $v(\tau)=w(\tau)$ for any $\tau \in(0, T]$ and for almost every $x$ in $\Omega$ and this concludes the proof.
\section{Proof of \cref{asymptotic}}{\label{sec8}}
In order to prove the asymptotic behavior, we first consider the difference of the problems \cref{asym1,asym2} solved respectively by $u_k$ and $v_k$ and take $G_l(u_k-v_k)$, $l>0$ as test function. We have, for every $0<t_1<t_2<T$,
\begin{equation}{\label{asym3}}
    \begin{array}{c}
    \frac{1}{2}  \int_{\Omega}\left|G_l\left(u_k-v_k\right)(t_2)\right|^2-\frac{1}{2} \int_{\Omega}\left|G_l\left(u_k-v_k\right)(t_1)\right|^2 
 +\int_{t_1}^{t_2} \int_{\Omega}\left|\nabla G_l\left(u_k-v_k\right)\right|^2dxdt\\+\int_{t_1}^{t_2}\int_{\mathbb{R}^{n}}\int_{\mathbb{R}^{n}} 
\frac{((u_k-v_k)(x,t)-(u_k-v_k)(y,t))(G_l(u_k-v_k)(x,t)-G_l(u_k-v_k)(y,t))}{|x-y|^{n+2s}}\\
= \int_{t_1}^{t_2} \int_{\Omega} T_k(f)\left(\frac{1}{\left(u_k+\frac{1}{k}\right)^{\gamma(x,t)}}-\frac{1}{\left(v_k+\frac{1}{k}\right)^{\gamma(x,t)}}\right) G_l(u_k-v_k).
\end{array}
\end{equation}
Note that 
\begin{equation*}
    \begin{array}{l}
\quad\int_{t_1}^{t_2}\int_{\mathbb{R}^{n}}\int_{\mathbb{R}^{n}} 
\frac{((u_k-v_k)(x,t)-(u_k-v_k)(y,t))(G_l(u_k-v_k)(x,t)-G_l(u_k-v_k)(y,t))}{|x-y|^{n+2s}}\\=\int_{t_1}^{t_2}\int_{\mathbb{R}^{n}}\int_{\mathbb{R}^{n}} 
\frac{(G_l(u_k-v_k)(x,t)-G_l(u_k-v_k)(y,t))^2}{|x-y|^{n+2s}}\\\quad+\int_{t_1}^{t_2}\int_{\mathbb{R}^{n}}\int_{\mathbb{R}^{n}} 
\frac{(T_l(u_k-v_k)(x,t)-T_l(u_k-v_k)(y,t))(G_l(u_k-v_k)(x,t)-G_l(u_k-v_k)(y,t))}{|x-y|^{n+2s}},
    \end{array}
\end{equation*} which is non-negative by [\citealp{Peral}, Lemma 4]. Observing now that the function $G_l(u_k-v_k)$ has the same sign of the function $(u_k-v_k)$, the right hand side of \cref{asym3} is negative and consequently we have
\begin{equation*}{\label{asy4}}
     \frac{1}{2}  \int_{\Omega}\left|G_l\left(u_k-v_k\right)(t_2)\right|^2-\frac{1}{2} \int_{\Omega}\left|G_l\left(u_k-v_k\right)(t_1)\right|^2 
 +\int_{t_1}^{t_2} \int_{\Omega}\left|\nabla G_l\left(u_k-v_k\right)\right|^2dxdt\leq 0.
\end{equation*}
Thus, by Sobolev inequality we obtain
\begin{equation}{\label{asy5}}
     \int_{\Omega}\left|G_l\left(u_k-v_k\right)(t_2)\right|^2-\frac{1}{2} \int_{\Omega}\left|G_l\left(u_k-v_k\right)(t_1)\right|^2 
 +2c\int_{t_1}^{t_2} \left(\int_{\Omega}\left|G_l\left(u_k-v_k\right)\right|^{2^*}dx\right)^{2/2^*}dt\leq 0.
\end{equation}Here we have denoted by $c=c(n)$ the Sobolev constant. 
Now take as test function in \cref{asym1,asym2} the function $\phi=\left\{1-\frac{1}{\left[1+\left|G_l\left(u_k-v_k\right)\right|\right]^\delta}\right\} \operatorname{sign}(u_k-v_k), \delta>1$. Subtracting the results we get for every $0 \leq t_0 \leq t<T$
\begin{equation}{\label{asy6}}
\begin{array}{l}
\int_{\Omega}|G_l(u_k-v_k)|(t)+\frac{1}{\delta-1} \int_{\Omega}\left\{1-\frac{1}{\left[1+\left|G_l(u_k-v_k)\right|(t)\right]^{\delta-1}}\right\} 
\smallskip\\+\delta \int_{t_0}^t \int_{\Omega}(\nabla u_k-\nabla v_k) \frac{\nabla G_l(u_k-v_k)}{\left[1+\left|G_l(u_k-v_k)\right|\right]^{\delta+1}}\smallskip \\+\int_{t_1}^{t_2}\int_{\mathbb{R}^{n}}\int_{\mathbb{R}^{n}} 
\frac{((u_k-v_k)(x,t)-(u_k-v_k)(y,t))(\phi(x,t)-\phi(y,t))}{|x-y|^{n+2s}}\smallskip\\
 \leq \int_{\Omega}\left|G_l\left(u_k-v_k\right)\right|(t_0)+\frac{1}{\delta-1} \int_{\Omega}\left\{1-\frac{1}{\left[1+\left|G_l\left(u_k-v_k\right)\right|(t_0)\right]^{\delta-1}}\right\} \smallskip\\
 +\int_{t_1}^{t_2} \int_{\Omega} T_k(f)\left(\frac{1}{\left(u_k+\frac{1}{k}\right)^\gamma}-\frac{1}{\left(v_k+\frac{1}{k}\right)^\gamma}\right)\left\{1-\frac{1}{\left[1+\mid G_l\left(u_k-v_k\right)\right]^\delta}\right\} \operatorname{sign}\left(u_k-v_k\right).\end{array}
\end{equation}
Note that the function $\phi(r)=\left\{1-\frac{1}{\left[1+\left|G_l(r)\right|\right]^\delta}\right\} \operatorname{sign}(r)$ is increasing in $r$, and hence the nonlocal integral is non-negative. Observe also that the last integral is negative (or null). Thus, by \cref{asy6} we deduce that
\begin{equation*}
\int_{\Omega}\left|G_l(u_k-v_k)\right|(t) \leq \int_{\Omega}\left|G_l(u_k-v_k)\right|\left(t_0\right)+\frac{|\Omega|}{\delta-1},    
\end{equation*}
from which, thanks to the arbitrary choice of $\delta>1$ we derive that for every $0 \leq t_0 \leq t<T$,
\begin{equation*}
\int_{\Omega}\left|G_l(u_k-v_k)\right|(t) \leq \int_{\Omega}\left|G_l(u_k-v_k)\right|\left(t_0\right). 
\end{equation*}
Noting \cref{asy5}, we can now apply [\citealp{Porziomm}, Theorem 2.2] obtaining the following estimate
\begin{equation}{\label{asy7}}
\left\|u_k(t)-v_k(t)\right\|_{L^{1}(\omega)}\leq c\left\|u_k(t)-v_k(t)\right\|_{L^{\infty}(\Omega)} \leq M_1 \frac{\left\|T_k(u_0)-T_k(v_0)\right\|_{L^1(\Omega)}}{t^{\frac{n}{2}} e^{\sigma t}},\text{ for a.e. } t\in(0,T), 
\end{equation} for each $\omega\subset\subset\Omega$, for some positive constants $\sigma, h, M_1$.

Note that \cref{convergenceinL1} can be proved by taking some different approximation of $u_0$ than $T_k(u_0)$. Proceeding as in the proof of \cref{asy7}, taking $\bar {u}_k$ (where $\bar {u}_k$ is a solution of \cref{asym1}, with $u_k(\cdot,0)= \bar {u}_{0k}$ and $\{\bar {u}_{0k}\}$ is a sequence of non-negative $L^\infty(\Omega)$ functions converging to $u_0$ in $L^1(\Omega)$) in place of $v_k$, we deduce the following estimate
\begin{equation*}{\label{asy8}}
\left\|u_k(t)-\bar{u}_k(t)\right\|_{L^{1}(\omega)}\leq c\left\|u_k(t)-\bar{u}_k(t)\right\|_{L^{\infty}(\Omega)} \leq M_1 \frac{\left\|T_k(u_0)-\bar{u}_{0k}\right\|_{L^1(\Omega)}}{t^{\frac{n}{2}} e^{\sigma t}},\text{ for each }\omega\subset\subset\Omega,
\end{equation*} for some positive constants $\sigma, h, M_1$.
Hence, passing to the limit on $k$ and recalling that by \cref{convergenceinL1} $u$ and $\bar{u}$ are, the a.e. limit in $\Omega_T$ of (respectively) $u_k$ and $\bar{u}_k$, we deduce that $u=\bar u$ a.e. in $\Omega_T$. Note that, we have also used the convergence of $\{u_k\}$ and $\{\bar{u}_k\}$ in $L^1_{\mathrm{loc}}(\Omega_T)$, as per \cref{convergenceinL1}. Therefore the approximation choice of initial data does not influence the weak solution. In fact, approximation (non-negative) choices for $f$ or $g$ (not simultaneously together) also do not influence the weak solution.

Finally, passing to the limit on $k$ in \cref{asy7} and recalling by \cref{convergenceinL1} that $u$ and $v$ are, the a.e. limit in $\Omega_T$ of (respectively) $u_k$ and $v_k$, we deduce that 
\begin{equation*}{\label{asy9}}
\left\|u(t)-v(t)\right\|_{L^{1}(\omega)} \leq M_1 \frac{\left\|u_0-v_0\right\|_{L^1(\Omega)}}{t^{\frac{n}{2}} e^{\sigma t}},\text{ for almost every } t\in(0,T).
\end{equation*}Again we have used the convergence of $\{u_k\}$ and $\{v_k\}$ in $L^1_{\mathrm{loc}}(\Omega_T)$, as per \cref{convergenceinL1}.

Let $u_{0, k} =T_k(u_0)\in L^{\infty}(\Omega)$ be such that $\{u_{0, k} \}\to u_0$ in $L^1(\Omega)$ and let $T_0>0$ arbitrarily fixed. By \cref{exis1} there exists a weak solution $u$ of \cref{asy} in $\Omega_{T_0}$ obtained as the a.e. limit in $\Omega_{T_0}$ of $u_k \in L^2(0, T_0 ; W_0^{1, 2}(\Omega)) \cap L^{\infty}(\Omega_{T_0}) \cap C([0, T_0] ; L^2(\Omega))$ (solutions to \cref{asym1}). Also the approximation of initial data does not affect the solution.
Therefore, since the structure conditions are satisfied for every $T>0$, the approximating solutions $u_k$ are global solutions of \cref{asy}. In particular, if $u_k\in L^2(0,2 T_0 ; W_0^{1, 2}(\Omega)) \cap L^{\infty}(\Omega_{2 T_0}) \cap C([0,2 T_0] ; L^2(\Omega))$ are solutions of our problem in $\Omega_{2 T_0}$, it is possible to extract a subsequence of $u_k$, that we denote $u_k^{(1)}$, converging a.e. in $\Omega_{2 T_0}$ to a weak solution $u^{(1)}$ of our problem in $\Omega_{2 T_0}$. We notice that $u=u^{(1)}$ in $\Omega_{T_0}$. Iterating this procedure we get a global solution $u$ which, by construction, satisfies
\begin{equation}{\label{asy10}}
\left\|u(t)-v(t)\right\|_{L^{1}(\omega)} \leq M_1 \frac{\left\|u_0-v_0\right\|_{L^1(\Omega)}}{t^{\frac{n}{2}} e^{\sigma t}},\text{ for almost every } t\in(0,\infty).
\end{equation}
Taking limit $t\to \infty$ in \cref{asy10}, we deduce the conclusion.
\section{Further directions, open problems}
The results of this paper pose several questions and problems. Without pretending to
be exhaustive concerning the directions one might take, we give a short list of possible
issues here.\smallskip\\
\textbf{$1$. The quasilinear case:} The first possible generalization one can think is to consider a quasilinear analogue of \cref{mainproblem}. We believe that at least for $p\geq 2$, the quasilinear version will show similar results.\smallskip\\\textbf{$2$. The case $g\neq 0$ in \cref{representation}:} Our existence results \cref{exis1} and \cref{exis2} consider diffuse measures which belongs to $L^1(\Omega_T)+L^{q^\prime}(0,T;W^{-1,q^\prime}(\Omega))$ i.e. $g=0$ in \cref{representation}. Is it possible to consider $g\neq 0$ i.e. to take the full class of diffuse measures and still have the existence results? In that case, the solutions seem to be defined in some other sense. Which sense will be perfect to define the solutions?\smallskip\\\textbf{$3$. Singular measures:} In \cref{exis1} and \cref{exis2}, it have been assumed that $\nu$ is non-singular with respect to the Lebesgue measure $\mathcal{L}$. Looking at the ellictic analogue \cite{biswas2025existence}, it is reasonable conjecture that the same existece results will still hold if $\nu$ is singular with respect to $\mathcal{L}$.\smallskip\\\textbf{$4$. The measure is of the type $\nu=\nu_d+\nu_c$:} If $\mu$ is in $\mathcal{M}(\Omega_T)$, thanks to a well known decomposition result (see for instance \cite{fukus}), we can split it into a sum (uniquely determined) of its absolutely continuous part $\mu_d$ with respect to $p$-capacity and its singular part $\mu_c$, that is $\mu_c$ is concentrated on a set $E$ of zero $p$-capacity; we will say that $\mu_c \perp$ cap$_p$. Noting the elliptic analogue \cite{lindda}, the possible questions that occur are: What happens when one simultaneously takes $\nu_d,\nu_c\neq 0$? Recalling \cref{nonexistence}, will the solution have only effect from the term $\nu_d$? What will be the suitable notion of defining solutions for the case $\nu_d,\nu_c\neq 0$? \section*{Acknowledgements} The author wishes to express sincere gratitude to Prof. Francesco Petitta for his valuable discussions. The author gratefully acknowledges the financial support provided by Indian Institute of Technology Kanpur.
\bibliography{main.bib}

\bibliographystyle{plain}


\end{document}